\documentclass[12pt,thmsb]{article}
                
\usepackage[section]{placeins}
\usepackage{amsmath,amsfonts, latexsym, amssymb, amsthm}
\usepackage{natbib, url, setspace, lscape, graphicx, titlesec}   
\usepackage[top=1in, bottom=1in, left=1in, right=1in]{geometry}
\usepackage{bm, bbm, xcolor, pdflscape} 
\usepackage[multiple]{footmisc}  %% footnote 
\usepackage{graphicx, caption, subcaption}

\DeclareMathVersion{sans}  
\SetSymbolFont{letters}{sans}{OML}{cmssm}{m}{it}     
\SetSymbolFont{symbols}{sans}{OMS}{cmsy}{m}{n}
\SetMathAlphabet{\mathit}{sans}{OT1}{cmss}{m}{sl} 
\SetMathAlphabet{\mathbf}{sans}{OT1}{cmss}{bx}{n}
\SetMathAlphabet{\mathtt}{sans}{OT1}{cmtl}{m}{n}

\renewcommand{\arraystretch}{.5} 
\advance \topmargin by -\headheight
\advance \topmargin by -\headsep
\evensidemargin \oddsidemargin
\let\tilde=\widetilde
\newtheorem{theorem}{Theorem}
\newtheorem{proposition}{Proposition}
\newtheorem{lemma}{Lemma}

\newtheorem{corollary}{Corollary}

\newtheorem{property}{Property}
\renewcommand{\arraystretch}{1.5}
\DeclareMathOperator*{\argmax}{\arg\,\max} 
\begin{document}

%%%%%%%%%%%%%%%%%%%%%%%%%%%%%%%%%%%%%%%%%%%%%%%%%%%%%%%%%%%%%%%%%%%%%%

\title{ 
  \fontsize{15}{15}\selectfont
  \textbf{Testing for Common Breaks in a Multiple Equations System}\thanks{
    We thank the Editor, Oliver Linton, an Associate Editor and 
    three anonymous referees for their constructive comments,
    which improved the paper. 
    We would like to thank 
    Jushan Bai,  
    Alastair Hall, 
    Eiji Kurozumi,
    James Morley, 
    Zhongjun Qu, 
    Mototsugu Shintani,
    Denis Tkachenko,
    seminar participants at Boston University and 
    participants at
    the 2009 Far East and South Asia Meeting of the Econometric Society
    for useful comments. 
    We are also grateful to Douglas
    Sondak for advices on the computations.
    Oka gratefully acknowledges
    the financial support from Singapore Ministry of Education
    Academic Research Fund Tier 1 (FY2015-FRC3-003)
    and
    also
    gratefully acknowledges the financial support from
    Monash Business School.
  }
}

\author{
  \textbf{Tatsushi Oka}\thanks{
    Department of Econometrics and Business Statistics, Monash University    
(\texttt{tatsushi.oka@monash.edu}).} \\
%EndAName
 \and \textbf{Pierre Perron}\thanks{%
Department of Economics, Boston University, USA 
 (\texttt{perron@bu.edu}).} \\
%EndAName
\vspace{0.05cm}
}

\date{
  \today
}
\maketitle

\begin{abstract}
The issue addressed in this paper is that of testing for common breaks
across or within equations of a multivariate system. Our framework is very general and allows
integrated regressors and trends as well as stationary regressors. The null
hypothesis is that 
breaks in different parameters 
occur at common locations and are separated by some positive fraction of the sample size
unless they occur across different equations.
Under the alternative hypothesis, the break dates across parameters are not the same 
and also need not be separated by a positive fraction of the sample size
whether within or across equations. The test considered is the quasi-likelihood ratio test
assuming normal errors, though as usual the limit distribution of the test
remains valid with non-normal errors. 
Of independent interest,
we provide results about the rate of
convergence of the estimates when searching over all possible partitions subject only to the
requirement that each regime contains at least as many observations as some positive fraction of the sample size,
allowing break dates not separated by a positive fraction of the sample size
across equations.  
Simulations show that the test has good finite sample
properties. 
We also provide an application to issues related to 
level shifts and persistence for 
various measures of inflation
to illustrate its usefulness. 
 
\vspace{1cm}

\vspace{0.5cm}

\end{abstract}

\thispagestyle{empty}\setcounter{page}{0}\baselineskip=18.0pt\newpage

\pagestyle{plain}
%%%%%%%%%%%%%%%%%%%%%%%%%%%%%%%%%%%%%%%%%%%%%%%%%%%%%%%%%%%%%%%%%%%%%%

%%%%%%%%%%%%%%%%%%%%%%%%%%%%%%%%%%%%%%%%%%%%%%%%%%%%%%%%%%%%%%%%%%%%%%
%%%%%%%%%%%%%%%%%%%%%%%%%%%%%%%%%%%%%%%%%%%%%%%%%%%%%%%%%%%%%%%%%%%%%%
\section{Introduction}

%% Literature Review 
Issues related to structural change have been extensively studied in the statistics and econometrics literature
\citep[see][for comprehensive reviews]{CH1997Book, Perron2006HB}.
In the last twenty years or so, substantial advances have been made
in the econometrics literature to cover models at a level of generality that makes them relevant across time-series applications in the context of unknown change points. 
For example,  
\cite{Bai1994JTSA, Bai1997REStat}
studies the least squares estimation of a
single change point in regressions involving stationary and/or trending
regressors. 
\cite{BaiPerron1998Emtca, BaiPerron2003JAE}
extend the testing and estimation analysis to the case of multiple structural changes
and 
present an efficient algorithm. 
\cite{Hansen1992JBES}
and
\cite{KejriwalPerron2008JoE}
consider regressions with integrated variables.    
\cite{Andrews1993Emtca}
and 
\cite{HallSen1999JBES}
consider nonlinear models estimated by 
generalized method of moments.  
\cite{Bai1995ET, Bai1998JSPI}
studies structural changes in least absolute deviation regressions,
while  
\cite{Qu2008JoE},
\cite{SuXiao2008SPL}
and
\cite{OkaQu2011JoE}
analyze structural changes in regression quantiles.
\cite*{HallHanBoldea2012JoE} 
and 
\cite{PerronYamamoto2014, PerronYamamoto2015JAE}
consider   
structural changes in linear models with endogenous regressors.
Studies about structural changes in panel data models 
include 
\cite{Bai2010JoE}, \cite{Kim2011JoE}, \cite*{BaltagiFengKao2015JoE} 
and \cite{QianSu2016JoE} 
for linear panel data models
and 
\cite{BreitungEickmeier2011JoE},
\cite*{ChenEtAl2016RES}, 
\cite{CorradiSwanson2014JoE},
\cite{HanInoue2015ET}
and
\cite{YamamotoTanaka2015JoE}
for factor models.

%% Aim and Issue 
The literature on structural breaks in a multiple equations system 
includes \cite{BaiLumsdaineStock1998}, \cite{Bai2000AEF} and \cite{QuPerron2007Emtca},
among others.  Their analysis relies on a common breaks assumption, under which 
breaks in different basic parameters (regression coefficients and elements of the covariance matrix of the errors) 
occur at a common location or are separated by some positive fraction of the sample size
(i.e., asymptotically distinct).\footnote{
  The concept of common breaks here is quite distinct from
  the notion of co-breaking or co-trending 
  \citep[e.g.,][]{HatanakaYamada2003Book, HendryMizon1998EE}. 
  In this literature, the focus is on whether
  some linear combination of series with breaks do not have a break, a concept
  akin to that of cointegration.
} 
\cite{BaiLumsdaineStock1998}
assume a single common break across equations   
for a multivariate system with stationary regressors and trends 
as well as for cointegrated systems.
For the case of multiple common breaks,  
\cite{Bai2000AEF} analyzes vector autoregressive models
for stationary variables 
and 
\cite{QuPerron2007Emtca}
cover multiple system equations, allowing for more general stationary regressors
and arbitrary restrictions across parameters. 
Under the framework of \cite{QuPerron2007Emtca},
\cite{KurozumiTu2011JoE} propose model selection procedures for 
a system of equations with multiple common breaks
and 
\cite{EoMorley2015QE} consider a confidence set for the common break date based on inverting the likelihood ratio test.
In this literature,
it has been documented that common breaks allow more precise estimates of the break dates in  multivariate systems.
Given unknown break dates, however, 
an issue of interest for most applications concerns the validity of the assumption of common breaks.\footnote{ 
  The common breaks assumption is also used in the literature on panel data
  \citep[e.g.][]{Bai2010JoE, Kim2011JoE, BaltagiFengKao2015JoE}.  
  In this paper, we consider a multiple equations system in which the number of 
  equations are relatively small, and thus panel data models are outside our scope. 
  However, testing for common breaks in a system with a large number of equations 
  is an interesting avenue for future research. 
}   
To our knowledge, no test has been proposed to address this issue.

%% LR test 
%% 1) general framework
Our paper addresses three outstanding issues about testing for common breaks.
First, we propose a quasi-likelihood ratio test under a very general framework.\footnote{
  %% LM test
  One may also consider other type of tests, such as LM-type tests. 
  The literature on structural breaks, however, documents that  
  even though LM-type tests have simple asymptotic representations, 
  they tend to exhibit poor finite sample properties with respect to power. 
  Thus, this paper focuses on the LR test 
  \citep[see][for instance]{DP2008JoE, KimPerron2009JoE, PerronYamamoto2016ER}.
} We consider a multiple equations system under a likelihood framework with normal errors, though the limit distribution of the proposed test remains valid with non-normal, serially dependent and heteroskedastic errors. 
Our framework allows integrated regressors and trends as well as stationary regressors as in \cite{BaiLumsdaineStock1998}
and also accommodates multiple breaks and arbitrary restrictions across  parameters as in \cite{QuPerron2007Emtca}.
Thus, our results apply for general systems of multiple equations considered in existing studies.
A case not covered in our framework is when the regressors depend on the break date.
This occurs when considering joint segmented trends and this issue was analyzed in \cite{Kim2017WP}.

%% 2) within/across equations 
Second, we propose a test for common breaks not only across equations 
within a multivariate system, but also within an equation. 
As in \cite{BaiLumsdaineStock1998}, the issue of common breaks is often associated 
with breaks occurring across equations, whereas 
one may want to test for common breaks in the parameters within a regression equation,
whether a single equation or a system of multiple equations are considered. 
More precisely, the null hypothesis of interest is that some subsets of the
basic parameters share one or more common break dates,
so that each regime is separated by some positive fraction of the sample size. 
Under the alternative hypothesis, the break dates are not the same and also need not be separated by a positive fraction of the sample size, or be asymptotically distinct.

%% 3) asymptotic derivation
%% Asymptotic 
Third, we derive the asymptotic properties of the quasi-likelihood and the parameter estimates, 
allowing for the possibility that the break dates associated with different basic parameters may not be asymptotically distinct.  
This poses an additional layer of difficulty,
since existing studies establish the consistency and rate of convergence of estimators
only when the break dates are assumed to either have a common location or be asymptotically distinct,
at least under the level of generality adopted here.
Moreover, we establish the results in the presence of 
integrated regressors and trends as well as stationary regressors.
This is by itself a noteworthy contribution. 
These asymptotic results will allow us to derive the limit distribution of our test statistic
under the null hypothesis and also facilitate asymptotic power analyses 
under fixed and local alternatives. 
We can show that 
our test is consistent under fixed alternatives and also has non-trivial local power.

%% Simulation 
There is one additional layer of difficulty compared to Bai and
Perron (1998) or Qu and Perron (2007). In their analysis, 
it is possible to transform the limit distribution so that 
it can be evaluated using a closed form solution
and thus critical values can be tabulated.
Here, no such solution is available and we need to obtain 
critical values for each case through simulations.
This involves simulating the Wiener processes with 
consistent parameter estimates 
and 
evaluating each realization of the limit distribution 
with and without the restriction of common breaks. 
While it is conceptually straightforward
and 
quick enough to be feasible for common applications,  
the procedure needs to be repeated many times to obtain the
relevant quantities
and 
can be quite computationally intensive. 
This is because we need to search over many possible combinations of all the permutations of the break locations for each replication of the simulations. 
To reduce the computational burden, we propose an alternative 
procedure based on 
the particle swarm optimization method developed by \cite{Eberhart1995}
with 
the Karhunen-Lo{\`e}ve representation of stochastic processes. 
Our simulation results suggest that 
the test proposed has reasonably good size and power performance even 
in small samples under both computation procedures.
Also, we apply our test to inflation series,
following the work of \cite{Clark2006JAE} to illustrate its usefulness.

%%% Outline 
The remainder of the paper is as follows.
Section 2 introduces the models with and without the common breaks assumption 
and describes the estimation methods under the quasi-likelihood framework.
Section 3 presents the assumptions and asymptotic results including 
the asymptotic null distribution and asymptotic power analyses.
Section 4 examines the finite sample properties of our procedure
via Monte Carlo simulations.
Section 5 presents an empirical application
and Section 6 concludes.
An appendix contains all the proofs.

%%%%%%%%%%%%%%%%%%%%%%%%%%%%%%%%%%%%%%%%%%%%%%%%%%%%%%%%%%%%%%%%%%%%%%%%%%%%%%%%%%
%%%%%%%%%%%%%%%%%%%%%%%%%%%%%%%%%%%%%%%%%%%%%%%%%%%%%%%%%%%%%%%%%%%%%%%%%%%%%%%%%%
\section{Models and quasi-likelihood method}

In this section, we first introduce models for a multiple equations system  
with and without common breaks.
Subsequently, we describe the quasi-likelihood estimation method
assuming normal errors
and then propose the quasi-likelihood ratio test for common breaks. 
For illustration purpose, we also discuss some examples. 

%%%%%%%%%%%%%%%%
% Notation
%%%%%%%%%%%%%%%%
As a matter of notation, \textquotedblleft $\overset{p}{\rightarrow }$%
\textquotedblright $\ $denotes convergence in probability, \textquotedblleft 
$\overset{d}{\rightarrow }$\textquotedblright $\ $convergence in
distribution and \textquotedblleft $\Rightarrow $\textquotedblright\ weak
convergence in the space $D[0,\infty)\ $under the Skorohod topology. 
We
use 
$\mathbb{R}$, 
$\mathbb{Z}$ 
and
$\mathbb{N}$ 
to denote the set of all real numbers, all integers and all positive integers,
respectively.
For a vector $x$, we use $\| \cdot \|$ to denote the Euclidean norm
(i.e., $\|x\|= \sqrt{x' x}$), while 
for a matrix $A$, we use the vector-induced norm 
(i.e., $\|A\|= \sup_{x \not = 0}\|A x\|/\|x\|$).
Define
the $L_{r}$-norm of a random matrix $X$ as $\left\Vert X\right\Vert
_{r}=(\sum_{i}\sum_{j}E\left\vert X_{ij}\right\vert ^{r})^{1/r}$ for $r\geq
1 $.
Also,
$a \wedge b=\min \{a,b\}$ 
and
$a \vee   b=\max \{a,b\}$ 
for any $a, b \in \mathbb{R}$. 
Let $\circ$ denote the Hadamard product (entry-wise product)
and 
let $\otimes$
denote the Kronecker product.
Define 
$\mathbbm{1}_{ \{\cdot \}}$
as 
the
indicator function taking value one when its argument is true, and zero otherwise 
and 
$e_{i}$ as a unit vector having 1 at the $i^{th}$ entry 
and 0 for the others.
We use the operator $\mathrm{vec}(\cdot)$ to convert a matrix into a column vector by stacking the columns of the matrix
and 
the operator $\mathrm{tr}(\cdot)$ to denote the trace of a matrix.
The largest integer not greater than $a \in \mathbb{R}$
is denoted by $[a]$
and 
the sign function is defined as 
$\mathrm{sgn}(a)= -1, 0, 1$  
if $a >0$, $a=0$ or $a <0$, respectively.

%%%%%%%%%%%%%%%%%%%%%%%%%%%%%%%%%%%%%%%%%%%%%%%%%%%%%%%%%%%%%%%%%%%%%%%%%%%%%%
\subsection{The models with and without common breaks}

%% data 
Let the data consist of observations 
$\{(y_{t}, x_{tT})\}_{ t=1}^{T}$,
where 
$y_{t}$ is an $n \times 1$ vector of dependent variables
and  
$x_{tT}$ is a $q \times 1$ vector of explanatory variables 
for $n, q \in \mathbb{N}$
with a subscript $t$ indexing a temporal observation 
and 
$T$ denoting the sample size.
We allow the regressors
$x_{tT}$ to include 
stationary variables,
time trends 
and 
integrated processes, 
while scaling by the sample size $T$ so that the order of all components is the
same.
In what follows, we consider
\begin{equation*} 
  x_{tT}=
  \big ( 
  z_{t}',
  \varphi(t/T)',
  T^{-1/2}w_{t}^{\prime }
  \big )'.
\end{equation*}%
Here,
$z_{t}$, $\varphi(t/T)$ and $w_{t}$ respectively denote 
vectors of stationary, trending and integrated variables 
with sizes being  $q_{z}{\times} 1$, $q_{\varphi}{\times} 1$
and  
$q_{w}{\times} 1$, 
so that  
$q\equiv q_{z}+q_{\varphi}+q_{w}$.\footnote{
  The normalization is simply a theoretical device to reduce notational burden. Without it, we would need to 
handle different convergence rates of the estimates
by
introducing 
additional notations.
% a diagonal matrix with elements $(1, T, \sqrt{T})$ all over the paper (when only a linear trend is present). Note that since our result shows the parameter estimates constructed with the normalized variables to converge at rate $\sqrt{T}$, this implies a different rate of convergence for the basic parameters; e.g.~rate $T$ for coefficients on integrated regressors and rate $T^{3/2}$  for the slope of a linear trend function.   
} 
Also,
\begin{eqnarray*}
\varphi(t/T)
:=
[(t/T), (t/T)^{2}, \dots, (t/T)^{q_{\varphi}}]' 
\ \ \ \ \mathrm{and} \ \ \ \
w_{t} =w_{t-1}+u_{wt}, 
\end{eqnarray*}
where 
$w_{0}\ $is assumed, for simplicity, to be either $O_{p}(1)$ random
variables or fixed finite constants,
and 
$u_{wt}$ is a vector of unobserved random variables 
with zero means. 
We label the variables $z_{t}$ as $I(0)$
if the partial sums of the associated noise components
satisfy a functional central limit theorem,
while we label a variable as $I(1)$ if it is the
accumulation of an $I(0)$ process. 
We discuss in more details the specific conditions in Section 3.

%% Model under the null
We first explain the case of common breaks through a model 
in which all of the parameters including those of the covariance matrix of the errors
change, i.e., a pure structural change model.
The model of interest is a multiple equations system 
with $n$ equations and $T$ time periods, 
excluding the initial conditions if lagged dependent variables
are used as regressors.
We denote the break dates in the system by $T_{1}, \dots, T_{m}$
with $m$ denoting the total number of structural changes
and 
we use 
the convention that $T_{0}=0$ and $T_{m+1}=T$.

%% Linear model 
With a subscript $j$ indexing a regime for $j=1,...,m+1$, the model is given by 
\begin{equation}
  \label{basic-dgp}
  y_{t}=(x_{tT}^{\prime }\otimes I_{n})S\beta _{j}+u_{t}, 
  \ \ \mathrm{for} \ T_{j-1} +1 \le t \le T_{j},
\end{equation}
where 
$I_{n}$ is an $n \times n$ identity matrix, 
$S$ is an $n q \times p$ selection matrix with full column rank,
$\beta_{j}$ is a $p \times 1$ vector of unknown coefficients, 
and 
$u_{t}$ is an $n \times 1$ vector of errors having zero means and covariance matrix  
$\Sigma_{j}$.\footnote{ 
  An example of models involving stationary and integrated variables is the dynamic ordinary least squares method 
to estimate cointegrating vectors
\citep[e.g.][]{Saikkonen1991ET, SW1993Emtca}. 
} 
The selection matrix $S$
usually consists of elements that are $0$ or $1$ and, hence,
specifies which regressors appear in each equation,
although in principle it is allowed to have entries that are arbitrary constants.
To ease notation, define the $n \times p$ matrix 
$X_{tT}:=$ $S' (x_{tT}\otimes I_n)$ so that (\ref{basic-dgp}) becomes,
for $j=1,...,m+1$,
\begin{equation}
  \label{eq:dgp-x}
  y_{t}=X_{tT}^{\prime }\beta _{j}+u_{t} ,  
  \ \ \mathrm{for} \ T_{j-1} +1 \le t \le T_{j}.
\end{equation}

%% Parameter Space and Restrictions 
The set of basic parameters in the $j^{th}$ regime consists of the coefficients 
$\beta_{j}$ and the elements of the covariance matrix $\Sigma _{j}$,
and we denote it by  
$\theta_{j}:=(\beta_{j}, \Sigma_{j})$
for each regime $j = 1, \dots, m+1$.
We use 
$ \Theta_{j} 
  \subset 
  \mathbb{R}^{p} 
  \times 
  \mathbb{R}^{n\times n} 
$
to denote a parameter space for $\theta_{j}$
and 
we also define 
a product space
$\Theta:= \Theta_{1} \times \dots \times \Theta_{m+1}$
for  
$\theta:=(\theta_{1}, \dots, \theta_{m+1})$.
In model (\ref{eq:dgp-x}),
we allow for the imposition of a set of $r$ restrictions through 
a function $R: \Theta \to \mathbb{R}^{r}$, given by
\begin{equation}  
  \label{restrictions}
  R( \theta )=0.
\end{equation}
Note that the equation in (\ref{restrictions}) can impose restrictions
both within and across equations and regimes. 
Thus the model in (\ref{eq:dgp-x}) with some restrictions of the form (\ref{restrictions})
can accommodate structural break models 
other than a pure structural change model,   
such as partial structural change models in which 
a part of the basic parameters are constant across regimes.
For a discussion of how general the framework is, see Qu and Perron (2007).% \footnote{
%   Here,  our argument concerns the generality of pure structural changes models 
%   together with the restrictions in (\ref{restrictions}).
%   If one wants to test the validity of the restrictions, then
%   one can resort the likelihood ratio test allowing for the possibility 
%   that break dates in different basic parameters are not common. 
%   As shown in Theorem 2 below, 
%   the break date estimates and the estimates of basic parameters 
%   are asymptotically independent, and thus one can derive the limit distribution of 
%   the likelihood ratio test for the restriction. 
% } 

%% Model under the general assumption
Next, we consider a pure structural change model allowing for the possibility 
that the break dates are not necessarily common across basic parameters. 
In the equations system with the $p \times 1$ vector of coefficients,
we can assign each coefficient an index from 
$1$ to $p$
and we then group the $p$ indices into 
disjoint subsets
$\mathcal{G}_{1}, \dots, \mathcal{G}_{G}
\subset \{1, \dots, p\}$
with
$G$ standing for the total number of groups,
such that 
coefficients indexed by elements of 
$\mathcal{G}_{g}$
share
the same break dates 
for each group $g=1,\dots, G$
and 
$\cup _{g=1}^{G}\mathcal{G}_{g} = \{1,...,p\}$.
Given a collection $\{\mathcal{G}_{g}\}_{g=1}^{G}$,
we define, for $(g, j) \in \{1, \dots, G\} {\times} \{1,...,m+1\}$, 
\begin{eqnarray}
  \label{eq:beta}
  \beta _{gj} :=  
  \sum_{l \in \mathcal{G}_{g}} e_{l} \circ \beta _{j}.
\end{eqnarray}
Without loss of generality, we assume
that the elements of the covariance matrix $\Sigma_{j}$ have break
dates that are common to those in the last group $G$. If none of the
regression coefficients change at the same time as the elements of the covariance matrix $\Sigma_{j}$, then $\mathcal{G}_{G}$ is simply an empty set.\footnote{
  We assume that the different elements of the covariance matrix of the errors
  change at the same time.
  The results can be extended to the case where
  different parameters 
  have distinct break dates, although additional notations would be needed. 
  For the sake of notational simplicity, we only consider 
  the case where the break dates are common within all elements of the covariance matrix. 
}
Here, we introduce groups of basic parameters to 
accommodate a wide range of empirical applications under our framework.
Sometimes, researchers have economic models of interest or empirical knowledge
that suggest specific parameter groups having common breaks.
Even when one has no knowledge to form parameter groups, 
our analysis can be applied by considering
all basic parameters
as separate groups.

%% model
To denote the break date for regime $j$ and group $g$,
we use $k_{gj}$ 
for $(g, j) \in \{1, \dots, G\} \times \{1, \dots, m\}$
with the convention that $k_{g0}=0$ and $k_{g,m+1}=T$
for any $g= 1, \dots, G$.
Also, define a collection of break dates as,
\begin{eqnarray*}
  \mathcal{K}:=\{\mathcal{K}_{1}, \dots, \mathcal{K}_{G}\}
  \ \ \mathrm{with} \ \
  \mathcal{K}_{g}:=(k_{g1}, \dots, k_{gm})
  \ \mathrm{for} \ g = 1, \dots, G.
\end{eqnarray*}
The regression model can be expressed as 
one depending on time-varying basic parameters according to the collection $\mathcal{K}$:
\begin{equation}
  \label{eq:dgp-K}
  y_{t}=X_{tT}^{\prime }\beta _{t,\mathcal{K}}+u_{t},
\end{equation}%
where 
$\beta_{t,\mathcal{K}} 
  :=
  \sum_{g=1}^{G}\beta _{g, t, \mathcal{K}}
$
and 
$E[u_{t}u_{t}^{\prime }] = \Sigma_{t,\mathcal{K}}$ with  
\begin{eqnarray} 
  \label{eq:time-vary}
  \beta _{g, t,\mathcal{K} } := 
  \beta_{gj}
  \ \ \mathrm{for} \ 
  k_{g,j-1} + 1 \le  t \le k_{gj} 
  \ \ \ \mathrm{and} \ \ \
  \Sigma _{t,\mathcal{K}} :=\Sigma_{j}
  \ \ \mathrm{for} \  
  k_{G,j-1} +1 \le   t\leq k_{Gj},
\end{eqnarray}
for $(g, j) \in \{1, \dots, G\} {\times} \{1,...,m+1\}$. 
We also use  
$\theta_{t, \mathcal{K}}:= 
(\beta_{t, \mathcal{K}}, \Sigma_{t, \mathcal{K}})
$
to denote time-varying basic parameters depending on 
the collection of break dates $\mathcal{K}$.
% We can express $\theta_{j}=(\beta_{j}, \Sigma_{j})$
% in 
% the model  (\ref{eq:dgp-K})
% because 
% we have $\beta_{j} = \sum_{g=1}^{G}\beta_{gj}$
% from definition. 
Thus the restrictions (\ref{restrictions}) can be imposed on 
the system (\ref{eq:dgp-K}) to accommodate more general models 
with structural breaks as in the one with common breaks.

%%% True break model 
In model (\ref{eq:dgp-K}),
the basic parameters, break dates and the number of breaks are unknown 
and have to be estimated. 
To select the total number of structural changes,
we can apply existing sequential testing procedures 
or information criteria. 
For example, if the breaks are common within each equation under both null and alternative hypotheses, but may differ across equations (see Example 1 below),
sequential testing procedures proposed by \cite{BaiPerron1998Emtca}
can be used to select the number of structural changes in 
each equation of a system
\citep[see][p.~65, for a discussion of the statistical properties of such sequential procedures]{BaiPerron1998Emtca}.
In a similar way, the sequential testing procedure 
in \cite{QuPerron2007Emtca} can be applied 
for sets of equations of a system separately. 
In order to handle more complex cases,
we can alternatively use the Bayesian information criterion
or the minimum description length principle
as in \cite{KurozumiTu2011JoE}, \cite{Lee2000JASA}
and \cite{aue2011}. 
Because we use the likelihood framework,
a likelihood function with a relevant penalty 
can be computed with the use of 
genetic algorithms \citep[see][for example]{Davis2001},
which 
consistently selects the number of structural breaks,
as in \cite{Lee2000JASA} and \cite{aue2011}. 
Thus, our analysis in what follows focuses on unknown basic parameters and  breaks dates,
given a total number of structural changes.

%% True parameters
We use a $0$ superscript to denote the true values of the parameters in 
both (\ref{eq:dgp-x}) and (\ref{eq:dgp-K}).
Thus,
the true basic parameters 
and 
break dates in (\ref{eq:dgp-x}) 
are denoted by 
$\{(\beta_{j}^{0},\Sigma_j^0)\}_{j=1}^{m+1}$ 
and 
$\{T_{j}^{0}\}_{j=1}^{m}$, respectively,
with the convention that $T_{0}^{0}=0$ and $T_{m+1}^{0}=T$,
whereas 
the ones in (\ref{eq:dgp-K}) 
are denoted by 
$\{\beta_{1j}^{0}, \dots, \beta_{Gj}^{0}, \Sigma_{j}^{0}\}_{j=1}^{m+1}$
and  
$\mathcal{K}_{g}^{0}:=(k_{g1}^{0},\dots, k_{gm}^{0})$
with $k_{g0}^{0}=0$ and $k_{g,m+1}^{0}=T$
for $g =1, \dots, G$. 
Also let 
$\mathcal{K}^{0}:=\{\mathcal{K}_{1}^{0}, \dots, \mathcal{K}_{G}^{0} \}$.
Given a collection of break dates $\mathcal{K}$,
let
$\theta_{t, \mathcal{K}}^{0}:= 
(\beta_{t, \mathcal{K}}^{0}, \Sigma_{t, \mathcal{K}}^{0})
$
with a $0$ superscript to denote 
time-varying true basic parameters 
$\theta^{0}$,
where
$\theta^{0}:=(\theta_{1}^{0}, \dots, \theta_{m+1}^{0})$
with  
$\theta_{j}^{0}
 := 
 (\beta_{j}^{0},\Sigma_j^0)
$
for $j=1, \dots, m+1$.

%%%%%%%%%%%%%%%%%%%%%%%%%%%%%%%%%%%%%%%%%%%%%%%%%%%%%%%%%%%%%%%%%%%%%%%%%%%%%%%%%%%%%
%%%%%%%%%%%%%%%%%%%%%%%%%%%%%%%%%%%%%%%%%%%%%%%%%%%%%%%%%%%%%%%%%%%%%%%%%%%%%%%%%%%%%
\subsection{The estimation and test under  the quasi-likelihood framework}

%% model and restrictions 
We consider the quasi-maximum likelihood estimation method 
with serially uncorrelated Gaussian errors
for model (\ref{eq:dgp-K}) with restrictions given by (\ref{restrictions}).\footnote{
  Our framework includes OLS-based estimation by setting the
  covariance matrix to be an identity matrix.  
}
Given the collection of break dates $\mathcal{K}$ and the basic parameters $\theta$,
the Gaussian quasi-likelihood function is defined as
\begin{equation*}
  L_{T}(\mathcal{K},\theta)
  :=
  \prod_{t=1}^{T}f(y_{t}|X_{tT},\theta_{t, \mathcal{K}}),
\end{equation*}
where
\begin{equation*}
  f(y_{t}|X_{tT}, \theta_{t, \mathcal{K}})
  :=
  \frac{
    1
  }{
    (2 \pi)^{n/2} |\Sigma_{t, \mathcal{K}}|^{1/2}
  }
  \exp
  \Big (
    -\frac{1}{2}
    \big \|
    \Sigma _{t,\mathcal{K}}^{-1/2}(y_{t}-X_{tT}^{\prime }\beta _{t,\mathcal{K}}) 
    \big \|^{2}
  \Big ).
\end{equation*}
To obtain maximum likelihood estimators, 
we impose a restriction on the set of permissible 
partitions with a trimming parameter $\nu>0$ as follows\footnote{
  For the asymptotic analysis, the trimming value
  $\nu$ can be an arbitrary small constant such that a positive fraction of the
  sample size $T\nu$ diverges at rate $T$.
}:
\begin{eqnarray*}
  \Xi_{\nu}
  :=
  \Big \{
  \mathcal{K}:
  \min_{1 \le g \le G}
  \min_{1 \le j \le m+1}
  (k_{gj} - k_{g,j-1} )\ge T\nu
  \Big \}.
\end{eqnarray*}
This set of permissible partitions ensures that 
there are enough observations between any break dates within the same group $\mathcal{K}_{g}$,
while it accommodates the possibility that 
the break dates across different groups are not separated by a positive fraction of the sample size.

%% Hypothesis 
We propose a test for common breaks under the quasi-likelihood framework. 
The null hypothesis of common breaks in model (\ref{eq:dgp-x}) can be stated as 
\begin{eqnarray}
  \label{eq:H0}
  H_{0}:
  \mathcal{K}_{g_{1}}^{0} = \mathcal{K}_{g_2}^{0}
  \ \ \ \mathrm{for \ all} \ g_{1}, g_{2}
  \in \{1, \dots, G\},
\end{eqnarray}
and the alternative hypothesis is 
\begin{eqnarray}
  \label{eq:H1-Alt}
  H_{1}:
  \mathcal{K}_{g_{1}}^{0} \not = \mathcal{K}_{g_2}^{0}
  \ \ \ \mathrm{for \ some} \ g_{1}, g_{2}
  \in \{1, \dots, G\}.
\end{eqnarray}
The set of permissible partitions
under the null hypothesis 
can be expressed as 
\begin{eqnarray*}
  \Xi_{\nu, H_{0}}
  :=
  \{
  \mathcal{K} 
  \in
  \Xi_{\nu}:
  \mathcal{K}_{1} = \cdots = \mathcal{K}_{G}
  \}. 
\end{eqnarray*}

%% LR estimation 
The test considered is simply the quasi-likelihood ratio test that compares the
values of the likelihood function with and without the common breaks restrictions.
The quasi-maximum likelihood estimates 
under the null hypothesis, 
denoted by $(\tilde{\mathcal{K}}, \tilde{\theta} )$,
can be obtained  
from the following maximization problem 
with a restricted set of candidate break dates:
\begin{equation*}
  (\tilde{\mathcal{K}}, \tilde{\theta} )
  :=
  \argmax_{(\mathcal{K},\theta ) \in {\Xi}_{\nu, H_{0}} \times \Theta}
  \log L_{T}(\mathcal{K},\theta)
  \ \ \  
  \mathrm{s.t.}
  \ \ 
  R(\theta)  = 0,
\end{equation*}
where 
$\tilde{\mathcal{K}}:= (\tilde{\mathcal{K}}_{1}, \dots, \tilde{\mathcal{K}}_{G})$
with 
$\tilde{\mathcal{K}}_{g}:=(\tilde{k}_{1}, \dots, \tilde{k}_{m})$
for all $g=1,\dots, G$,
$\tilde{\theta}:=(\tilde{\beta}, \tilde{\Sigma})$
with 
$\tilde{\beta}:=(\tilde{\beta}_{1}, \dots, \tilde{\beta}_{m+1})$
and 
$\tilde{\Sigma}:=(\tilde{\Sigma}_{1}, \dots, \tilde{\Sigma}_{m+1})$.
Also, the quasi-maximum likelihood estimates under the alternative,
denoted by $(\hat{\mathcal{K}}, \hat{\theta} )$,
are obtained from the following problem: 
\begin{equation}
  \label{k-hat}
  (\hat{\mathcal{K}}, \hat{\theta} )
  :=
  \argmax_{(\mathcal{K},\theta ) \in \Xi_{\nu} \times \Theta}
  \log L_{T}(\mathcal{K},\theta)
  \ \ \  
  \mathrm{s.t.}
  \ \ 
  R(\theta)  = 0,
\end{equation}
where
$\hat{\mathcal{K}}:= (\hat{\mathcal{K}}_{1}, \dots, \hat{\mathcal{K}}_{G})$
with 
$\hat{\mathcal{K}}_{g}:=(\hat{k}_{g1}, \dots, \hat{k}_{gm})$
for $g=1,\dots, G$,
$\hat{\theta}:=(\hat{\beta}, \hat{\Sigma})$
with 
$\hat{\beta}:=(\hat{\beta}_{1}, \dots, \hat{\beta}_{m+1})$
and 
$\hat{\Sigma}:=(\hat{\Sigma}_{1}, \dots, \hat{\Sigma}_{m+1})$.
Using the estimates $\hat{\theta}$, we can define 
$\hat{\beta}_{gj}$ as in (\ref{eq:beta})
and 
$\hat{\theta}_{t, \mathcal{K}}:= 
(\hat{\beta}_{t, \mathcal{K}}, \hat{\Sigma}_{t, \mathcal{K}})
$
as in (\ref{eq:time-vary}) 
given a collection of break dates $\mathcal{K}$.

%% LR test 
We define the quasi-likelihood ratio test for common breaks as  
\begin{eqnarray*}
  CB_{T} 
  := 
  2 
  \{
  \log L_{T}(\hat{\mathcal{K}},\hat{\theta}) 
  -
  \log L_{T}(\tilde{\mathcal{K}},\tilde{\theta})
  \}. 
\end{eqnarray*} 
For the asymptotic analysis, it is useful to 
employ a normalization by using the log-likelihood function evaluated at
the true parameters $(\mathcal{K}^{0}, \theta^{0})$
and we consider 
\begin{eqnarray*}
  CB_{T}
  =   
  2\{
  \ell_{T}(\hat{\mathcal{K}},\hat{\theta}) 
  -
  \ell_{T}(\tilde{\mathcal{K}},\tilde{\theta})
  \},
\end{eqnarray*}
where 
$  \ell_{T}(\mathcal{K}, \theta)
  :=
  \log L_{T}(\mathcal{K},\theta) 
  -
  \log L_{T}(\mathcal{K}^{0},\theta^{0})
$
for any $(\mathcal{K}, \theta) \in \Xi_{\nu} \times \Theta$.
The common break test $CB_{T}$ depends on two log-likelihoods
with and without the common breaks assumption.
The break date estimates $\tilde{\mathcal{K}}$ under the null hypothesis are required to either 
have common locations 
or 
be separated by a positive fraction of the sample size.
Without common breaks restrictions, however, 
the break date estimates $\hat{\mathcal{K}}$ are
simply allowed to be distinct but not necessarily separated by a positive
fraction of the sample size across groups.
This will be important since the setup of 
\cite{Bai2000AEF}
and
\cite{QuPerron2007Emtca} requires the maximization to be taken over asymptotically
distinct elements and their proof for the convergence rate of the estimates
relies on this premise. Hence, we will need to provide a detailed proof of
the convergence rate under this less restrictive
maximization problem (see Section 3).

%%%%%%%%%%%%%%%%%%%%%%%%%%%%%%%%%%%%%%%%%%%%%%%%%%%%%%%%%%%%%%%%%%%%%%%%%%%%%%%%
\subsection{Examples}

Given that the notation is rather complex, it is useful to illustrate the
framework explained in the preceding subsection via examples. 

\vspace{0.2cm}
\noindent
\textbf{Example 1 (changes in intercepts):}
We consider a two-equations system of autoregressions
with structural changes in intercepts,
for $j=1,2$,
\begin{eqnarray*}
  y_{1t} = \mu_{1 j} +\alpha_{1} y_{1,t-1} +u_{1t}
  \ \ \ \mathrm{and} \ \ \
  y_{2t} = \mu_{2 j} +\alpha_{2} y_{2,t-1}+u_{2t},
  \ \ 
  \mathrm{for} 
  \ T_{j-1}+1 \le t \le T_{j},
\end{eqnarray*}
where $(u_{1t}, u_{2t})'$ have a covariance matrix $\Sigma$.
In this model, 
the basic parameters except the intercepts are assumed to be constant
and
the intercepts change at a common break date $T_{1}$.
In equation (\ref{basic-dgp}),
we have 
$x_{tT}=(1, y_{1,t-1}, y_{2, t-1})'$, 
$\beta_{j}=(\mu_{1j}, \alpha_{1j}, \mu_{2j}, \alpha_{2j})'$ 
and 
$E[u_{t} u_{t}'] = \Sigma_{j}$.
The selection matrix $S=<s_{ij}>$
is a $6\times 4$ matrix taking value 1 at the entries $s_{11}$,$s_{22}$, $s_{33}$ and $s_{64}$
and 0 elsewhere.
Also, by setting 
$R(\theta) =
\big(\alpha_{11}-\alpha_{12}, \alpha_{21}-\alpha_{22}, 
  \mathrm{vec}(\Sigma_{1})-  
  \mathrm{vec}(\Sigma_{2}) 
\big )' = 0$
in  (\ref{restrictions}),
we impose restrictions on the basic parameters
so that 
a partial structural change model is considered 
with no changes in the 
autoregressive parameters and the 
covariance matrix of the errors. 
On the other hand, when we allow the possibility that 
break dates can differ across the two equations
as in the model (\ref{eq:dgp-K}), we consider 
the following system, for $j=1,2$,
\begin{eqnarray*}
  y_{1t} &=& \mu_{1 j} +\alpha_{1} y_{1,t-1} +u_{1t}, 
  \ \ \mathrm{for} \ k_{1,j-1} +1 \le t \le k_{1j},\\
  y_{2t} &=& \mu_{2 j} +\alpha_{2} y_{2,t-1}+u_{2t},
  \ \ \mathrm{for} \  k_{2,j-1} + 1 \le t \le k_{2j}. 
\end{eqnarray*}
Here, we separate $\beta_{j}$
into 
$\beta_{1j}=(\mu_{1j}, \alpha_{1j}, 0, 0)'$
and 
$\beta_{2j}=(0, 0, \mu_{2j}, \alpha_{2j})'$, 
so that 
we can set 
$\mathcal{G}_{1}=\{1, 2\}$ and 
$\mathcal{G}_{2}=\{3, 4\}$.  
We have two possibly distinct break dates $k_{11}$ and $k_{21}$
for the parameter groups $\{\beta_{1j}\}_{j=1}^{2}$ and $\{(\beta_{2j}, \Sigma_{j})\}_{j=1}^{2}$, respectively.
We address the issue of testing the null hypothesis 
$H_{0}: k_{11} = k_{21}$
against the alternative hypothesis 
$H_{1}: k_{11} \not = k_{21}$.

\vspace{0.2cm}
\noindent
\textbf{Example 2 (a single  equation model):}
Consider a single equation model:
\begin{eqnarray*}
  y_{1t} &=& \mu + \alpha_{j} z_{1,t}+ 
  \gamma_{j}(t/T)  + \rho_{j} T^{-1/2} w_{1t} + u_{1t}, 
\end{eqnarray*}
for $T_{j-1} +1 \le t \le T_{j}$ with $j = 1, 2, 3$,
where
$u_{1t}$ denotes the error term 
with 
$E[u_{1t}] = 0$
and 
$E[u_{1t}^{2}] = \sigma_{j}^{2}$.
In this example, the basic parameters other than the intercepts 
have two structural changes.
Under model (\ref{eq:dgp-x}) with break dates $T_{1}$
and $T_{2}$, we have 
$x_{tT}=(1, z_{1t}, t/T, T^{-1/2}w_{1t})'$, 
$S= I_{4}$,
$\beta_{j}=(\mu_{j}, \alpha_{j}, \gamma_{j}, \rho_{j})'$. 
Restrictions of the form (\ref{restrictions})
are imposed by the function
$R(\theta) =
(\mu_{1} - \mu_{2}, \mu_{2} - \mu_{3}
)' = 0$.
We consider a test for common breaks against the 
alternative that all coefficients change at distinct break dates, 
while the coefficient $\rho_{j}$ and the variance $\sigma_{j}^{2}$ 
change at the same break dates.
In this case, 
we separate $\beta_{j}$ into three vectors 
$\beta_{1j}=(\mu_{j}, \alpha_{j}, 0, 0)$, 
$\beta_{2j}=(0, 0, \gamma_{j}, 0)$
and 
$\beta_{3j}=(0, 0, 0, \rho_{j})$. 
For these parameters groups, we assign a set of break dates 
$\mathcal{K}_{g}=(k_{g1}, k_{g2})$ for $g=1,\dots, 3$
and 
we set 
$\mathcal{G}_{1}=\{1, 2\}$, 
$\mathcal{G}_{2}=\{3\}$
and 
$\mathcal{G}_{3}=\{4\}$.
The break dates for the last group, $\mathcal{K}_{3}$, are also the ones for the variance. 
This example shows that our framework can accommodate common breaks not only across equations in a system but also within an equation.

%%%%%%%%%%%%%%%%%%%%%%%%%%%%%%%%%%%%%%%%%%%%%%%%%%%%%%%%%%%%%%%%%%%%%%%%%%%%%%%%%%%%%%
%%%%%%%%%%%%%%%%%%%%%%%%%%%%%%%%%%%%%%%%%%%%%%%%%%%%%%%%%%%%%%%%%%%%%%%%%%%%%%%%%%%%%%
%%%%%%%%%%%%%%%%%%%%%%%%%%%%%%%%%%%%%%%%%%%%%%%%%%%%%%%%%%%%%%%%%%%%%%%%%%%%%%%%%%%%%%
\section{Asymptotic results}

This section presents the relevant asymptotic results. 
We first provide the convergence rates of the estimates of the break dates and the basic parameters,
allowing for the possibility that 
the break dates of different basic parameters may not be asymptotically distinct.
This condition is substantially less restrictive than 
the ones usually assumed in the existing literature
and   
particularly includes the assumption of common breaks as a special case.
Next, we provide the limiting distribution of the quasi-likelihood ratio test for common breaks under the null hypothesis. 
Finally, we provide asymptotic power analyses of the test under 
a fixed alternative as well as a local one.
Our result shows non-trivial asymptotic power.

%%%%%%%%%%%%%%%%%%%%%%%%%%%%%%%%%%%%%%%%%%%%%%%%%%%%%%%%%%%%%%%%%%%%%%%%%%%%%%%%%%%%%%
%%%%%%%%%%%%%%%%%%%%%%%%%%%%%%%%%%%%%%%%%%%%%%%%%%%%%%%%%%%%%%%%%%%%%%%%%%%%%%%%%%%%%%
\subsection{The rate of convergence of the estimates.}

We consider the case where 
we obtain the quasi-likelihood estimates $(\hat{\mathcal{K}}, \hat{\theta})$
as in (\ref{k-hat}),
using the observations $\{(y_t , x_{tT} )\}_{t=1}^{T}$
generated by model (\ref{eq:dgp-K})
with collections of true parameter values $(\mathcal{K}^{0}, \theta^{0})$.
The results presented in this subsection can apply for 
the estimates obtained from the model under the null hypothesis
since it is a special case of the setup adopted.
To obtain the asymptotic results, 
the following assumptions are imposed.

%%%%%%%%%%%%%%%%%%%%%%%%%%%%%%%%%%%%%%%%%%%%%%%%%%%%%%%%%%%%%%%%%%%%%%%%%%%%
\vspace{0.1cm}
\noindent 
\textbf{Assumptions:}
\begin{description}
  \item[A1.] 
    There exists a constant $k_{0}>0\ $such that for all $k>k_{0}$, 
    the minimum eigenvalues of the matrices 
    $k^{-1}\sum_{t=s}^{s+k}x_{tT}x_{tT}^{\prime }$ 
    are bounded away from zero 
    for every $s = 1,\dots, T-k$.

  \item[A2.] 
    Define the sigma-algebra 
    $\mathcal{F}_{t}:=\sigma(\{z_{s}, u_{ws}, \eta_{s}\}_{s\le t})$
    for $t \in \mathbb{Z}$,
    where
    $\eta_{s}:= (\Sigma_{s, \mathcal{K}^{0}}^{0})^{-1/2} u_{s}$.
    (a) 
    Define $\zeta_{t}:= (z_{t}', u_{wt}')'$
    and 
    let $z_{t}$ include a constant term.
    The sequence 
    $\left\{ 
       \zeta_{t} \otimes \eta_{t},\mathcal{F}_{t}\right\}_{t \in \mathbb{Z}}$ 
    forms a strongly mixing ($\alpha $-mixing)
    sequence with size $-(4+\delta)/\delta$ for some $\delta \in (0,1/2)$
    and 
    satisfies 
    $E[z_{t}\otimes \eta_{t}]=0$
    and 
    $\sup_{t\in \mathbb{Z}}\left\Vert \zeta _{t}\otimes \eta_{t} \right\Vert _{4+\delta}<\infty$. 
    (b)
    It is also assumed that $\{\eta_{t}\eta_{t}^{\prime }-I_{n}\}_{t\in \mathbb{Z}}$ 
    satisfies the same mixing and moment conditions 
    as in (a).
    (c)
    The sequence 
    $\{w_{0} \otimes \eta_{t}\}_{t \in \mathbb{Z}}$
    forms a strong mixing sequence as in (a)
    with $\sup_{t \in \mathbb{Z}}\|w_{0} \otimes \eta_{t}\|_{4+\delta}<\infty$
    and 
    the initial condition $w_{0}$ is $\mathcal{F}_{0}$-measurable.

  \item[A3.] 
    The collection of the true break dates 
    $\mathcal{K}^{0}$ is included in $\Xi_{\nu}$
    and 
    satisfies 
    $k_{gj}^{0}=\left[ T\lambda _{gj}^{0}\right] $
    for every 
    $(g, j) \in \{1, \dots, G\} {\times} \{1,...,m\}$, 
    where  
    $0<\lambda _{g1}^{0}<\dots<\lambda _{gm}^{0}<1$.

  \item[A4.] 
    For every parameter group $g$ and regime $j$,
    there exists 
    a $p \times 1$ vector $\delta_{gj}$ 
    and 
    an $n \times n$ matrix $\Phi_{j}$
    such that 
    $\beta _{g, j+1}^{0}-\beta _{g j}^{0}=v_{T}\delta _{gj}$ and 
    $\Sigma _{j+1}^{0}-\Sigma _{j}^{0}=v_{T}\Phi _{j}$,
    where 
    both 
    $\delta_{gj}$ and $\Phi _{j}$ are independent of $T$,
    and 
    $v_{T}>0\ $is a scalar satisfying 
    $v_{T}\rightarrow 0\ $and 
    $\sqrt{T}v_{T}/\log T \rightarrow \infty \ $as 
    $T\rightarrow \infty $.
    Let $\delta_{j}:=\sum_{g=1}^{G} \delta_{gj}$ 
    for $j=1,\dots, m+1$.

  \item[A5.]
    The true basic parameters 
    $(\beta^{0}, \Sigma^{0})$
    belong to 
    the compact parameter space 
    \begin{eqnarray*}
      \Theta
      :=
      \Big \{
      \theta:
      \max_{1 \le j \le m +1}
      \|\beta_{j}\| \le c_{1},  \
      c_{2}
      \le 
      \min_{1 \le j \le m+1}
      \lambda_{\min}(\Sigma_{j}), \
      \max_{1 \le j \le m+1}
      \lambda_{\max}(\Sigma_{j})
      \le c_{3}
      \Big \},      
    \end{eqnarray*}
    for some constants $c_{1}< \infty$,
    $0 < c_{2} \le c_{3} < \infty$,
    where 
    $\lambda_{\min}(\cdot)$
    and 
    $\lambda_{\max}(\cdot)$
    denote
    the smallest and largest eigenvalues 
    of the matrix in its argument, respectively.
\end{description}

Assumption A1 ensures that there is no local collinearity problem so that 
a standard invertibility requirement holds 
if the number of observations in some sub-sample 
is greater than $k_{0}$, not depending on $T$.
Assumption A2 determines the dependence structure of 
$\{\zeta_{t} \otimes \eta_{t}\}$,
$\{\eta_{t}\eta_{t}' - I_{n}\}$
and 
$\{w_{0} \otimes \eta_{t}\}$
to guarantee that they are short memory processes 
and have bounded fourth moments. The assumptions are
imposed to obtain a functional central limit theorem and a generalized 
\cite{HajekRenyi1955AMH} type inequality that 
allow us to derive the relevant convergence rates. 
Assumption A2 also specifies that the stationary regressors are contemporaneously
uncorrelated with the errors
and 
that  
a constant term is included in $z_{t}$. 
The former is a standard requirement to obtain consistent estimates
and 
the latter is for notational simplicity 
since the results reported below are the same without a constant term.\footnote{
  One can use the usual ordinary least squares framework
  to simply estimate the break dates and test for structural change
  even in the presence of the correlation between 
  the stationary regressors and the errors
  \citep[see][]{PerronYamamoto2015JAE}.
  One may also use a two-stage least squares method 
  if relevant instrumental variables are available 
  \citep[see][]{HallHanBoldea2012JoE, PerronYamamoto2014}. 
}\footnote{
  When a constant term is not included in $z_{t}$, in contrast to Assumption A2, 
  one additionally needs to assume that 
  the sequence $\{\eta_{t}\}_{t \in \mathbb{Z}}$ satisfies 
  the same mixing and moment conditions as in Assumption A2(a).
}
It is important to note that no assumption is
imposed on the correlation between the innovations to the $I(1)$ regressors
and the errors. Hence, we allow endogenous $I(1)$ regressors. 
Assumption A3 ensures that 
$\lambda_{gj}^{0} -\lambda_{g,j-1}^{0}  > \nu$ holds 
for every pair of group and regime  $(g, j)$
and 
thus implies asymptotically distinct breaks within each parameter group,
but not necessarily across groups.
Assumption A4 implies a shrinking shifts asymptotic framework whereby the magnitudes of the shifts converge to zero as the sample size increases. 
This condition is necessary to develop a limit distribution theory 
for the estimates of the break dates that does not depend on the exact distributions 
of the regressors and the errors,
as commonly used in the literature 
\citep[e.g.,][]{Bai1997REStat, BaiPerron1998Emtca, BaiLumsdaineStock1998}.
Assumption A5 implies that 
the data are generated by a model with a finite conditional mean
and innovations having a non-degenerate covariance matrix.

As stated above, the break dates are estimated from a set $\Xi_{\nu}$,
which requires candidate break dates to be separated by some
fraction of the sample size only within parameter groups.
Thus, we cannot appeal to the results in 
\cite{Bai2000AEF} and Qu and Perron (2007) about the rate of convergence of the
estimates, and more general results are needed. 
The following theorem presents results about the convergence rates of 
the estimates.

%%%%%%%%%%%%%%%%%%%%%%%%%%%%%%%%%%%%%%%%%%%%%%%%%%%%%%%%%%%%%%%%%%%%%%%%%%%%%%%%%%%%%%
\begin{theorem}
  \label{theorem:rate}
  Suppose that Assumptions A1-A5 hold. Then, \\
  (a) 
  uniformly in $(g, j) \in \{1,...,G\} \times \{1,...,m \}$,
  \begin{eqnarray*}
    v_{T}^{2}
    (
    \hat{k}_{gj}-k_{gj}^{0}
    )
    = 
    O_p(1),
  \end{eqnarray*}
  (b)
  uniformly in $(g, j) \in \{1,...,G\} {\times} \{1,...,m +1\}$,
  \begin{eqnarray*}
    \sqrt{T}(\hat{\beta}_{gj}-\beta _{gj}^{0})=O_{p}(1)  
    \ \ \mathrm{and}  \ \ 
    \sqrt{T}(\hat{\Sigma}_{j}-\Sigma_{j}^{0})=O_{p}(1).
  \end{eqnarray*}
\end{theorem}
%%%%%%%%%%%%%%%%%%%%%%%%%%%%%%%%%%%%%%%%%%%%%%%%%%%%%%%%%%%%%%%%%%%%%%%%%%%%%%%%%%%%%%

This theorem establishes the convergence rates 
obtained in 
\cite{BaiPerron1998Emtca},  
\cite{BaiLumsdaineStock1998},
\cite{Bai2000AEF}
and
\cite{QuPerron2007Emtca},
while assuming less restrictive conditions
regarding the optimization problem and
the time-series properties of the regressors.

The importance of these results is that they will allow us to analyze the
properties of our test under compact sets for the parameters, namely,
for some $M>0$,
\begin{eqnarray*}
  \bar{\Xi}_{M} 
  &:=&
  \big \{
  \mathcal{K} \in \Xi_{\nu}
  :
  \max_{1 \le g \le G}
  \max_{1 \le j \le m}
  |k_{gj}-k_{gj}^{0}|\leq
  Mv_{T}^{-2}
  \big \} \\
  \bar{\Theta}_{M}
  &:=&
  \big \{
  \theta \in \Theta:
  \max_{1 \le g \le G}
  \max_{1 \le j \le m+1} 
  \|\beta_{gj}-\beta_{gj}^{0} \|
  \le
  M T^{-1/2},
  \max_{1 \le j \le m+1}
  \|\Sigma _{j}-\Sigma _{j}^{0}\|\leq M T^{-1/2}
  \big \}.
\end{eqnarray*}%

We also have a result that expresses the restricted likelihood in two parts:
one that involves only the break dates and the true values of the coefficients;
the other involving 
the true values of the break dates, 
the basic parameters
and 
the restrictions.
Thus, asymptotically the estimates of the break dates are not affected by the restrictions 
imposed on the coefficients, while the limiting distributions of these estimates are
influenced by the restrictions. 

%%%%%%%%%%%%%%%%%%%%%%%%%%%%%%%%%%%%%%%%%%%%%%%%%%%%%%%%%%%%%%%%%%%%%%%%%%%%%%%%%%%%%%
\begin{theorem}
  \label{theorem:parts}
  Suppose that Assumptions A1-A5 hold. Then,
  \begin{eqnarray}
    \label{eq:qqq}
    \sup_{(\mathcal{K}, \theta) \in \bar\Xi_{M} \times    \bar{\Theta}_{M}}
    \ell_{T, R}(\mathcal{K},\theta)
    = 
    \sup_{\mathcal{K} \in \bar\Xi_{M} }
    \ell_{T}(\mathcal{K},\theta^{0})
    +
    \sup_{ \theta \in  \bar{\Theta}_{M}}
    \ell_{T, R}(\mathcal{K}^{0},\theta)
    +
    o_{p}(1),
  \end{eqnarray}
  where 
  $
    \ell_{T, R}(\mathcal{K},\theta)
    :=
    \ell_{T}(\mathcal{K},\theta)
    + 
    \gamma' R(\theta)
  $
  with
  a Lagrange multiplier
  $\gamma$.
\end{theorem}  
%%%%%%%%%%%%%%%%%%%%%%%%%%%%%%%%%%%%%%%%%%%%%%%%%%%%%%%%%%%%%%%%%%%%%%%%%%%%%%%%%%%%%%

The result in Theorem 2 implies that when analyzing the asymptotic properties of the break date estimates, one
can ignore the restrictions in (\ref{restrictions}). This will prove especially convenient to obtain
the limit distribution of our test.
Since the quasi-likelihood ratio test can be expressed as a difference 
of two normalized log likelihoods evaluated at different break dates, 
the second term on the right-hand side of (\ref{eq:qqq}) is canceled out
in the test statistic.
The result in Theorem 2 has been obtained in \cite{Bai2000AEF}
for vector autoregressive models
and
\cite{QuPerron2007Emtca}
for more general stationary regressors,
when break dates are assumed to either have a common location or be asymptotically distinct. 
We establish the results, allowing for the possibility that the break dates associated with different basic parameters may not be asymptotically distinct,
and thus expand the scope of prior work such as 
\cite{BaiLumsdaineStock1998}, \cite{Bai2000AEF}
and \cite{QuPerron2007Emtca}.

%%%%%%%%%%%%%%%%%%%%%%%%%%%%%%%%%%%%%%%%%%%%%%%%%%%%%%%%%%%%%%%%%%%%%%%%%%%%%%%%%%%%%%
%%%%%%%%%%%%%%%%%%%%%%%%%%%%%%%%%%%%%%%%%%%%%%%%%%%%%%%%%%%%%%%%%%%%%%%%%%%%%%%%%%%%%%
\subsection{The limit distribution of the likelihood ratio test}

We now establish the limit distribution of the quasi-likelihood ratio test
under the null hypothesis of common breaks in (\ref{eq:H0}).  
To this end, let the data consist of 
the observations $\{(y_t , x_{tT} )\}_{t=1}^{T}$
from model (\ref{eq:dgp-x})
with true basic parameters $\theta^{0}=(\beta^{0}, \Sigma^{0})$
and 
true break dates $\mathcal{T}^{0}$
consisting of $T_{1}^{0}, \dots, T_{m}^{0}$.
Theorem \ref{theorem:rate}(a) shows that,
uniformly in
$(g, j) \in 
\{1, \dots, G\}
\times  
\{1, \dots, m \}
$,
there exists a sufficiently large $M$ such that 
$|\hat{k}_{gj} - T_{j}^{0}| \le M v_{T}^{-2}$
and 
$|\tilde{k}_{j} - T_{j}^{0}| \le M v_{T}^{-2}$
with probability approaching 1.
This implies that we can restrict our analysis to an interval 
centered at the true break $T_{j}^{0}$ with 
length $2 M v_{T}^{-2}$
for each regime $j \in \{1,\dots ,m\}$.
More precisely, 
given a sufficiently large $M$, 
we have that 
$\theta_{t, \hat{\mathcal{K}}}^{0}
 =
 \theta_{t, \mathcal{T}^{0}}^{0}
$
and
$\theta_{t, \tilde{\mathcal{K}}}^{0}
 =
 \theta_{t, \mathcal{T}^{0}}^{0}
$
for all $t \not \in \cup_{j=1}^{m}[T_{j}^{0}-Mv_{T}^{-2}, T_{j}^{0}+Mv_{T}^{-2}]$,
with probability approaching 1.
This follows since the break dates estimates are asymptotically in neighborhoods
of the true break dates;  
hence that there are some miss-classification of regimes around the neighborhoods,
while the regimes are correctly classified outside of the neighborhoods. 
This together with Theorem \ref{theorem:parts} yields that,
under the null hypothesis specified by (\ref{eq:H0}),
\begin{eqnarray*} 
  CB_{T}&=& 
  2 \max_{\mathcal{K} \in \bar\Xi_{M}} 
  \sum_{j=1}^{m}\sum_{\underline{k}_{j}+1}^{\bar{k}_{j}}
  \left\{ 
    \log f(y_{t}|X_{tT},\theta_{t,\mathcal{K}}^{0})
    -
    \log f(y_{t}|X_{tT},\theta_{t,\mathcal{T}^{0}}^{0})
  \right\} \\
  &&- 
  2 \max_{\mathcal{K} \in \bar\Xi_{M, H_{0}}} 
  \sum_{j=1}^{m}\sum_{\underline{k}%
    _{j}+1}^{\overline{k}_{j}}
  \left\{ 
    \log f(y_{t}|X_{tT},\theta_{t,\mathcal{K}}^{0})
    -
    \log f(y_{t}|X_{tT},\theta_{t,\mathcal{T}^{0}}^{0})
  \right\} 
  + o_{p}(1),
\end{eqnarray*} 
where 
$\overline{k}_{j}:=\max \{k_{1j},\dots, k_{Gj}, T_{j}^{0}\}$,
$\underline{k}_{j}:=\min \{k_{1j},\dots, k_{Gj}, T_{j}^{0}\}$,
and 
$\Xi_{M, H_{0}} = \Xi_{M} \cap \Xi_{\eta, H_{0}}$.
Under the null hypothesis, the true break dates $T_{1}^{0}, \dots, T_{m}^{0}$  
are separated by some positive fraction of the sample size
and 
we can obtain the limit distribution of the common break test 
by separately analysing terms of the test for each neighborhood of the true break date. 
We consider a shrinking framework under which the break date estimates 
$\hat{k}_{gj}$ and $\tilde{k}_{j}$ diverge to $\infty$ as $v_{T}$ decreases
and thus 
an application of a Functional Central Limit Theorem
for each neighborhood yields a limit distribution of the test 
which does not depend on the exact distributions. 
To derive the limit distribution, we make the following additional assumptions.

%%%%%%%%%%%%%%%%%%%%%%%%%%%%%%%%%%%%%%%%%%%%%%%%%%%%%%%%%%%%%%%%%%%%%%%%%%%%
\vspace{0.1cm}
\noindent 
\textbf{Assumptions:}
\begin{description}
  \item[A6.]  
    The matrix 
    $(\Delta T_{j}^{0})^{-1} \sum_{t = T_{j-1}^{0} + 1}^{T_{j}^{0}}x_{tT}x_{tT}$
    converges to a (possibly) random matrix not necessarily the same 
    for all $j = 1, \dots, m+1$,
    as $\Delta T_{j}^{0}:=(T_{j}^{0} - T_{j-1}^{0})\rightarrow \infty $.
    Also, 
    $
      ( \Delta T_{j}^{0} )^{-1}
      \sum_{t=T_{j-1}^{0}+1}^{T_{j-1}^{0}+[s\Delta T_{j}^{0}]}
      z_{t}
      \overset{p}{\rightarrow}
      s\mu_{z,j} 
    $
    and 
    $
      ( \Delta T_{j}^{0} )^{-1}
      \sum_{t=T_{j-1}^{0}+1}^{T_{j-1}^{0}+[s\Delta T_{j}^{0}]}
      z_{t}z_{t}^{\prime}
      \overset{p}{\rightarrow}
      sQ_{zz,j} 
    $
    uniformly in $s\in [0,1]$
    as $\Delta T_{j}^{0}\rightarrow \infty$,
    where 
    $Q_{zz,j}$ 
    is a non-random positive definite matrix.
 
  \item[A7.] 
    Define 
    $S_{k,j}(l):=\sum_{T_{j-1}^{0}+l +1}^{T_{j-1}^{0}+l +k}(\zeta _{t}\otimes \eta_{t})$ 
    for $k, l \in \mathbb{N}$
    and 
    for $j=1,...,m+1$.
    (i) If 
    $\{\zeta _{t} \otimes \eta_{t}\}_{t \in \mathbb{Z}}$ is weakly stationary within each segment, 
    then, 
    for any vector $e \in \mathbb{R}^{(q_{z}+q_{w})n}$ with $\|e\|=1$,
    $\mathrm{var}\big(  e'S_{k,j}(0) \big) \geq v(k)$ 
    for some function $v(k) \rightarrow \infty$ as 
    $k\rightarrow \infty$. 
    (ii) If $\{\zeta _{t} \otimes \eta_{t}\}_{t\in \mathbb{Z}}$ 
    is not weakly stationary within each
    segment, we additionally assume that there is a
    positive definite matrix $\Omega =\left[ w_{i,s}\right] $ such that for any 
    $i,s \in \{1,...,p\}$, we have, uniformly in $\ell ,$ 
    $\big |
      k^{-1}
      E
      \big [
       \big (S_{k,j}(\ell ) \big ) _{i}
       \big ( S_{k,j}(\ell) \big ) _{s}
      \big] 
      -
      w_{i,s}
    \big | 
    \leq k^{-\psi }
    $, 
    for some $C>0$
    and 
    for some $\psi >0$.
    We also assume the same conditions for $\{\eta_{t}\eta_{t}^{\prime }-I_{n}\}_{t \in \mathbb{Z}}$. 

    \item[A8.] 
      Let 
      $V_{T, w}(r):=T^{-1/2}\sum\nolimits_{t=1}^{\left[ Tr\right]} u_{wt}$
      for 
      $r\in [0,1]$.
      $
      V_{T, w}(\cdot)
      \Rightarrow 
      \mathbb{V}_{w}(\cdot)$,
      where 
      $\mathbb{V}_{w}(\cdot)$
      is 
      a Wiener processes
      having a covariance function 
      $\mathrm{cov}(
       \mathbb{V}_{w}(r),
       \mathbb{V}_{w}(s)
       )
       =
       (r \wedge s)
       \Omega _{w}
      $
      for $r, s \in [0,1]$
      with 
      a positive definite matrix 
      $\Omega _{w}:= \lim_{T \to \infty} 
       \mathrm{var}
       \big (  
       T^{-1/2}\sum\nolimits_{t=1}^{T} u_{wt}
       \big)
      $.

  \item [A9.] 
    For all $1 \le s, t \le T$, 
    (a) $E[(z_{t} \otimes \eta_{t}) w_{s}^{\prime }]=0$, 
    (b) $E[(z_{t} \otimes \eta_{t}) \mathrm{vec}(\eta _{s}\eta_{s}')']=0$, 
    and 
    (c) $E[(u_{zt} \otimes \eta_{t}) \mathrm{vec}(\eta _{s}\eta_{s}')']=0$.
\end{description}
\vspace{0.1cm}

%% remarks 
Assumption A6 rules out trending variables in the stationary regressors $z_{t}$.
Assumption A7 is mild in the sense that the conditions allow for
substantial conditional heteroskedasticity and autocorrelation.
It can be shown to apply to a
large class of linear processes including those generated by all stationary
and invertible ARMA models.
This assumption is useful to describe the asymptotic behavior of the test
and 
in particular to characterize the limit distribution. 
Here, we introduce some processes used later.
For each $j=1, \dots, m$, 
let 
$\mathbb{V}_{z \eta ,j}^{(1)}(\cdot)$
and 
$\mathbb{V}_{z \eta ,j}^{(2)}(\cdot)$
be Brownian motions defined on the space $D[0,\infty)^{nq}$
with 
zero means
and 
covariance functions given by, 
for $l = 1, 2$
and 
for $s_{1},s_{2}>0$,
\begin{eqnarray*}
  E
  \big [
  \mathbb{V}_{z \eta ,j}^{(l)} (s_{1}) 
  \mathbb{V}_{z \eta ,j}^{(l)} (s_{2})' 
  \big ]
  = (s_{1}\wedge s_{2})
  \lim_{T \to \infty}
  \mathrm{var}
  \big (
  \bar{V}_{T, z\eta,j}^{(l)}
  \big ),
\end{eqnarray*}
where 
$
  \bar{V}_{T, z\eta,j}^{(1)}
  :=
  (\Delta T_{j}^{0})^{-1/2}
  \sum_{t=T_{j-1}^{0}+1}^{T_{j}^{0} }
  (z_{t}\otimes \eta_{t})
$
and 
$ 
  \bar{V}_{T, z\eta,j}^{(2)}
  :=
  (\Delta T_{j+1}^{0})^{-1/2}
  \sum_{t=T_{j}^{0}+1}^{T_{j+1}^{0}}
  (z_{t}\otimes \eta_{t})
$.
Similarly, define 
$\mathbb{V}_{\eta\eta ,j}^{(1)}(\cdot)$
and 
$\mathbb{V}_{\eta\eta ,j}^{(2)}(\cdot)$
as Brownian motions defined on the space $D[0,\infty)^{n^2}$
with 
zero means
and 
covariance functions given by, 
for $l = 1, 2$
and for $s_{1}, s_{2} >0$,
\begin{eqnarray*}
  E
  \big [
  \mathrm{vec}
  \big (
  \mathbb{V}_{\eta \eta ,j}^{(l)} (s_{1}) 
  \big )
  \mathrm{vec}
  \big (
  \mathbb{V}_{\eta \eta ,j}^{(l)} (s_{2})
  \big )' 
  \big ] 
  = 
  (s_{1}\wedge s_{2}) 
  \lim_{T \to \infty}
  \mathrm{var}
  \big \{
  \mathrm{vec} 
  \big (
  \bar{V}_{T, \eta\eta,j}^{(l)}
  \big )
  \big \},
\end{eqnarray*}
where 
$
  \bar{V}_{T, \eta\eta,j}^{(1)}
  :=
  (\Delta T_{j}^{0})^{-1/2}
  \sum_{t=T_{j-1}^{0}+1}^{T_{j}^{0}}
  (\eta_{t}\eta_{t}' {-} I_{n})
$
and
$
  \bar{V}_{T, \eta\eta,j}^{(2)}
  :=
  (\Delta T_{j+1}^{0})^{-1/2}
  \sum_{t=T_{j}^{0}+1}^{T_{j+1}^{0}}
  (\eta_{t}\eta_{t}' {-} I_{n}) 
$.
%% Two-sided Brownian motion 
We define the following two-sided Brownian motions
\begin{eqnarray*}
  \mathbb{V}_{z\eta ,j}(s)  
  :=
  \left \{
  \begin{array}{ll}
    \mathbb{V}_{z\eta,j}^{(1)}(-s), & s \le 0 \\     
    \mathbb{V}_{z\eta ,j}^{(2)}(s), & s>0 
  \end{array}
  \right.
  \ \ \ \mathrm{and} \ \ \
  \mathbb{V}_{\eta\eta ,j}(s)  
  :=
  \left \{
  \begin{array}{ll}
    \mathbb{V}_{\eta\eta,j}^{(1)}(-s), & s \le 0 \\     
    \mathbb{V}_{\eta \eta ,j}^{(2)}(s), & s>0 .
  \end{array}
  \right .
\end{eqnarray*}
Under Assumption A2, $z_{t}$ is assumed to include a constant term 
and the process 
$\mathbb{V}_{z\eta ,j}^{(l)} (\cdot)$ 
includes some process depending purely on $\{\eta_{t}\}$.
We denote it by $\mathbb{V}_{\eta ,j}^{(l)} (\cdot)$ for each $l=1,2$ 
and also define 
a two-sided Brownian motion, denoted by $\mathbb{V}_{\eta,j}(\cdot)$,
as before.

%% Assumption A8
Assumption A8 requires the integrated regressors to follow 
a homogeneous distribution throughout the sample.  
Allowing for heterogeneity in the distribution of the errors underlying the 
$I(1)$ regressors would be considerably more difficult, since 
we would, instead of having the limit distribution in terms of standard Wiener processes,
have time-deformed Wiener processes according to the variance profile of the
errors through time; see, e.g., \cite{CT2007}.
This would lead
to important complications given that, as shown below, the limit
distribution of the estimates of the break dates depends on the whole time
profile of the limit Wiener processes. 
It is possible to allow for trends in the $I(1)$ regressors. 
The limiting distributions of the test to be
derived will remain valid under different Wiener processes 
\citep[see][]{Hansen1992JBES}.
The positive definiteness of the matrix 
$\Omega _{w}$
rules out cointegration among the $I(1)$ regressors and
is needed to ensure a set of regressors that has a positive definite limit.

%% Assumption A9
Assumption A9 is quiet mild and is sufficient but not necessary
to obtain a manageable limit distribution of the test.
It requires the independence of most Wiener processes described above. 
Condition (a) ensures that 
the autocovariance structure of the $I(0)$ regressors and the errors
are uncorrelated with the $I(1)$ variables.
This guarantees that 
$\mathbb{V}_{z\eta,j}(\cdot)$ and $\mathbb{V}_{w,j}(\cdot )$
are uncorrelated and thus independent because of Gaussianity.
Without these conditions, the analysis would be much more complex. 
Similarly,
the conditions (b) and (c) imply the independence 
between
$\mathbb{V}_{z\eta,j}(\cdot )$ and $\mathbb{V}_{\eta \eta}(\cdot )$.
See \cite{KejriwalPerron2008JoE} for more details.

In order to characterize the limit distribution of $%
CB_{T}$ it is useful to first state some preliminary results about the limit
distribution of some quantities.
For $s \in \mathbb{R}$ and for $j = 1, \dots, m$,
let $\overline{T}_{j}(s):=\max \{T_{j}(s),T_{j}^{0}\}$ and 
$\underline{T}_{j}(s):=\min\{T_{j}(s),T_{j}^{0}\}$
where 
$T_{j}(s):= T_{j}^{0} + [s v_{T}^{-2}]$.
For $s, r \in \mathbb{R}$,
we define 
$
  B_{T,j}(s, r) 
  {:=} 
  v_{T}^{2}
  \sum_{t= \underline{T}_{j}^{0}(s)+ 1}^{  \overline{T}_{j}(s)}
  X_{tT}
  (\Sigma_{j + \mathbbm{1}_{ \{T_{j}(r) < t\} }}^{0} )^{-1} 
  X_{tT}'
$
and
$
  W_{T,j}(s, r)
  {:=} 
  v_{T}
  \sum_{t= \underline{T}_{j}^{0}(s)+ 1}^{  \overline{T}_{j}(s)}
  X_{tT}
  (\Sigma_{j + \mathbbm{1}_{ \{T_{j}(r) < t\} }}^{0} )^{-1} 
  u_{t}
$
for $j \in \{1, \dots, m\}$.

%%%%%%%%%%%%%%%%%%%%%%%%%%%%%%%%%%%%%%%%%%%%%%%%%%%%%%%%%%%%%%%%%%%%%%%%%%%%%
\begin{lemma}
  Suppose that Assumptions A1-A9 hold. Then,  
  \begin{equation*}
    \big \{
    B_{T,j}(\cdot, \cdot),
    W_{T,j}(\cdot, \cdot)
    \big \}_{j=1}^{m}
    \Rightarrow
    \big \{
    \mathbb{B}_{j}(\cdot, \cdot),
    \mathbb{W}_{j}(\cdot, \cdot)
    \big \}_{j=1}^{m},
  \end{equation*}
  where 
  \begin{eqnarray*}
    \mathbb{B}_{j}(s,r)
    :=
    |s|
    S' 
    \mathbb{D}_{j}(s)   
    \otimes 
    (\Sigma_{j+\mathbbm{1}_{ \{r \le s \}}}^{0})^{-1} 
    S
    -
    \mathbbm{1}_{ \{ |r| \le |s|\} }
    |r|
    S'
    \mathbb{D}_{j}(s)
    \otimes
    \{
    (\Sigma_{j+1}^{0})^{-1}
    -
    (\Sigma_{j}^{0})^{-1}
    \}
    S, 
  \end{eqnarray*}
  and 
  \begin{eqnarray*}
    \mathbb{W}_{j}(s,r)
    :=
    S'
    \big (
    I_{q} \otimes (\Sigma_{j+\mathbbm{1}_{ \{r \le s \}}}^{0})^{-1}
    \big )
    \mathbb{V}_{j}(s)   
     -
    \mathrm{sgn}(r)
    \mathbbm{1}_{ \{ |r| \le |s|\} }
    S'
    \big [
    I_{q} \otimes
     \{
     (\Sigma_{j+1}^{0})^{-1}
     -
     (\Sigma_{j}^{0})^{-1}
     \}
    \big ]
    \mathbb{V}_{j}(r),
  \end{eqnarray*}
  % \begin{eqnarray*}
  %   \mathbb{B}_{j}(s,r)
  %   &:=&
  %   |s|
  %   S' 
  %   \mathbb{D}_{j}(s)   
  %   \otimes 
  %   (\Sigma_{j+\mathbbm{1}_{ \{r \le s \}}}^{0})^{-1} 
  %   S
  %   -
  %   \mathbbm{1}_{ \{ |r| \le |s|\} }
  %   |r|
  %   S'
  %   \mathbb{D}_{j}(s)
  %   \otimes
  %   \{
  %   (\Sigma_{j+1}^{0})^{-1}
  %   -
  %   (\Sigma_{j}^{0})^{-1}
  %   \}
  %   S, \\
  %   \mathbb{W}_{j}(s,r)
  %   &:=&
  %   S'
  %   \big (
  %   I_{q} \otimes (\Sigma_{j+\mathbbm{1}_{ \{r \le s \}}}^{0})^{-1}
  %   \big )
  %   \mathbb{V}_{j}(s)   
  %    -
  %   \mathrm{sgn}(r)
  %   \mathbbm{1}_{ \{ |r| \le |s|\} }
  %   S'
  %   \big [
  %   I_{q} \otimes
  %    \{
  %    (\Sigma_{j+1}^{0})^{-1}
  %    -
  %    (\Sigma_{j}^{0})^{-1}
  %    \}
  %   \big ]
  %   \mathbb{V}_{j}(r),
  % \end{eqnarray*}
  with 
  $
    \mathbb{V}_{j}(s):=
    \big (
    I_{q} \otimes (\Sigma_{j+\mathbbm{1}_{ \{0 \le s \}}}^{0})^{1/2}
    \big )
    \big [ 
      \mathbb{V}_{z\eta,j}(s)',
      \varphi(\lambda _{j}^{0})' \otimes \mathbb{V}_{\eta ,j}(s)',
      \mathbb{V}_{w}(\lambda _{j}^{0})' \otimes \mathbb{V}_{\eta ,j}(s)'
    \big]'
  $
  and 
  \begin{equation*}
    \mathbb{D}_{j}(s)
    :=
    \left( 
      \begin{array}{ccc}
        Q_{zz,j+\mathbbm{1}_{\{0<s\}}} &  
        \mu_{z,j+\mathbbm{1}_{\{0<s\}}} \varphi(\lambda_{j}^{0})' &  
        \mu _{z,j+\mathbbm{1}_{\{0<s\}}}\mathbb{V}_{w}(\lambda_{j}^{0} )^{\prime }\\ 
        \varphi(\lambda_{j}^{0})\mu _{z,j+\mathbbm{1}_{\{0<s\}}}^{\prime}
        & 
        \varphi(\lambda_{j}^{0}) \varphi(\lambda_{j}^{0})'
        & 
        \varphi(\lambda_{j}^{0}) \mathbb{V}_{w}(\lambda_{j}^{0})^{\prime } 
        \\
        \mathbb{V}_{w}(\lambda_{j}^{0}) \mu _{z,j+\mathbbm{1}_{\{0<s\}}}' 
        & 
        \mathbb{V}_{w}(\lambda_{j}^{0} )
        \varphi(\lambda_{j}^{0})'
        &
        \mathbb{V}_{w}(\lambda_{j}^{0} )\mathbb{V}_{w}(\lambda_{j}^{0})^{\prime }
        \\ 
      \end{array}%
    \right).
  \end{equation*}
\end{lemma}
%%%%%%%%%%%%%%%%%%%%%%%%%%%%%%%%%%%%%%%%%%%%%%%%%%%%%%%%%%%%%%%%%%%%%%%%%%%%%

The theorem below presents the main result of the paper concerning the limit distribution of the test statistic, which can be expressed 
as the difference of the maxima of a limit process with and without restrictions implied by
the assumption of common breaks.

%%%%%%%%%%%%%%%%%%%%%%%%%%%%%%%%%%%%%%%%%%%%%%%%%%%%%%%%%%%%%%%%%%%%%%%%%%%%%
\begin{theorem}
  \label{theorem:limit-dist}
  Let $\bm{s}_{j}=(s_{1j}, \dots, s_{Gj})'$
  for $j=1,\dots, m$
  and 
  let $\bm{1}$ be a $G \times 1$ vector having 1 at all entries.
  Suppose Assumptions A1-A9 hold. Then, under the null hypothesis
  (\ref{eq:H0}),
  \begin{equation*} 
    CB_{T}
    \Rightarrow
    CB_{\infty}
    :=
    \sup_{\bm{s}_{1}, \dots, \bm{s}_{m} }
    \sum_{j=1}^{m}
    CB_{\infty}^{(j)}( \bm{s}_{j})
    -
    \sup_{s_{1}, \dots, s_{m}}
    \sum_{j=1}^{m}
    CB_{\infty }^{(j)}(s_{j} \cdot \bm{1}),
  \end{equation*}%
  where 
  \begin{eqnarray}
    CB_{\infty}^{(j)}( \bm{s}_{j})
    &:=&
    \mathrm{tr} \notag
    \Big (
    \Pi_{j}(s_{Gj})
    \mathbb{V}_{\eta\eta,j}(s_{G})
    \Big )
    +
    \frac{|s_{Gj}|}{2}
    \mathrm{tr}
    \big (
     \{\Pi_{j}(s_{Gj})\}^2
    \big )  
    - 2
    \sum_{g=1}^{G}
    \mathrm{sgn}(s_{gj})
    \Delta_{gj}' \mathbb{W}_{j}(s_{gj}, s_{Gj})
    \\
    && \notag
    -
    \sum_{g=1}^{G}
    \sum_{h=1}^{G}
    \Delta_{gj}' 
    \Big \{
    \mathbbm{1}_{ \{s_{gj} \vee s_{hj} \le  0\}}
    \mathbb{B}_{j}
    \big ( s_{gj}{\vee}s_{hj}, s_{Gj} \big )
    +
    \mathbbm{1}_{ \{0 <  s_{gj} \wedge s_{hg} \}}
    \mathbb{B}_{j}
    \big ( s_{gj}{\wedge}s_{hj}, s_{Gj} \big )
    \Big \}
    \Delta_{hj}, \\
    \label{eq:pai-def}
    \Pi_{j}(s_{Gj}) 
    &:=&
    \left \{
    \begin{array}{cl}
      \big (\Sigma _{j}^{0} \big)^{-1/2}
      \Upsilon _{j}
      \big (\Sigma _{j+1}^{0} \big)^{-1}
      \big (\Sigma_{j}^{0} \big)^{1/2},
      & \ \mathrm{if} \ s_{Gj} \leq 0\\
      -
      \big (\Sigma _{j+1}^{0} \big)^{-1/2}
      \Upsilon _{j}(\Sigma _{j}^{0})^{-1}
      \big (\Sigma_{j+1}^{0} \big)^{1/2},
      &
      \ \mathrm{if} \ s_{Gj}>0 
    \end{array}
    \right .,
  \end{eqnarray}
  with 
  $\Delta_{gj}:= 
   \big ( \|\delta_{j}\|^2+ \mathrm{tr}(\Phi_{j}^2) \big )^{-1/2} 
   \delta_{gj}
  $
  and 
  $
  \Upsilon_{j}
  :=
  \big ( \|\delta_{j}\|^2+ \mathrm{tr}(\Phi_{j}^2) \big )^{-1/2} 
  \Phi_{j}
  $.
\end{theorem}
%%%%%%%%%%%%%%%%%%%%%%%%%%%%%%%%%%%%%%%%%%%%%%%%%%%%%%%%%%%%%%%%%%%%%%%%%%%%%

The limit distribution in Theorem \ref{theorem:limit-dist}
is quite complex and depends on nuisance parameters.
However, they can be consistently estimated and it is easy to show that the
coverage rates will be asymptotically valid provided $\sqrt{T}$-consistent
estimates are used instead of the true values. The various quantities can be
estimated as follows: 
for
$\Delta \tilde{k}_{j}:= \tilde{k}_{j}-\tilde{k}_{j-1}$,
we can use
$\tilde{Q}_{zz,j}
=(\Delta \tilde{k}_{j})^{-1}\sum_{t=\tilde{k}_{j-1} + 1}^{\tilde{k}_{j}}z_{t}z_{t}^{\prime }
$, 
$\tilde{\mu}_{z,j}
 = (\Delta \tilde{k}_{j})^{-1}\sum_{t=\tilde{k}_{j-1}+1}^{\tilde{k}_{j}}z_{t}$, 
$\Delta \tilde{\beta}_{j}:= \tilde{\beta}_{j}-\tilde{\beta}_{j-1}$
and 
$\tilde{\Sigma}_{j}
=
(\Delta \tilde{k}_{j})^{-1}\sum_{t=\tilde{k}_{j-1} + 1}^{\tilde{k}_{j}}\tilde{u}_{t}\tilde{u}_{t}^{\prime }
$,  
$\tilde{\Delta}_{gj}:= 
\big \{ \|\Delta \tilde{\beta}_{j}\|^2 + \mathrm{tr}\big ( ( \Delta \tilde{\Sigma}_{j})^2\big)
\big \}^{-1/2}
\sum_{l \in \mathcal{G}_{g}}
e_{l}\circ \Delta \tilde{\beta}_{j+1}
$ 
and 
$\tilde{\Upsilon}_{j}:= 
\big \{ \|\Delta \tilde{\beta}_{j}\|^2 + \mathrm{tr}\big ( ( \Delta \tilde{\Sigma}_{j})^2\big)
\big \}^{-1/2}
\Delta \tilde{\Sigma}_{j}
$, 
where 
$\Delta \tilde{\beta}_{j}:= \tilde{\beta}_{j}-\tilde{\beta}_{j-1}$
and
$\Delta \tilde{\Sigma}_{j}:= \tilde{\Sigma}_{j}-\tilde{\Sigma}_{j-1}$.
Also, the estimates of the long run variances of $\{z_{t}\otimes \eta_{t}\}$
and $\{\eta_{t}\eta_{t} - I_{n}\}$ 
can be constructed using 
a method based on a weighted sum of sample autocovariances of the relevant
quantities, as discussed in \cite{Andrews1991Emtca}, for instance. Though only 
$\sqrt{T}$-consistent estimates of $(\beta ,\Sigma )$ are needed, it is likely that
more precise estimates of these parameters will lead to better finite sample
coverage rates. Hence, it is recommended to use the estimates obtained
imposing the restrictions in (\ref{restrictions}) even though imposing restrictions does not have a
first-order effect on the limiting distribution of the estimates of the
break dates.

In some cases, the limit distribution of the common breaks test can be derived and expressed in a simpler manner. 
For illustration purpose, our supplemental material states the limit distribution of the test under the setup of Examples 1 and 2. 
When the covariance matrix is constant over time
(i.e., $\Sigma_{j}^{0} = \Sigma^{0}$ for $j=1, \dots, m+1$), the limit distribution above can be further simplified as stated in the following corollary.

%%%%%%%%%%%%%%%%%%%%%%%%%%%%%%%%%%%%%%%%%%%%%%%%%%%%%%%%%%%%%%%%%%%%%%%%%%%%%
\begin{corollary}
  \label{corollary:limit-dist}
  Let $\bm{s}_{j}=(s_{1j}, \dots, s_{Gj})'$
  for $j=1,\dots, m$
  and 
  let $\bm{1}$ be a $G \times 1$ vector having 1 at all entries.
  Suppose that Assumptions A1-A9 hold and also that 
  the covariance matrix $\Sigma_{j}^{0}$ is constant over time. 
  Then, under the null hypothesis
  (\ref{eq:H0}),
  \begin{equation*} 
    CB_{T}
    \Rightarrow
    \tilde{CB}_{\infty}
    :=
    \sup_{\bm{s}_{1}, \dots, \bm{s}_{m} }
    \sum_{j=1}^{m}
    \tilde{CB}_{\infty}^{(j)}( \bm{s}_{j})
    -
    \sup_{s_{1}, \dots, s_{m}}
    \sum_{j=1}^{m}
    \tilde{CB}_{\infty }^{(j)}(s_{j} \cdot \bm{1}),
  \end{equation*}%
  where 
  \begin{eqnarray*}
    \tilde{CB}_{\infty}^{(j)}( \bm{s}_{j})
    &:=& 
    - 2
    \sum_{g=1}^{G}
    \mathrm{sgn}(s_{gj})
    \Delta_{gj}' \tilde{\mathbb{W}}_{j}(s_{gj})
    \\
    && 
    -
    \sum_{g=1}^{G}
    \sum_{h=1}^{G}
    \Delta_{gj}' 
    \Big \{
    \mathbbm{1}_{ \{s_{gj} \vee s_{hj} \le  0\}}
    \tilde{\mathbb{B}}_{j}
    \big ( s_{gj}{\vee}s_{hj} \big )
    +
    \mathbbm{1}_{ \{0 <  s_{gj} \wedge s_{hg} \}}
    \tilde{\mathbb{B}}_{j}
    \big ( s_{gj}{\wedge}s_{hj} \big )
    \Big \}
    \Delta_{hj}, 
  \end{eqnarray*}
  with 
  $
  \tilde{\mathbb{W}}_{j}(s)
  :=
  S'
  \big (
  I_{q} \otimes (\Sigma^{0})^{-1/2}
  \big )
  \big [ 
  \mathbb{V}_{z\eta,j}(s)',
  \varphi(\lambda _{j}^{0})' \otimes \mathbb{V}_{\eta ,j}(s)',
  \mathbb{V}_{w}(\lambda _{j}^{0})' \otimes \mathbb{V}_{\eta ,j}(s)'
  \big]'
  $
  and 
  $
    \tilde{\mathbb{B}}_{j}(s)
    :=
    |s|
    S' 
    \mathbb{D}_{j}(s)   
    \otimes 
    (\Sigma^{0})^{-1} 
    S
  $
  for $s \in \mathbb{R}$.
\end{corollary}
%%%%%%%%%%%%%%%%%%%%%%%%%%%%%%%%%%%%%%%%%%%%%%%%%%%%%%%%%%%%%%%%%%%%%%%%%%%%%

As another immediate corollary to Theorem \ref{theorem:limit-dist}, when no integrated variables are present,  the limit distribution of the test for a common break date only involves the pre and post break date regimes, as is the case for the limit distribution of the estimates when multiple breaks are present \citep[e.g.][]{BaiPerron1998Emtca}.  
Also, the above result can be easily extended to test the hypothesis of
common break dates for a part of the parameter groups, 
while the break dates of the other groups are not necessarily common.
We illustrate the application of the test for common breaks in (\ref{eq:H0})
and its variant through an application in Section 5.

%% computational issue 
As discussed in Section 1,
there is one additional layer of difficulty compared to Bai and
Perron (1998) or Qu and Perron (2007). 
In their analysis, 
the limit distribution can be evaluated using a closed form solution
after some transformation,  
while no such solution is available here
and thus we need to resort simulations to obtain the critical values. 
This involves first simulating the Wiener processes   
appearing in the various Brownian motion processes by partial sums of 
$i.i.d.$ normal random vectors (independent of each others given Assumption
A9). One can then evaluate one realization of the limit distribution by
replacing unknown values by their estimates as stated above. The procedure
is then repeated many times to obtain the relevant quantiles. 
While conceptually straightforward, this procedure is nevertheless 
computationally intensive. The reason is that for each replication
we need to search over many possible combinations of all the
permutations of the locations of the break dates. The procedure suggested is
nevertheless quick enough to be feasible for common applications involving
testing for few common break dates but the computational burden increases
exponentially with the number of common breaks being tested. 
In Section 4, we propose an alternative approach to alleviate this
issue and examine its performance. 
        
%%%%%%%%%%%%%%%%%%%%%%%%%%%%%%%%%%%%%%%%%%%%%%%%%%%%%%%%%%%%%%%%%%%%%%%%%%%%%%%%%%%%%%
%%%%%%%%%%%%%%%%%%%%%%%%%%%%%%%%%%%%%%%%%%%%%%%%%%%%%%%%%%%%%%%%%%%%%%%%%%%%%%%%%%%%%%
\subsection{Asymptotic power analysis}

In this subsection, we provide an asymptotic power analysis of the test
statistic $CB_{T}$ when using a critical value $c_{\alpha }^{\ast}$
at the significance level $\alpha$
from the asymptotic null distribution $CB_{\infty}$.
As a fixed alternative hypothesis,
we consider, for some $\delta>0$ 
\begin{equation}
  \label{eq:alt-fix}
  H_{1}:
  \max_{1 \le g_{1}, g_{2} \le G}
  |
  k_{g_{1},j}^{0}
  -
  k_{g_{2},j}^{0}
  |
  \ge \delta T
  \mathrm{\ \ for \ some \ }
  j = 1, \dots, m.
\end{equation}
Given that 
$k_{gj}^{0} = [T\lambda_{gj}^{0}]$
for $(g, j) \in \{1, \dots, G\} {\times} \{1,...,m\}$
under Assumption A3, 
the above condition is asymptotically equivalent to 
$
  \max_{1 \le g_{1}, g_{2} \le G}
  |\lambda_{g_{1},j}^{0} - \lambda_{g_{2},j}^{0}|
  \ge \delta
$
for some 
$j = 1, \dots, m$,
and 
thus can be considered as a fixed alternative hypothesis
in term of break fractions.
As a local alternative hypothesis, we consider 
\begin{equation}
  \label{eq:alt-local}
  H_{1T}:
  \max_{1 \le g_{1}, g_{2} \le G}
  |
  k_{g_{1},j}^{0}
  -
  k_{g_{2},j}^{0}
  |
  \ge M v_{T}^{-2}
  \mathrm{\ \ for \ some \ }
  j = 1, \dots, m,
\end{equation}%
for some constant $M>0$, where $v_{T}$
satisfies the condition in Assumption A4.
We can also express (\ref{eq:alt-local}) as
$
  \max_{1 \le g_{1}, g_{2} \le G}
  |
  \lambda_{g_{1},j}^{0}
  -
  \lambda_{g_{2},j}^{0}
  |
  \ge M (\sqrt{T}v_{T})^{-2}
$
for some 
$j = 1, \dots, m$.
The following theorem shows that the proposed test statistic 
is consistent against fixed alternatives 
and 
also has non-trivial local power against local
alternatives. 
   
%%%%%%%%%%%%%%%%%%%%%%%%%%%%%%%%%%%%%%%%%%%%%%%%%%%%%%%%%%%%%%%%%%%%%%%%%%%%%%%%%
\begin{theorem} 
  \label{theorem:alternative}
  Let 
  $c_{\alpha }^{\ast}:=\inf \big \{c \in \mathbb{R}:\Pr\{CB_{\infty} \leq c \}\geq 1-\alpha \big \}$.
  Suppose that Assumptions A1-A9 hold. Then, 
  (a)
  under the fixed alternative (\ref{eq:alt-fix})
  with any $\delta \in (0,1]$,
  \begin{equation*}
    \lim_{T\rightarrow \infty }
    \Pr
    \big \{ 
    CB_{T}
    >
    c_{\alpha }^{\ast}
    \big \}
    = 1, 
  \end{equation*}%
  (b) under the local alternative (\ref{eq:alt-local}),
  for any $\epsilon>0$, there exits an $M$ 
  defined in
  (\ref{eq:alt-local})
  such that 
  \begin{equation*}
    \lim_{T\rightarrow \infty }
    \Pr 
    \big \{
    CB_{T}
    >
    c_{\alpha }^{\ast}
    \big \}
    > 1 - \epsilon. 
  \end{equation*}
\end{theorem}
%%%%%%%%%%%%%%%%%%%%%%%%%%%%%%%%%%%%%%%%%%%%%%%%%%%%%%%%%%%%%%%%%%%%%%%%%%%%%%%%%

%%%%%%%%%%%%%%%%%%%%%%%%%%%%%%%%%%%%%%%%%%%%%%%%%%%%%%%%%%%%%%%%%%%%%%%%%%%%%%%%%
%%%%%%%%%%%%%%%%%%%%%%%%%%%%%%%%%%%%%%%%%%%%%%%%%%%%%%%%%%%%%%%%%%%%%%%%%%%%%%%%%
\section{Monte Carlo simulations}
 
This section provides simulation results about the finite sample
performance of the test in terms of size and power.
We first consider a direct simulation-based approach to obtain the critical values
and then a more computationally efficient algorithm. 
As a data generating process (DGP),
we adopt a similar setup to the one used in \cite{BaiLumsdaineStock1998}, 
namely a bivariate autoregressive system with a single break in intercepts
as in Example 1. 
Hence, only the intercepts are allowed to change at some dates $k_{i1}$
for equation $i \in \{1,2\}$.
We test the null hypothesis $H_0: k_{11} = k_{21}$
against the alternative hypothesis $H_1: k_{11} \not= k_{21}$. 
The number of
observations is set to $T=100$, and we use $500$ replications.
Results are reported
for autoregressive parameters $\alpha \in \{0.0,0.4, 0.8\}$.
We set $\mu_{i1}=1$ and let $\delta _{i}:=\mu_{i2}-\mu_{i1}$,  
the magnitude of the mean shift, take values
$\{0.50, 0.75, 1.00, 1.25, 1.50\}$.

\textbf{A direct simulation-based approach:}
We first present results when we resort direct simulations to obtain the critical values, which
involves simulating the Wiener processes by partial sums of i.i.d.~normal random vectors
and 
searching over all possible combinations of the break dates. 
Given the computational cost, we 
choose a simple setup and focus on limited cases.
To examine the empirical sizes and power,
we here consider the errors $(u_{1t},u_{2t})^{\prime }$ following $i.i.d.$ $N(0,I_{2})$
and we use 3,000 repetitions to generate the critical values.

%%%% Size 
We first examine the empirical rejection frequencies under the null hypothesis
that $k_{11} = k_{21} = 50$
with a trimming parameter $\nu=0.15$.
The results are reported in Table 1 for nominal sizes of 10\%, 5\% and 1\%.
First, when the autoregressive process has no or moderate dependency 
($\alpha=0.0$ or $\alpha=0.4$), the empirical size of the test 
is either slightly conservative or close to the nominal size. 
Given the small sample size, this size property is satisfactory. 
When the autoregressive parameter is close to the boundary of the non-stationary region,
e.g.~$\alpha =0.8$, as expected there are some liberal size distortions.
When the magnitudes of the breaks are small, the test tends to over-reject the null
hypothesis.
This is due to the fact that for very small breaks the break date estimates are quite imprecise and are more likely to be affected by the highly dependent series than the break sizes themselves, 
so that the test depends on the log likelihoods evaluated outside neighborhoods of the true break dates. 
When the magnitude of the break sizes increases, 
the size of the test quickly approaches the nominal level. These
results are encouraging given the small sample size.
 
%% Power 
To analyze power, 
we also set $\mu_{i1}=1$, while we consider values $\{0.50,  1.00, 1.50\}$ for the magnitude of the mean shift.
The break date in the first equation is kept
fixed at $k_{1}=35$, while the break date in the second equation takes values 
$k_{2}=35,40,45,50,55$. 
The power is a function of the difference between the break dates, $k_{2}-k_{1}$.
The results are presented in Figure 1,
where the horizontal axis in each box represents the difference $k_{2}-k_{1}$
and 
the vertical axis shows the empirical rejection frequency. 
As before, when the magnitudes of the breaks are small, 
the data are 
not informative enough to reject the common breaks null hypothesis
and 
the test has little power.
However, when the
magnitudes of the changes reach 1, the power increases rapidly as the
distance between the break dates increases. The results are qualitatively
similar for all values of $\alpha $ considered.

\textbf{An alternative approach:}
The direct simulation-based procedure involves a combinatorial optimization problem
and 
the computational burden increases exponentially with the number of common breaks being tested.
Such a procedure may be feasible for a small number of breaks
in a parsimonious system. However, in more general cases, it may be 
prohibitive. Hence, we also propose an alternative approach 
that solves this problem, 
using heuristic algorithms that find approximate, if not optimal, solutions. 
Because heuristic algorithms have mainly been developed to optimize functions
having explicit forms,   
we use the Karhunen-Lo{\`e}ve (KL) representation of stochastic processes, 
which expresses a Brownian motion as an infinite sum 
of sine functions with independent Gaussian random multipliers
\citep[see][p.~26, for instance]{Bosq2012}. 
A truncated series of the KL representation was used to obtain critical values 
by \cite{Durbin1970} and \cite{Krivyakov1978}, among others.
Similarly, we use a truncated series with 500 terms 
and apply a change of variables 
to approximately obtain an explicit form of the objects being maximized 
in the limit distribution of the common breaks test.
Also, we use the particle swarm optimization method, 
which is an evolutionary computation algorithm developed by 
\cite{Eberhart1995}.\footnote{
  For our simulations, we use the particle swarm algorithm  ``\textit{particleswarm}''
  of the Matlab Global Optimization Toolbox.
  We also tried the genetic algorithm ``\textit{ga}''
  from Matlab
  and found that the two algorithms yield very similar, frequently the same, critical values,
  while the particle swarm algorithm is faster.
}

%% Setup 
We examine the performance of the common breaks test using the alternative algorithm
under various setups
in order to show that similar good finite sample properties are obtained 
compared to the direct optimization method.
In addition to the setup used above, we 
consider a trimming value $\nu = 0.10$, a pair of break dates (35, 35)
and 
normal errors with correlation coefficient being 0.5
across equations.
Columns (1)-(4) of Table 2 present 
empirical rejection frequencies under the null hypothesis 
for a nominal size of 5\%.
Whether the errors are correlated or not, 
the empirical size of the test is either conservative or close to the nominal size
in cases of moderate dependency ($\alpha=0.0$ or $\alpha=0.4$).
Also the trimming parameter has little impact. 
With uncorrelated errors, 
there are size distortions 
in cases of high dependency ($\alpha =0.8$) 
and small break sizes. 
When the errors are correlated, however, the empirical sizes get closer 
to the nominal level in all cases. 
This is likely due to efficiency gains from using a SUR estimation method. 
Columns (5)-(6) of Table 2
report the empirical power for the case 
$(k_{1}, k_{2}) = (35, 50)$ 
and the results show satisfactory power, comparable to the direct method.

%%%%%%%%%%%%%%%%%%%%%%%%%%%%%%%%%%%%%%%%%%%%%%%%%%%%%%%%%%%%%%%%%%%%%%%%%%%%%%%%%%%%%%
\section{Application}

%% Clark: abstract  
In this section, we apply the common breaks test to inflation
series, following \cite{Clark2006JAE}.
He analyzes the persistence of a number of disaggregated inflation series 
based on the sum of the autoregressive (AR) coefficients in an AR model, 
and 
documents that 
the persistence is very high and close to one without allowing for a mean shift,
whereas 
the persistence declines substantially when allowing for one. 
Although such features have been documented theoretically
in the literature \cite[e.g.][]{Perron1990JBES},
he finds that 
the decline in persistence is more pronounced amongst disaggregated measures compared to various aggregate measures.
The issue of importance is that 
\cite{Clark2006JAE} assumes a common mean shift for all series,
following \cite{BaiLumsdaineStock1998},
but 
the validity of this assumption is not established.

%% Setup
We consider a subset of the series analyzed in \cite{Clark2006JAE}, namely the
inflation measures for durables, nondurables and services.
These are taken from the NIPA accounts and cover the period 1984-2002 at the quarterly
frequency; see \cite{Clark2006JAE} for more details. 
Let $\{(y_{1t}, y_{2t}, y_{3t})\}_{t=1}^{T}$ denote the inflation series of 
durables, nondurables and services
and 
consider an AR model allowing for a mean shift
for each series $i= 1, 2, 3$:
\begin{eqnarray*}
  y_{it} 
  = 
  \mu_{i} + \delta_{i} \mathbbm{1}_{ \{ k_{i}+1 \le t \}}
  + \alpha_{i}^{(1)} y_{i,t-1} + \cdots + \alpha_{i}^{(p_{i})} y_{i,t-p_{i}} + u_{it}, \ \ \ \
  t = 1, \dots, T,
\end{eqnarray*}
where 
$\mu_{i}$ is an intercept parameter,  
$\delta_{i}$ is the magnitude of the mean shift with $k_{i}$ being a break date.
The parameters,
$\alpha_{i}^{(1)}, \dots, \alpha_{i}^{(p_{i})}$,
are AR coefficients with $p_{i}$ denoting the lag length 
and 
$u_{it}$ is an error term.
The persistence of each series is measured by the sum $\alpha_{i}^{(1)} + \cdots + \alpha_{i}^{(p_{i})}$ 
for $i=1,2, 3$. 
\cite{Clark2006JAE} uses the Akaike information criterion (AIC) 
to select the AR lag length such that $(p_{1}, p_{2}, p_{3}) = (4, 5, 3)$
and 
also presents some evidence to support a mean shift in the AR models by applying break tests
for each series and for groups.

%% Estimation Result 
We present our empirical results in Table 3.  
We first replicate a part of the results in \cite{Clark2006JAE}.
We find that 
when not allowing for a mean shift, the persistence measure is indeed
quite high ranging from 0.855 to 0.921.
Also, the persistence measure decreases to a large extent for non-durables and services but not so much for durables 
when a common break is imposed for the intercept 
at the break date 1993:Q1,
which is not estimated but treated as known
in \cite{Clark2006JAE}.
When we use the Seemingly Unrelated Regressions (SUR) method
with an unknown common break date, following \cite{BaiLumsdaineStock1998},
the point estimates are similar expect that the break date is estimated at 1992:Q1.

%% Testing 
We now use our test to assess the validity of the common breaks specification.
In Table 3, we report values of the test statistic for several null hypotheses
as well as critical values corresponding to a 5\% significance level,
obtained through 
the computationally efficient algorithm
described in Section 4
with 3,000 repetitions.
First, we consider the null hypothesis of common breaks in the three inflation series,
i.e., $H_{0}: k_{1} = k_{2} = k_{3}$. The value of the test statistic is 9.015
and the critical value is 5.242, so that the test rejects the null hypothesis of common breaks
at the 5\% significance level. 
Next, we test for common breaks in two inflation series 
within the full system of the three inflation series, separately. 
That is, we separately calculate the test statistic for 
$H_{0}: k_{1} = k_{2}$,
$H_{0}: k_{1} = k_{3}$,
and  
$H_{0}: k_{2} = k_{3}$.
The values of the test statistic 
are 9.735 and 7.684 
with corresponding critical values 
3.473 and 3.259
for 
$H_{0}: k_{1} = k_{2}$
and 
$H_{0}: k_{1} = k_{3}$, 
respectively,
and thus 
both hypotheses are rejected 
at the 5\% significance level.
On the other hand, 
the value of the statistic for $H_{0}: k_{2} = k_{3}$
is 0.749 with a critical value of 2.501. 
Thus, we cannot reject the null hypothesis of common breaks in the nondurables and service series.

%% estimation 
We then estimate a system with the three inflation series imposing a common break only in 
the nondurables and service series (i.e., $k_{2} = k_{3}$),
estimated at 1992:Q1, 
which is the same as when allowing for an unknown common break date in all series
(the parameter estimates are also broadly similar).
Things are quite different for the durables series. In this case, the estimate of the
break date is 1995:Q1. What is interesting is that with this break date the decrease in
persistence is very important with an estimate of 0.324 compared to 0.805 obtained assuming a common break date across the three series. 
Hence, allowing for different break dates for durables and the other 
series, 
we document a substantial decline in the persistence measure across all three series.
Moreover, we report the 95\% confidence intervals for the estimated break dates:
[1994:Q2, 1995:Q4] for durables and [1991:Q3, 1992:Q3] for the others.
These non-overlapping intervals are consistent with our results.

%%%%%%%%%%%%%%%%%%%%%%%%%%%%%%%%%%%%%%%%%%%%%%%%%%%%%%%%%%%%%%%%%%%%%%%%%%%%%%%
\section{Conclusion}

This paper provides a procedure to test for common breaks across or within
equations. Our framework is very general and allows 
integrated regressors and trends  as well as stationary regressors. 
The test considered is the
quasi-likelihood ratio test assuming normal errors, though as usual the
limit distribution of the test remains valid with non-normal errors. Of
independent interest, we provide results about the rate of
convergence when searching over all possible partitions subject only to the
requirement that each regime
contains at least as many observations as some positive fraction of the sample size,
allowing break dates not separated by a positive fraction of the sample size 
across equations.     
We propose two approaches to obtain critical values.
Simulations show that the test
has good finite sample properties. 
We also provide an application to issues related to 
level shifts and persistence for
various measures of inflation to illustrate its usefulness.

%%%%%%%%%%%%%%%%%%%%%%%%%%%%%%%%%%%%%%%%%%%%%%%%%%%%%%%%%%%%%%%%%%%%%%%%%%%%%%%
%%%%%%%%%%%%%%%%%%%%%%%%%%%%%%%%%%%%%%%%%%%%%%%%%%%%%%%%%%%%%%%%%%%%%%%%%%%%%%%
%%%%%%%%%%%%%%%%%%%%%%%%%%%%%%%%%%%%%%%%%%%%%%%%%%%%%%%%%%%%%%%%%%%%%%%%%%%%%%%%%%%%
\clearpage  
\setstretch{0.3}           
\setlength{\bibsep}{8pt}
\bibliographystyle{elsarticle-harv.bst} 
\bibliography{REF.bib}

%%%%%%%%%%%%%%%%%%%%%%%%%%%%%%%%%%%%%%%%%%%%%%%%%%%%%%%%%%%%%%%%%%%%%%%%%%%%%%%
%%%%%%%%%%%%%%%%%%%%%%%%%%%%%%%%%%%%%%%%%%%%%%%%%%%%%%%%%%%%%%%%%%%%%%%%%%%%%%%
%%%%%%%%%%%%%%%%%%%%%%%%%%%%%%%%%%%%%%%%%%%%%%%%%%%%%%%%%%%%%%%%%%%%%%%%%%%%%%%
\newpage
\clearpage
%<=== the same size as in the main text
%\topmargin=-35pt 
%\textheight=9.0in
%<=== the same size as in the main text
\baselineskip=15pt 
\setstretch{0.3}      

%\centerline{\bf Appendix} 
\section*{Appendix}

\setcounter{section}{0} 
\setcounter{equation}{0} 
\setcounter{lemma}{0}\setcounter{page}{1}\setcounter{proposition}{0} %
\renewcommand{\thepage}{A-\arabic{page}} \renewcommand{\theequation}{A.%
\arabic{equation}}\renewcommand{\thelemma}{A.\arabic{lemma}} 
\renewcommand{\theproposition}{A.\arabic{proposition}}

Throughout the appendix, 
we use $C$, $C_{1}$,$C_{2}$, $\dots$ to denote generic positive constants without further
clarification.
Also, 
we use $\mathrm{diag}(\cdot)$ to denote the operator that generates a square diagonal matrix with its diagonal entries being equal to its inputs. 
The key ingredients in the proofs are 
a Strong Approximation Theorem (SAT), 
a Functional Central Limit Theorem (FCLT) and 
a generalized Hajek-Renyi inequality.
We first state two technical lemmas.

%%%%%%%%%%%%%%%%%%%%%%%%%%%%%%%%%%%%%%%%%%%%%%%%%%%%%%%%%%%%%%%%%%%%%%%%%%%%%%%%%%% 

%%%%%%%%%%%%%%%%%%%%%%%%%%%%%%%%%%%%%%%%%%%%%%%%%%%%%%%%%%%%%%%%%%%%%%%%%%%%%%%%%%% 

%%%%%%%%%%%%%%%%%%%%%%%%%%%%%%%%%%%%%%%%%%%%%%%%%%%%%%%%%%%%%%%%%%%%%%%%%%%%%%%%%%% 
\vspace{0.1cm}
\begin{lemma}
  Let $\{\varsigma_{t}\}_{t \in \mathbb{Z}}$ 
  be a sequence of mean-zero, $\mathbb{R}^{d}$-valued random
  vectors satisfying Assumptions A2 and A7. 
  Define $S_{k}(\ell )=\sum_{t=\ell +1}^{\ell +k}\varsigma _{t}$, then, (a) (SAT) 
  the covariance matrix of $k^{-1/2}S_{k}(\ell ),$ $\Omega
  _{k},$ converge, with the limit denoted by $\Omega $, and there exists a
  Brownian Motion $(W(t))_{t\geq 0}$ with covariance matrix $\Omega $ such
  that $\sum_{i=1}^{t}\varsigma _{i}-W\left( t\right) =O_{a.s}(t^{1/2-\kappa })$ for
  some $\kappa >0;$ (b) (FCLT) $T^{-1/2}\sum_{t=1}^{[Tr]}\varsigma_{t}\Rightarrow
  \Omega^{1/2} W^{\ast }(r)$, where $W^{\ast }(r)$ is a $\mathbb{R}^{d}$-valued 
  vector of independent Wiener processes and ``$\Rightarrow $'' denotes weak convergence under the
  Skorohod topology.
\end{lemma}
\vspace{0.1cm}
%%%%%%%%%%%%%%%%%%%%%%%%%%%%%%%%%%%%%%%%%%%%%%%%%%%%%%%%%%%%%%%%%%%%%%%%%%%%%%%%%%% 

%%%%%%%%%%%%%%%%%%%%%%%%%%%%%%%%%%%%%%%%%%%%%%%%%%%%%%%%%%%%%%%%%%%%%%%%%%%%%%%%%%% 

%%%%%%%%%%%%%%%%%%%%%%%%%%%%%%%%%%%%%%%%%%%%%%%%%%%%%%%%%%%%%%%%%%%%%%%%%%%%%%%%%%% 

The above lemma is proved in Lemma A.1 of \cite{QuPerron2007Emtca}, 
who use Theorem 2 in \cite{Eberlein1986AoP} together with the arguments of \cite{Corradi1999ET}.
The following lemma is an extension of the Hajek-Renyi inequality.

%%%%%%%%%%%%%%%%%%%%%%%%%%%%%%%%%%%%%%%%%%%%%%%%%%%%%%%%%%%%%%%%%%%%%%%%%%%%%%%%%%% 

%%%%%%%%%%%%%%%%%%%%%%%%%%%%%%%%%%%%%%%%%%%%%%%%%%%%%%%%%%%%%%%%%%%%%%%%%%%%%%%%%%% 
\vspace{0.1cm}
\begin{lemma}
  \label{lemma:HRI}
  Suppose that 
  Assumptions A1, A2 and A5 hold. 
  Let 
  $\{b_{k}\}_{k \in \mathbb{N}}$
  be a sequence of positive, non-increasing constants
  and 
  let
  $\{\xi_{tT}\}$
  denote 
  either 
  $\{X_{tT} \Sigma_{t, \mathcal{K}}^{-1} u_{t}\}$
  or 
  $
  \{\eta_{t} \eta_{t}' - I_{n}\}
  $.    
  Then, 
  for any $B >0$ and 
  for any $k_{1}, k_{2} \in \mathbb{N}$
  with
  $k_{1} < k_{2}$,
  \begin{equation*}
    \Pr
    \bigg \{
    \sup_{k_{1}  \le k \le k_{2}}
    \frac{1}{k b_{k}}
    \bigg \|
      \sum_{t=1}^{k}
      \xi_{t T}
    \bigg \|
      > 
      B
    \bigg \}
    \le
    \frac{C}{B^{2}}
    \bigg ( 
    \frac{1}{k_{1} b_{k_{1}}^{2}}
    +
    \sum_{k=k_{1}+1}^{k_{2}} 
    \frac{1}{(k b_{k})^{2}}
    \bigg).
  \end{equation*}
\end{lemma}
\begin{proof}[\textbf{\textrm{Proof}}.]
  The assertion is proved if we show that 
  $\{X_{tT} \Sigma_{t, \mathcal{K}}^{-1} u_{t}\}$
  and 
  $\{
      \eta_{t}
      \eta_{t}'
      -
      I_{n}\} 
  $    
  satisfy
  the 
  $L^{2}$-mixingale condition 
  in Lemma A6 of Bai and Perron (1998),
  which 
  shows 
  the HajeK-Renyi inequality 
  for a $L^{2}$-mixingale sequence.\footnote{
    Lemma A6 of Bai and Perron (1998) obtains a Hajek-Renyi inequality
    with the the supremum taken over
    $[k_{1}, \infty]$
    rather than 
    the original one 
    with the the supremum taken over
    a finite range $[k_{1}, k_{2}]$ as in the assertion of this lemma. 
    Their argument, however, can easily be extended to cover the case considered here.
  }
  We consider only 
  $\{X_{tT} \Sigma_{t, \mathcal{K}}^{-1} u_{t}\}$
  because the proof 
  for 
  $\{
      \eta_{t}
      \eta_{t}'
      -
      I_{n}\} 
  $    
  is similar and actually simpler.
  We use the notation $E_{t}(\cdot):= E(\cdot|\mathcal{F}_{t})$
  for $t \in \mathbb{Z}$.

  %%%%%%%%
  % 1
  %%%%%%%%
  We can write 
  $X_{tT} \Sigma_{t, \mathcal{K}}^{-1} u_{t} 
   = 
   S'
   (
   I_{q} \otimes \Sigma_{t, \mathcal{K}}^{-1} 
   (\Sigma_{t, \mathcal{K}}^{0} )^{1/2}
   )
   (x_{tT}\otimes \eta_{t})    
  $,
  where 
  $ \|
   S'
   (
   I_{q} \otimes \Sigma_{t, \mathcal{K}}^{-1} 
   (\Sigma_{t, \mathcal{K}}^{0} )^{1/2}
   )
   \|
   \le 
   C_{1}
  $
  from Assumption A5
  and 
  the term 
  $(x_{tT}\otimes \eta_{t})$
  is 
  $\mathcal{F}_{t}$-measurable.
  Thus, it suffices to show that
  there exist non-negative constants 
  $\{\psi_{j}\}_{j \ge 0}$
  such that,
  for all $t \ge 1$ and $j \ge 0$, 
  \begin{eqnarray}
    \label{eq:3}
    \big 
    \Vert  
    E_{t-j}
    \big (
      x_{tT}\otimes \eta_{t} 
    \big )
    - 
    E
    \big( 
       x_{tT}\otimes \eta_{t}
    \big )
    \big \Vert_{2} 
    \le 
    C_{2}\psi_{j},
  \end{eqnarray}
  as well as 
  $\psi_{j} \to 0$
  as $j \to \infty$
  and 
  $\sum_{j=1}^{\infty}j^{1+\vartheta} \psi_{j} < \infty$
  for some $\vartheta > 0$.

  %%%% 
  %%%%
  %%%% 
  In order to show 
  (\ref{eq:3}),
  we write 
  $
    x_{tT}\otimes \eta_{t}
    =
    \big [ 
      z_{t}' \otimes \eta_{t}', 
      \varphi(t/T)' \otimes \eta_{t}', 
      T^{-1/2} w_{t}' \otimes \eta_{t}' 
   \big ]'    
  $
  and 
  observe that
  $E[z_{t} \otimes \eta_{t}] = 0 $
  and 
  $E[\eta_{t}] = 0$. 
  It follows from 
  Minkowski's inequality that 
  \begin{eqnarray*}
    \big 
    \Vert  
    E_{t-j}
    (
      x_{tT}\otimes \eta_{t} 
    )
    - 
    E
    ( 
       x_{tT}\otimes \eta_{t}
    )
    \big \Vert_{2} 
    &\leq& 
    \big \Vert 
    E_{t-j}
    (z_{t}\otimes \eta_{t})
    \big \Vert_{2} 
    +
    \big \Vert 
    \varphi(t/T)
    \otimes
     E_{t-j}(\eta_{t})
    \big \Vert_{2} \\
    &&
    +
    T^{-1/2}
    \big \Vert 
    E_{t-j}
    ( 
      w_{t}\otimes \eta_{t} 
    )
    -
    E
    (
     w_{t}\otimes \eta_{t}
    )
    \big \Vert_{2}  \\
    &=:&
    A_{1} + A_{2} + A_{3}.
  \end{eqnarray*}
  For $A_{1}$ and $A_{2}$,
  an application of the mixing inequality of 
  \cite{Ibragimov1962SIAM}
  yields that\footnote{
    For $A_{2}$, 
    we use the fact 
    $\| \varphi(t/T) \otimes \eta_{t} \|_{2}^{2} 
     = E[(\varphi(t/T) \otimes \eta_{t})'(\varphi(t/T) \otimes \eta_{t})]
     = \varphi(t/T)'\varphi(t/T) E[\eta_{t}' \eta_{t}]
    $,
    which implies that 
    $
    \| \varphi(t/T) \otimes \eta_{t} \|_{2}
    \le C \|\eta_{t}\|_{2}
    $.
  }
  \begin{eqnarray}
    \label{eq:A-1}
    A_{1}
    \le 
    2(\sqrt{2} + 1)
    \alpha_{j}^{1/2 - 1/\phi}
    \| z_{t}\otimes \eta_{t} \|_{\phi}
    \ \ \ \mathrm{and} \ \ \
    A_{2}
    \le 
    2(\sqrt{2} + 1)
    \alpha_{j}^{1/2 - 1/\phi}
    \| \eta_{t} \|_{\phi},
  \end{eqnarray}
  where $ \phi:= 4 + \delta$  with
  $\delta$ defined in Assumption A2. 
  %%%% (i) %%%%
  For the term $A_{3}$, we separately consider two cases:
  (i)   $t < j$
  and 
  (ii)  $t \ge j $,
  given $t \ge 1$.
  First, we consider case (i), i.e., $t - j < 0$.
  We have 
  $ w_{t}
    =
    w_{0}
    +
    \sum_{l=0}^{t-1}  u_{w, t-l} 
  $, 
  which with Minkowski's inequality implies that 
  \begin{eqnarray*}
    \sqrt{T}
    A_{3}
    \le 
    \big \|
     E_{t-j}(
     w_{0} \otimes \eta_{t}
     )
     - 
     E(
     w_{0} \otimes \eta_{t}
     )
    \big \|_{2} 
    +
    \sum_{l=0}^{t-1}
    \big \|
    E_{t-j}
    \big ( 
      u_{w, t-l} \otimes \eta_{t} 
    \big )
    -
    E
    \big ( 
    u_{w, t-l}\otimes \eta_{t}
    \big ) 
    \big \|_{2} .
  \end{eqnarray*}
  Since
  $
    \|
    E_{t-j}(V) - E(V)
    \|_{2} 
    \le 
    \|
    E_{t-j}(V)
    \|_{2} 
  $   
  for a random vector $V$,
  an application of Jensen's inequality 
  and Corollary 14.3 of \cite{Davidson1994Book}
  (a covariance inequality for a $\alpha $-mixing sequence)
  yields that 
  \begin{eqnarray}
    \label{eq:INEQ-1}
    \big \|
     E_{t-j}(
     w_{0} \otimes \eta_{t}
     )
     - 
     E(
     w_{0} \otimes \eta_{t}
     )
    \big \|_{2} 
    \le 
    \big \|
     w_{0} \otimes \eta_{t}
    \big \|_{2} 
    \le 
    C_{3}
    \alpha_{t}^{1/2 - 1/ \phi},
  \end{eqnarray}
  and that, for $0 \le l \le t-1$, 
  \begin{eqnarray}
    \label{eq:INEQ-2}
    \big \|
    E_{t-j}
    \big ( 
      u_{w, t-l} \otimes \eta_{t} 
    \big )
    -
    E
    \big ( 
    u_{w, t-l}\otimes \eta_{t}
    \big ) 
    \big \|_{2} 
    \le 
    \big \|
      u_{w, t-l} \otimes \eta_{t} 
    \big \|_{2} 
    \le 
    C_{4}
    \alpha_{l}^{1/2 - 1/ \phi}.
  \end{eqnarray}
  Also, using the mixing inequality of \cite{Ibragimov1962SIAM}, we can show that 
  \begin{eqnarray}
    \label{eq:INEQ-3}
    \big \|
     E_{t-j}(
     w_{0} \otimes \eta_{t}
     )
     - 
     E(
     w_{0} \otimes \eta_{t}
     )
    \big \|_{2} 
    \le 
    2(\sqrt{2}+1)
    \alpha_{j-t}^{1/2 - 1/ \phi}
    \big \|
     w_{0} \otimes \eta_{t}
    \big \|_{\phi},
  \end{eqnarray}
  and that, for $0 \le l \le t-1$,  
  \begin{eqnarray}
    \label{eq:INEQ-4}
    \big \|
    E_{t-j}
    \big ( 
      u_{w, t-l} \otimes \eta_{t} 
    \big )
    -
    E
    \big ( 
    u_{w, t-l}\otimes \eta_{t}
    \big ) 
    \big \|_{2} 
    \le 
    2(\sqrt{2}+1)
    \alpha_{j-l}^{1/2 - 1/ \phi}
    \big \|
     u_{w,t-l} \otimes \eta_{t}
    \big \|_{\phi},
  \end{eqnarray}
  where both moments on the right-hand side of 
  (\ref{eq:INEQ-3}) and  (\ref{eq:INEQ-4})
  are bounded from Assumption A2.
  It follows from 
  (\ref{eq:INEQ-1})-(\ref{eq:INEQ-4}) that, when $t < j$, we have
  \begin{eqnarray}
    \label{eq:case-i}
    A_{3}
    \le 
    C_{5}
    T^{-1/2}
    \sum_{l=0}^{t}
    \min\{
    \alpha_{l}^{1/2 - 1/ \phi}, \alpha_{j -l}^{1/2 - 1/ \phi}
    \}
    \le 
    C_{5}
    j^{1/2}
    \alpha_{[j/2]}^{1/2 - 1/ \phi},
  \end{eqnarray}
  where 
  the last inequality is due to the fact that 
  $\min\{
    \alpha_{l}^{1/2 - 1/ \phi}, \alpha_{j -l}^{1/2 - 1/ \phi}
    \}
    \le 
    \alpha_{[j/2]}^{1/2 - 1/ \phi}
  $
  for every 
  $0 \le l \le t$
  and 
  that 
  $T^{-1/2} t \le t^{1/2} \le j^{1/2}$
  for $t < j$.

  %%%%%%%%%%
  % (ii)
  %%%%%%%%%%
  Next, we consider case (ii), i.e., $0 \le t - j$.
  Since 
  $ w_{t}
    =
    w_{t-j}
    +
    \sum_{l=0}^{j-1}  u_{w, t-l} 
  $, 
  Minkowski's inequality leads to 
  \begin{eqnarray}
    \label{eq:BB1}
    \sqrt{T}
    A_{3}
    \le 
    \big \|
      w_{t-j} 
      \otimes 
      E_{t-j}(\eta_{t})
    \big \|_{2} 
    +
    \sum_{l=0}^{j-1}
    \big \|
    E_{t-j}
    \big ( 
      u_{w, t-l} \otimes \eta_{t} 
    \big )
    -
    E 
    \big ( 
    u_{w, t-l}\otimes \eta_{t}
    \big ) 
    \big \|_{2} .
  \end{eqnarray}
  Using the Cauchy-Schwarz 
  and 
  Ibragimov's mixing inequalities,
  we can show that 
  \begin{eqnarray}
    \label{eq:CC-1}
    \big \|
      w_{t-j}
      \otimes 
      E_{t-j}(\eta_{t}) 
    \big \|_{2} 
    \le 
    \big \|
    w_{t-j}
    \big \|_{2}
    \big \|
    E_{t-j}(\eta_{t}) 
    \big \|_{2} 
    \le 
    \big \|
    w_{t-j}
    \big \|_{2}
    C_{6}  
    \alpha_{j}^{1/2 - 1/\phi}.
  \end{eqnarray}
  Furthermore, we can write 
  $
   \| w_{t-j} \|_{2}^2 
   = 
   \sum_{s=1}^{t-j} 
   E[ u_{w s}' u_{w s} ] 
   +
   2
   \sum_{k=1}^{t-j - 1} 
   \sum_{s=1}^{t-j-k} 
   E[ u_{w s}' u_{w, s+ k } ]
   $,
   which with 
   Corollary 14.3 of \cite{Davidson1994Book}
   implies 
   \begin{eqnarray*}
     T^{-1} \| w_{t-j} \|_{2}^{2}
     \le 
     C_{7} 
     \Big (
     \frac{t- j}{T}
     + 
     \sum_{k=1}^{t-j - 1} 
     \frac{t-j-k}{T}
     \alpha_{k}^{1/2 - 1/\phi}
     \Big )
     \le 
     C_{8}. 
   \end{eqnarray*}
   Also, applying the same arguments used in case (i), 
   we can show that 
  \begin{eqnarray}
    \label{eq:AA2}
    \sum_{l=0}^{j-1}
    \big \|
    E_{t-j}   
    ( 
      u_{w, t-l} \otimes \eta_{t} 
    )
      -
    E( 
       u_{w, t-l}\otimes \eta_{t}
    ) 
    \big \|_{2}   
    \le 
    C_{9} 
    \sum_{l=0}^{j-1}
    \min\{
    \alpha_{l}^{1/2 - 1/\phi},
    \alpha_{j-l}^{1/2 - 1/\phi}
    \}.
  \end{eqnarray}
  Combining the results in (\ref{eq:CC-1})-(\ref{eq:AA2}),
  we obtain 
  \begin{eqnarray*}
    A_{3}
    \le 
    C_{10} 
    \big (
    \alpha_{j}^{1/2 - 1/\phi}
    +
    T^{-1/2}
    j
    \alpha_{[j/2]}^{1/2 - 1/\phi}
    \big )
    \le 
    C_{11} 
    j^{1/2}
    \alpha_{[j/2]}^{1/2 - 1/\phi}.
  \end{eqnarray*}
  Thus, 
  from the above equation and (\ref{eq:case-i}),
  we obtain that 
  $A_{3} \le 
    C_{12} 
    j^{1/2}
    \alpha_{[j/2]}^{1/2 - 1/\phi}
  $
  for every $t \ge 1$. 
  This result together with
  (\ref{eq:A-1})   and (\ref{eq:BB1}) 
  yields 
  \begin{eqnarray*}
    \big 
    \|
    E_{t-j}
    (
      x_{tT}\otimes \eta_{t} 
    )
    - 
    E
    ( 
       x_{tT}\otimes \eta_{t}
    )
    \big
    \|_{2}
    \le 
    C_{13} j^{1/2} \alpha_{[j/2]}^{1/2 - 1/\phi}.
  \end{eqnarray*}

  %%%%%%%%
  %%%%%%%%
  %%%%%%%%
  We set 
  $\psi_{j}=j^{1/2} \alpha_{[j/2]}^{1/2 - 1/\phi}$
  and it remains to show that 
  $
  \sum_{j=1}^{\infty} j^{1+\vartheta}\psi_{j}
  < \infty
  $
  for some $\vartheta>0$.
  Observe that 
  $\alpha_{[j/2]}^{1/2 - 1/\phi} = O(j^{\frac{5}{2} - \frac{1 - 2 \delta}{\delta}})$ 
  under Assumption A2.
  Thus,  
  for  
  $
  \vartheta 
  < (1 - 2 \delta) / \delta
  $,
  we can show that 
  $
  \sum_{j=1}^{\infty} j^{1+\vartheta}\psi_{j}
  \le 
  C_{14}
  \sum_{j=1}^{\infty} j^{-1 -\frac{1- 2 \delta}{\delta} + \vartheta }
  < \infty
  $.
  This completes the proof.
\end{proof}
\vspace{0.1cm}

%%%%%%%%%%%%%%%%%%%%%%%%%%%%%%%%%%%%%%%%%%%%%%%%%%%%%%%%%%%%%%%%%%%%%%%%%%%%%%%%%%%%%%%

%%%%%%%%%%%%%%%%%%%%%%%%%%%%%%%%%%%%%%%%%%%%%%%%%%%%%%%%%%%%%%%%%%%%%%%%%%%%%%%%%%%%%%%

%%%%%%%%%%%%%%%%%%%%%%%%%%%%%%%%%%%%%%%%%%%%%%%%%%%%%%%%%%%%%%%%%%%%%%%%%%%%%%%%%%%%%%%
 
In what follows, we shall use 
a collection of sub-intervals 
$\{[\tau_{l-1}+1, \tau_{l}]\}_{l=1}^{N}$
with 
$\tau_{0} =0$
and 
$\tau_{N} =T$
as a partition of the interval $[1,T]$
according to sets of break dates $\mathcal{K}$ and $\mathcal{K}^{0}$,
such that 
both the true basic parameters and their estimates are constant within each 
sub-interval 
and 
$N$ is set to be the smallest number of such sub-intervals;
that is, 
$
( 
\beta_{t, \mathcal{K}},
\beta_{t, \mathcal{K}^{0}}^{0},
\Sigma_{t,\mathcal{K}},
\Sigma_{t,\mathcal{K}^{0}}^{0}
)
=
(
\beta_{\tau_{l}, \mathcal{K}},
\beta _{\tau_{l}, \mathcal{K}^{0}}^{0},
\Sigma_{\tau_{l}, \mathcal{K}},
\Sigma_{\tau_{l},\mathcal{K}^{0}}^{0}
)$
for $\tau_{l-1} +1 \le t \le \tau_{l}$.
For each parameter group $g \in \{1, \dots, G\}$,
we similarly consider 
a collection $\{[\tau_{g, l-1}+1, \tau_{g l}]\}_{l=1}^{N_{g}}$
with  
$\tau_{0} =0$
and 
$\tau_{N_{g}} =T$
as 
a partition of the interval $[1, T]$
given $\mathcal{K}_{g}$ and $\mathcal{K}_{g}^{0}$,
where 
both the true basic parameters and their estimates for the $g^{th}$ group
are constant 
within each sub-interval 
and 
$N_{g}$ is the smallest number of such intervals.
Thus we have 
$
( 
\beta_{g, t, \mathcal{K}},
\beta_{g, t, \mathcal{K}}^{0}
)
=
(
\beta_{g, \tau_{gl}, \mathcal{K}},
\beta _{g, \tau_{g l}, \mathcal{K}^{0}}^{0}
)$
for $\tau_{g, l-1}+1 \le t \le \tau_{g l}$
and 
$
(
\Sigma_{t,\mathcal{K}},
\Sigma_{t,\mathcal{K}^{0}}^{0}
)
=
( 
\Sigma_{\tau_{G,l}, \mathcal{K}},
\Sigma_{\tau_{G, l}, \mathcal{K}^{0}}^{0}
)
$
for $\tau_{G, l-1} +1 \le t \le \tau_{G, l}$,
whereas 
the basic parameters of the other groups may change. 
For  $\tau_{G, l-1} +1 \le t \le \tau_{G, l}$
with $l \in \{ 1, \dots, N_{g} \}$,
we define 
\begin{eqnarray}
  \label{eq:var-0}
  \Psi_{l} 
  := 
  (\Sigma _{t, \mathcal{K}^{0}}^{0})^{- 1 / 2}
  (\Sigma_{t, \mathcal{K}} -\Sigma_{t, \mathcal{K}^{0}}^{0})
  (\Sigma _{t, \mathcal{K}^{0}}^{0})^{- 1  / 2},
\end{eqnarray}
where we have 
$
  I_{n} + \Psi_{l}
  =
  (\Sigma _{\tau_{G,l}, \mathcal{K}^{0}}^{0})^{-1 / 2}
  \Sigma_{\tau_{G,l}, \mathcal{K}} 
  (\Sigma _{\tau_{G,l}, \mathcal{K}^{0}}^{0})^{-1 / 2}
$.
Since $\Psi_{l} $ is an $n \times n$ symmetric matrix, there exits an
orthogonal matrix $U$ such that 
\begin{eqnarray*}
  U\Psi U^{\prime }
  =\mathrm{diag}\{
  \lambda_{l1}^{\Psi},...,\lambda _{ln}^{\Psi}\}  
  \ \ \ \ \mathrm{and} \ \ \ \ 
  U(I_{n} + \Psi )U^{\prime }
  =\mathrm{diag}\{
  1 + \lambda_{l1}^{\Psi},...,1 + \lambda _{ln}^{\Psi}\},  
\end{eqnarray*}
where 
$\lambda _{l1}^{\Psi}, \dots, \lambda_{ln}^{\Psi}$ are the eigenvalues of $\Psi_{l}$. 

In the lemma below, we shall obtain an upper bound for the normalized 
log likelihood based on sub-intervals. 
As a short-hand notation, we define, 
for $1 \le t \le T$ and $1 \le g \le G$,
\begin{eqnarray*}
  \Delta \beta_{t, \mathcal{K}}
  :=
  \beta_{t, \mathcal{K}}
  -
  \beta _{t, \mathcal{K}^{0}}^{0}  
  \ \ \ \mathrm{and} \ \ \ 
  \Delta \beta_{g, t, \mathcal{K}}
  :=
  \beta_{g, t, \mathcal{K}}
  -
  \beta _{g, t, \mathcal{K}^{0}}^{0}  .
\end{eqnarray*}

%%%%%%%%%%%%%%%%%%%%%%%%%%%%%%%%%%%%%%%%%%%%%%%%%%%%%%%%%%%%%%%%%%%%%%%%%%%%%%%%%%%%%%%

%%%%%%%%%%%%%%%%%%%%%%%%%%%%%%%%%%%%%%%%%%%%%%%%%%%%%%%%%%%%%%%%%%%%%%%%%%%%%%%%%%%%%%%

%%%%%%%%%%%%%%%%%%%%%%%%%%%%%%%%%%%%%%%%%%%%%%%%%%%%%%%%%%%%%%%%%%%%%%%%%%%%%%%%%%%%%%%
\vspace{0.1cm}
\begin{lemma}
  \label{lemma:decomposition-l}
  Suppose that Assumptions A1-A5 hold. Then,
  \begin{eqnarray*}
    \ell_{T}(\mathcal{K},\theta)
    \le 
    C 
    \bigg \{
    \sum_{g=1}^{G}
    \sum_{l=1}^{N_{g}}
    \bar{\ell}_{g,l}(\mathcal{K},\theta) 
    +
    \sum_{l=1}^{N_{G}}
    \bar{\ell}_{G+1,l}(\mathcal{K},\theta) 
    +
    \Delta_{T}(\mathcal{K}, \theta)
    \bigg \},
  \end{eqnarray*} 
  where, for $g=1,\dots, G$ and $l=1,\dots,N_{g}$,
  \begin{eqnarray*}
    \bar{\ell}_{g, l}(\mathcal{K},\theta)
    &:=&
    \bigg (
    \bigg  \| 
    \sum_{t=\tau_{g, l-1}+1}^{\tau_{g l}}
    X_{tT}\Sigma_{t, \mathcal{K}}^{-1}u_{t}
    \bigg  \| 
    -
    (\tau_{g l} - \tau_{g, l-1})
    \big \|
    \Delta \beta_{g, \tau_{gl}, \mathcal{K}} 
    \big \|
    \bigg )
    \big \|
    \Delta 
    \beta_{g, \tau_{gl}, \mathcal{K}}
    \big \|,\\
    \bar{\ell}_{G+1,l}(\mathcal{K},\theta)
    &:=&
    \sum_{i=1}^{n}
    \bigg (
    \bigg \|
    \sum_{t=\tau_{G, l-1}+1}^{\tau_{G l}}(\eta _{t}\eta
    _{t}^{\prime }-I_{n})
    \bigg \|
    -
    (\tau_{G l} - \tau_{G, l-1})
    |\lambda_{il}^{\Psi}|
    \bigg )  
    |\lambda_{il}^{\Psi}|, \\
    \Delta_{T}(\mathcal{K}, \theta)
    &:=&
    \max_{1 \le t \le T}
    \| 
    \Delta \beta_{t, \mathcal{K}}
    \|.
  \end{eqnarray*}
\end{lemma}
\begin{proof}[\textbf{Proof}.]
  We can write 
  $
    \log f(y_{t}|X_{tT}, \theta_{t, \mathcal{K}})
    =
    - 
    (1/2)
    \big (
    \log (2\pi )^{n}
    +
    \log |\Sigma_{t, \mathcal{K}} |
    + 
    \| 
    \Sigma_{t, \mathcal{K}} ^{-1/2}
    (
     u_{t}-X_{tT}^{\prime } \Delta \beta_{t,\mathcal{K}}
    )
    \|^{2}
    \big )
  $,
  which implies that 
  \begin{eqnarray*}
    \ell_{T}(\mathcal{K},\theta)
    &=&
    -\frac{1}{2}
    \sum_{t=1}^{T}
    \big (
    \log 
    \big | \Sigma_{t, \mathcal{K}} \big|
    -
    \log 
    \big | \Sigma_{t, \mathcal{K}^{0}}^{0} \big|
    +
    \big \| 
    \Sigma_{t, \mathcal{K}}^{-1/2}u_{t}
    \big \|^{2}
    -
    \big \| 
    (\Sigma _{t, \mathcal{K}^{0}}^{0})^{-1/2} u_{t}   
    \big \| ^{2}  
    \big ) \\
    &&  
    +
    \sum_{t=1}^{T}
    \Delta \beta_{t, \mathcal{K}}^{\prime }
    X_{tT}\Sigma_{t, \mathcal{K}}^{-1}u_{t} 
    -
    \frac{1}{2}
    \sum_{t=1}^{T}
    \big \|
    \Sigma_{t, \mathcal{K}}^{-1/2}
    X_{tT}^{\prime }
    \Delta \beta_{t, \mathcal{K}}
    \big \|^{2} \\
    &=:&   
    A_{1}
    +
    A_{2}
    +
    A_{3}.
  \end{eqnarray*}

  %%%%%%%%%%%%%%
  % I-1
  %%%%%%%%%%%%%%
  For the term $A_{1}$,
  we write 
  $ 
    \log 
    \big | \Sigma_{t, \mathcal{K}} \big|
    -
    \log 
    \big | \Sigma_{t, \mathcal{K}^{0}}^{0} \big|
    =
    \log 
    \big | 
    (\Sigma _{t, \mathcal{K}^{0}}^{0})^{-1/2}
    \Sigma_{t, \mathcal{K}} 
    (\Sigma _{t, \mathcal{K}^{0}}^{0})^{-1/2}
    \big|
  $
  and also
  $u_{t}= (\Sigma _{t, \mathcal{K}^{0}}^{0})^{1/2} \eta_{t}$.
  Since $A_{1}$ depends only on 
  $\mathcal{K}_{G}$
  and 
  $\mathcal{K}_{G}^{0}$,
  we have 
  \begin{eqnarray*}
    A_{1}
    = 
    \sum_{l=1}^{N_{G}}
    \bigg \{
    -\frac{1}{2}
    \sum_{t=\tau_{G,l-1}+1}^{\tau_{Gl}}
    \Big (
    \log 
    \big | 
    I_{n} + \Psi_{l}
    \big|
    +
    \mathrm{tr}
    \big (
    (I_{n} + \Psi_{l})^{-1}
    \eta_{t}
    \eta_{t}'
    \big )
    -
    \mathrm{tr}
    \big (
    \eta_{t}
    \eta_{t}'
    \big )
    \Big )
    \bigg \} 
    =:
    \sum_{l=1}^{N_{G}}
    A_{1,l}.
  \end{eqnarray*}
  For every $l= 1, \dots, N_{G}$,
  we have that 
  $
    \log 
    \big |
    I_{n} +\Psi_{l} 
    \big |
    = 
    \sum_{i=1}^{n} \log (1 + \lambda_{l i}^{\Psi})
  $
  and that 
  \begin{eqnarray*}
    \mathrm{tr}
    \big (
    (I_{n} + \Psi_{l})^{-1} \eta_{t}\eta_{t}' 
    \big )
    =
    \mathrm{tr}
    \bigg (
    \mathrm{diag}
    \left( 
      \left \{
        \frac{1}{1+\lambda _{li}^{\Psi}}
      \right \}_{i=1}^n 
      \right) 
        U^{\prime }
        \big (
        \eta _{t}\eta _{t}^{\prime}
        \big )
        U 
        \bigg ),
  \end{eqnarray*}
  which leads to 
  \begin{eqnarray*}
    A_{1,l}
    =
    -
    \frac{
      \tau_{Gl}-\tau_{G, l-1}
    }{
      2
    }
    \sum_{i=1}^{n} 
    \log (1+\lambda _{li}^{\Psi})
    +
    \frac{1}{2}
    \mathrm{tr} 
    \bigg (
      \mathrm{diag}\left( \left \{ 
          \frac{\lambda _{li}^{\Psi}}{1+\lambda _{li}^{\Psi}}
        \right \}_{i=1}^n \right)
      U^{\prime }
      \bigg (
      \sum_{t=\tau_{G, l-1}+1}^{\tau_{Gl}}
      \eta _{t}\eta _{t}^{\prime}
      \bigg )
      U 
    \bigg ).
  \end{eqnarray*}
  We can show that  
  $- \log (1+a)
   + 
   a/ (1+a)
   \le  
   -
   a^2
   /
   (1 + a)
  $
  for $0 < a < \infty$
  \citep[see][for instance]{Dragomir2016}.
  Thus,
  \begin{eqnarray*}
    A_{1,l}
    \le
    -
    \frac{
      \tau_{Gl}-\tau_{G, l-1}
    }{
      2
    }
    \sum_{i=1}^{n} 
    \frac{
      |\lambda_{i}^{\Psi}|^2
    }{
      1 + \lambda_{li}^{\Psi}
    }
    +
    \frac{1}{2}
    \mathrm{tr} 
    \bigg (
    \mathrm{diag}
    \left(
      \left \{ 
        \frac{
          \lambda _{li}^{\Psi}
        }{
          1+\lambda _{li}^{\Psi}
        }
      \right \}_{i=1}^n 
    \right)
    U'
    \bigg (
    \sum_{t= \tau_{G, l-1}+1}^{\tau_{Gl}}
    ( \eta _{t} \eta _{t}' - I_{n} )  
    \bigg )
    U 
    \bigg ).
  \end{eqnarray*}
  Since 
  the maximum of 
  the diagonal elements of 
  $U^{\prime} \big ( \sum_{t=\tau_{G,l-1}+1}^{\tau_{Gl}}(\eta _{t}\eta _{t}^{\prime}-I_{n}) \big )U$
  is bounded from above by 
  $\| 
   U^{\prime }
   \big (
   \sum_{t=\tau_{G,l-1}+1}^{\tau_{Gl}}(\eta _{t}\eta _{t}^{\prime}-I_{n})  
   \big )
  U
  \|
  $ with $\| U\| = 1$,    
  we have 
  \begin{eqnarray}
    \label{eq:l-G-l}
    A_{1,l}
    \leq 
    \frac{1}{2}
    \sum_{i=1}^{n}
    \bigg \{
    -
    (\tau_{Gl}-\tau_{G, l-1})
    \frac{
      |\lambda_{li}^{\Psi}|^2
    }{
      1 + \lambda_{li}^{\Psi}
    }
    +
    \frac{
      |\lambda _{li}^{\Psi}|
    }{
      1+\lambda _{li}^{\Psi}
    }
    \bigg \|
    \sum_{t=\tau_{G,l-1}+1}^{\tau_{Gl}}(\eta _{t}\eta _{t}^{\prime}-I_{n})
    \bigg \|
    \bigg \}. 
  \end{eqnarray}
  From the compactness of $\Theta$
  and (\ref{eq:var-0}),
  we have 
  $\max_{1 \le i \le n}(1+ \lambda_{li}^{\Psi} )= \|I_{n}+\Psi_{l} \| \le C_{1}$
  and 
  \begin{eqnarray*}
    1+ \min_{1 \le i \le n} \lambda_{l i}^{\Psi}
    = 
    \min_{a \in \mathbb{R}^n}
    \frac{
      a' (I_{n} + \Psi_{l} ) a
    }{ 
      a'a
    }
    \ge 
    \bigg (
    \min_{b \in \mathbb{R}^n}
    \frac{
      b' 
      \Sigma_{\tau_{Gl}, \mathcal{K}} 
      b
    }{ 
      b'b
    }
    \bigg )
    \times
    \bigg (
    \min_{a \in \mathbb{R}^n}
    \frac{
      a' 
      (\Sigma _{\tau_{Gl}, \mathcal{K}^{0}}^{0})^{-1}
      a
    }{ 
      a'a
    } 
    \bigg )
    \ge C_{2}.
  \end{eqnarray*}
  Thus we have that 
  $ C_{2} \le  1+\lambda_{li}^{\Psi} \le C_{1}$
  for all $i = 1, \dots, n$.
  This together with (\ref{eq:l-G-l}) yields 
  \begin{eqnarray*}
    A_{1,l}
    \leq 
    C_{3}
    \sum_{i=1}^{n}
    \bigg \{
    -
    (\tau_{Gl}-\tau_{G, l-1})
    |\lambda_{li}^{\Psi}|^2
    +
    |\lambda _{li}^{\Psi}|
    \bigg \|
    \sum_{t=\tau_{G,l-1}+1}^{\tau_{Gl}}(\eta _{t}\eta _{t}^{\prime}-I_{n})
    \bigg \|
    \bigg \}. 
  \end{eqnarray*}
  It follows that 
  $A_{1} 
   \le 
   C_{4}
   \sum_{l=1}^{N_{G}}
   \bar{\ell}_{G+1,l}(\mathcal{K},\theta) 
  $.

  %%%%%%%%%%%%%%
  % I-2 and I-3
  %%%%%%%%%%%%%%
  We now consider $A_{2}$ and $A_{3}$.
  Note that 
  $  
  \Delta \beta_{t, \mathcal{K}}
  = 
  \sum_{g=1}^{G}
  \Delta 
  \beta_{g, t, \mathcal{K}}
  $,
  and 
  \begin{eqnarray}
    \label{eq:I-1}
    A_{2}
    =
    \sum_{g=1}^{G}
    \sum_{t=1}^{T}
    \Delta 
    \beta_{g, t, \mathcal{K}}'
    X_{tT}\Sigma_{t, \mathcal{K}}^{-1}u_{t}.
  \end{eqnarray}
  Also, given 
  $
  X_{tT}\Sigma_{t, \mathcal{K}}^{-1}X_{tT}^{\prime }
  = S' (x_{tT}x_{tT}^{\prime }\otimes \Sigma_{\tau_{l}, \mathcal{K}}^{-1} )S
  $
  for $\tau_{l-1} +1 \le t \le \tau_{l}$,
  we can show that 
  \begin{eqnarray*}
    A_{3}
    =
    \sum_{l=1}^{N}
    \bigg \{
    -
    \frac{1}{2}
    \bigg \|
    \bigg (
    \sum_{t=\tau_{l-1}+1}^{\tau_{l}}x_{tT}x_{tT}^{\prime}
    \otimes 
    \Sigma_{\tau_{l}, \mathcal{K}}^{-1}
    \bigg )^{1/2}
    S
    \Delta 
    \beta_{\tau_{l}, \mathcal{K}}
    \bigg \|^{2}
    \bigg \}
    =:
    \sum_{l=1}^{N}
    A_{3,l}.
  \end{eqnarray*}
  Under Assumption A1,
  there exists a finite integer $k_{0}$ such that 
  the minimum eigenvalue of 
  $
  (\tau_{l} - \tau_{l-1})^{-1}
  \sum_{t=\tau_{l-1}+1}^{\tau_{l}}x_{tT}x_{tT}^{\prime} 
  $ 
  is strictly positive 
  for every $(\tau_{l}- \tau_{l-1})\ge k_{0}$
  and 
  also the eigenvalues of 
  $\Sigma_{\tau_{l},\mathcal{K}}$
  take finite positive values
  in ${\Theta}$ from Assumption A5.
  Thus, 
  an application of the result that 
  $\min_{1 \le i \le n}\lambda_{i}(A) \| b\|^2\leq b^{\prime }Ab
  \leq \max_{1 \le i\le n}\lambda _{i}(A) \|b\|^2$
  for an $n \times 1$ vector $b$
  and 
  an $n \times n$ symmetric matrix $A$ with eigenvalues $\{\lambda _{i}(A)\}_{i=1}^{n}$
  yields that, when $\tau_{l}- \tau_{l-1} \ge k_{0}$,
  \begin{eqnarray}
    \label{eq:I3-1}
    A_{3,l}
    \le 
    - 
    C_{5}
    (\tau_{l}- \tau_{l-1})
    \|
    S
    \Delta
    \beta_{\tau_{l}, \mathcal{K}} 
    \|^2
    \le 
    - 
    C_{6}
    (\tau_{l}- \tau_{l-1})
    \|
    \Delta
    \beta_{\tau_{l}, \mathcal{K}} 
    \|^2,
  \end{eqnarray}
  where 
  the last inequality is due to the fact that 
  $S'S$ is positive definite.\footnote{
    The selection matrix $S$ is 
    of dimension $nq \times p$
    with full column rank
    and 
    thus 
    $S v \not = 0$
    for all $v \in \mathbb{R}^p$
    with $v \not = 0$.
    It follows that 
    $v' S'S v \not = 0$
    for all $v \in \mathbb{R}^p$
    with $v \not = 0$
    and 
    $S'S$ positive definite.
    This implies that 
    there exists a constant $c>0$
    such that 
    $\|S b \|  \ge c \| b\|$
    for any $b \in \mathbb{R}^p$.
  }
  When $\tau_{l}- \tau_{l-1} < k_{0}$, we have that 
  $ (\tau_{l}- \tau_{l-1}) 
    \|
    \Delta \beta_{\tau_{l}, \mathcal{K}}
    \|^2
    \le 
    C_{7}
    \|
    \Delta \beta_{\tau_{l}, \mathcal{K}}
    \|^2
  $, which yields 
  \begin{eqnarray}
    \label{eq:I3-2}
    A_{3,l}
    \le 
    0 
    \le 
    -
    C_{8}
    (\tau_{l}- \tau_{l-1}) 
    \|
    \Delta \beta_{\tau_{l}, \mathcal{K}}
    \|^2
    +
    C_{9}
    \|
    \Delta \beta_{\tau_{l}, \mathcal{K}}
    \|^2.
  \end{eqnarray}
  It follows from 
  (\ref{eq:I3-1})
  and
  (\ref{eq:I3-2})
  that
  $
    A_{3}
    \le 
    -
    C_{10}
    \sum_{l=1}^{N}
    (\tau_{l}- \tau_{l-1}) 
    \big \|
    \beta_{\tau_{l}, \mathcal{K}} -\beta _{\tau_{l}, \mathcal{K}^{0}}^{0}
    \big \|^2
    +
    C_{11}
    \Delta_{T}(\mathcal{K}, \theta)
  $.
  Also,
  we can show that 
  $
    \sum_{l=1}^{N}
    (\tau_{l}- \tau_{l-1}) 
    \big \|
    \Delta \beta_{\tau_{l}, \mathcal{K}} 
    \big \|^2
    =
    \sum_{t=1}^{T}
    \big \|
    \Delta \beta_{t, \mathcal{K}}
    \big \|^2
  $
  and that
  $
  \big \|
    \Delta \beta_{t, \mathcal{K}} 
  \big \|^2
  = 
  \sum_{g=1}^{G}
  \big \|
    \Delta 
    \beta_{g, t, \mathcal{K}} 
  \big \|^2
  $  
  because 
  $
  (
  \Delta 
  \beta_{g_{1}, t, \mathcal{K}} 
  )'
  \Delta
  \beta_{g_{2}, t, \mathcal{K}} 
  = 0
  $
  for all
  $g_{1}, g_{2} 
   {\in}
   \{1, \dots, G\}
  $
  with $g_{1} \not = g_{2}$.
  Thus,
  \begin{eqnarray}
    \label{eq:I-2}
    A_{3} 
    \le 
    -
    C_{12} 
    \sum_{g=1}^{G}
    \sum_{t=1}^{T}
    \big \|
    \Delta \beta_{g, t, \mathcal{K}} 
    \big \|^2
    + 
    C_{13}
    \Delta_{T}(\mathcal{K}, \theta).
  \end{eqnarray}
  For each $g = 1, \dots, G$, 
  we have partitions $\{ [\tau_{g, l-1}+1,\tau_{gl}]\}$
  of an interval $[1, T]$. 
  From,
  (\ref{eq:I-1})
  and 
  (\ref{eq:I-2}),
  $ A_{2} + A_{3} 
    \le 
    C_{14}
    \{
    \sum_{g=1}^{G}
    \sum_{l=1}^{N_{g}}
    \bar{\ell}_{g,l}(\mathcal{K},\theta)
    +
    \Delta_{T}(\mathcal{K}, \theta)
    \}
  $.  
  Hence, the result follows.
\end{proof}
\vspace{0.1cm}
%%%%%%%%%%%%%%%%%%%%%%%%%%%%%%%%%%%%%%%%%%%%%%%%%%%%%%%%%%%%%%%%%%%%%%%%%%%%%%%%%%%%%%%
 
%%%%%%%%%%%%%%%%%%%%%%%%%%%%%%%%%%%%%%%%%%%%%%%%%%%%%%%%%%%%%%%%%%%%%%%%%%%%%%%%%%%%%%%

%%%%%%%%%%%%%%%%%%%%%%%%%%%%%%%%%%%%%%%%%%%%%%%%%%%%%%%%%%%%%%%%%%%%%%%%%%%%%%%%%%%%%%%

We shall establish several properties of the terms 
$\{\bar{\ell}_{g, l}(\mathcal{K}, \theta)\}_{g=1}^{G+1}$
based on subsamples free from structural changes.  
To this end, we consider a sequence $\{\xi_{t}\}_{t=1}^{T}$ of some random vectors or matrices 
satisfying the condition under which the Hajek-Renyi inequality in Lemma \ref{lemma:HRI}
holds. 
Let $\gamma$ be a parameter vector or matrix as an element of the bounded parameter space
$\Gamma:= \{\gamma: \|\gamma \| \le C\}$.
We define an object depending on a subsample of $k$ observations free from structural changes
in $\gamma$, namely for $k=1, \dots, T$,
\begin{eqnarray*}
  \ell_{k}^{(0)}(\gamma):=
  \bigg (
  \bigg \|
  \sum_{t=1}^{k}
  \xi_{t}
  \bigg \|
  - 
  k 
  \|\gamma \|
  \bigg )
  \|\gamma \|.  
\end{eqnarray*}

We now establish a series of properties related to the likelihood function
that will enable us to prove the rate of convergence of the estimates. 
Under the level of generality adopted here, one can apply 
the arguments used in \cite{BaiLumsdaineStock1998}
to prove the properties of the likelihood function
with some modifications. 
However, since these properties are key ingredients to prove theorems, 
we provide the whole proof.

%%%%%%%%%%%%%%%%%%%%%%%%%%%%%%%%%%%%%%%%%%%%%%%%%%%%%%%%%%%%%%%%%%%%%%%%%%%%%%%%%%%%%%%

%%%%%%%%%%%%%%%%%%%%%%%%%%%%%%%%%%%%%%%%%%%%%%%%%%%%%%%%%%%%%%%%%%%%%%%%%%%%%%%%%%%%%%%

%%%%%%%%%%%%%%%%%%%%%%%%%%%%%%%%%%%%%%%%%%%%%%%%%%%%%%%%%%%%%%%%%%%%%%%%%%%%%%%%%%%%%%%
\vspace{0.1cm}
\begin{property}
  \label{property:opt1} 
  \ \ 
  $
    \sup_{1\leq k\leq T} 
    \sup_{\gamma \in \Gamma}
    \ell_{k}^{(0)}(\gamma)
    \le 
    |O_p\big ( \log T \big ) |.
  $
\end{property}
\begin{proof}[\textbf{Proof}.] 
  Let $D>0$ 
  and 
  define 
  $\Gamma_{1, k}(D)
   := 
   \{
   \gamma \in \mathcal{G}: 
   \sqrt{k} \| \gamma \| \le D (\log T)^{1/2} 
   \}
  $
  for 
  $1 \le k \le T$.
  We can write 
  $
  \ell_{k}^{(0)}(\gamma)
  =
  \big (
  k^{-1/2}
  \|
  \sum_{t=1}^{k}
  \xi_{t}
  \|
  - 
  \sqrt{k} 
  \|\gamma \|
  \big )
  \sqrt{k}
  \|\gamma \|
  $
  for every 
  $1 \le k \le T$.
  It follows that, for any $1 \le k \le T$,
  \begin{eqnarray*}
    \sup_{\gamma \in \Gamma \setminus \Gamma_{1, k}(D)}
    \ell_{k}^{(0)}(\gamma)
    \le 
    \sup_{\gamma \in \Gamma \setminus \Gamma_{1, k}(D)}
    \bigg (
    \frac{1}{\sqrt{k}}
    \bigg \|
    \sum_{t=1}^{k}
    \xi_{t}
    \bigg \|
    - 
    D (\log T)^{1/2}
    \bigg )
    \sqrt{k}
    \|\gamma \|,
  \end{eqnarray*}
  and 
  \begin{eqnarray*}
    \sup_{\gamma \in \Gamma_{1, k}(D)}
    \ell_{k}^{(0)}(\gamma)
    \le 
    \frac{1}{\sqrt{k}}
    \bigg \|
    \sum_{t=1}^{k}
    \xi_{t}
    \bigg \|
    D (\log T)^{1/2}.
  \end{eqnarray*}
  Lemma \ref{lemma:HRI}
  implies that, for any $B_{1}>0$,
  \begin{eqnarray*}
    \Pr 
    \bigg \{
    \sup_{1 \leq k\leq T}
    \frac{
      1
    }{
      \sqrt{ k \log T}
    }
    \bigg \| 
    \sum_{t=1}^{k}\xi _{t}
    \bigg \| 
    \ge 
    B_{1}
    \bigg \}
    \le 
    \frac{C_{1}}{B_{1}^{2} \log T}
    \sum_{k=1}^{T} \frac{1}{k}. 
  \end{eqnarray*}
  The right-hand side of the above inequality
  becomes arbitrarily small 
  for a sufficiently large $B_{1}$
  because $\sum_{k=1}^{T} k^{-1} = O(\log T)$.
  Thus, 
  $
  \sup_{1 \le k \le T}
  k^{-1/2}
  \|
  \sum_{t=1}^{k}
  \xi_{t}
  \|
  - D (\log T)^{1/2}
  < 0
  $
  with probability approaching 1
  for a sufficiently large $D$, so that
  \begin{eqnarray*}
    \sup_{1\leq k\leq T} 
    \sup_{\gamma \in \Gamma \setminus \Gamma_{1, k}(D)}
    \ell_{k}^{(0)}(\gamma)
    \le 
    - 
    C_{2} D^2 \log T
    \ \ \ \mathrm{and} \ \ \
    \sup_{1\leq k\leq T} 
    \sup_{\gamma \in \Gamma_{1, k}(D)}
    \ell_{k}^{(0)}(\gamma)
    \le 
    C_{3} D \log T,
  \end{eqnarray*}
  with probability approaching 1.
  Hence, the desired conclusion follows.
\end{proof}
\vspace{0.1cm}
%%%%%%%%%%%%%%%%%%%%%%%%%%%%%%%%%%%%%%%%%%%%%%%%%%%%%%%%%%%%%%%%%%%%%%%%%%%%%%%%%%%%%%%

%%%%%%%%%%%%%%%%%%%%%%%%%%%%%%%%%%%%%%%%%%%%%%%%%%%%%%%%%%%%%%%%%%%%%%%%%%%%%%%%%%%%%%%

%%%%%%%%%%%%%%%%%%%%%%%%%%%%%%%%%%%%%%%%%%%%%%%%%%%%%%%%%%%%%%%%%%%%%%%%%%%%%%%%%%%%%%%
\begin{property}
 \label{property:smallout} 
 For any $D>0$, there exists a constant $A>0$ such that,
 for any deterministic sequence $m_{T} \ge A v_{T}^{-2}$, 
  \begin{equation*}
    \sup_{m_{T} \leq k\leq T} \
    \sup_{ \gamma: \|\gamma\| \ge D v_{T}}
    \ell_{k}^{(0)}\big (\gamma \big)
    \le 
    -
    \big |
    O_p
    \big (
    (D v_{T})^2 
    m_{T}
    \big )
    \big |.
  \end{equation*}
\end{property}
\begin{proof}[\textbf{Proof}.]
  Let $D>0$ be fixed.
  We have, for every $1 \le k \le T$,
  \begin{eqnarray*}
    \sup_{\gamma:  \| \gamma \| \ge D v_{T}}
    \frac{1}{k}
    \ell_{k}^{(0)}(\gamma)
    \le 
    \sup_{\gamma:  \| \gamma \| \ge D v_{T}}
    \bigg (
    \frac{1}{k}
    \bigg \|
    \sum_{t=1}^{k}
    \xi_{t}
    \bigg \|
    - 
    D v_{T}
    \bigg )
    \|\gamma \|.
  \end{eqnarray*}
  Lemma \ref{lemma:HRI} yields that, 
  for any $A>0$
  and
  for any $\epsilon>0$,
  \begin{equation}
    \label{eq:HRI-pro5}
    \Pr
    \bigg \{
    \sup_{A v_{T}^{-2}  \le k \le T}
    \frac{
      1
    }{
      k v_{T} 
    }
    \bigg \|
      \sum_{t= 1}^{k}
      \xi_{t T}
    \bigg \| 
    >
    \epsilon
    \bigg \}
    \le
    \frac{C_{1}}{\epsilon^2}
    \bigg ( 
    \frac{1}{A}
    +
    \frac{1}{v_{T}^{2}}
    \sum_{k=Av_{T}^{-2}}^{T} 
    \frac{1}{k^2}
    \bigg).
  \end{equation}
  Because 
  $ \sum_{k=Av_{T}^{-2}}^{T} 
    k^{-2}
    = 
    O\big ( (A v_{T}^{-2})^{-1} \big)
  $,
  we can show that 
  the right-hand side of
  (\ref{eq:HRI-pro5})
  becomes arbitrarily small 
  for a sufficiently large $A>0$.
  Since 
  $\epsilon$ can be arbitrarily small,
  there exists an $A$ such that 
  \begin{eqnarray*}
    \sup_{Av_{T}^{-2}\leq k\leq T} \
    \sup_{ \gamma: \|\gamma\| \ge D v_{T}}
    \frac{
      1
    }{
      k
    }
    \ell_{k}^{(0)}\big (\gamma \big)
    \le 
    - 
    C_{2}
    (D v_{T})^{2}.
  \end{eqnarray*}
  with probability approaching 1.
  The result follows because 
  $-m_{T}^{-1} \le -k^{-1}$
  when $k \ge m_{T}$.
\end{proof}
\vspace{0.1cm}
%%%%%%%%%%%%%%%%%%%%%%%%%%%%%%%%%%%%%%%%%%%%%%%%%%%%%%%%%%%%%%%%%%%%%%%%%%%%%%%%%%%%%%%

%%%%%%%%%%%%%%%%%%%%%%%%%%%%%%%%%%%%%%%%%%%%%%%%%%%%%%%%%%%%%%%%%%%%%%%%%%%%%%%%%%%%%%%

%%%%%%%%%%%%%%%%%%%%%%%%%%%%%%%%%%%%%%%%%%%%%%%%%%%%%%%%%%%%%%%%%%%%%%%%%%%%%%%%%%%%%%%
\begin{property}
  \label{property:optT} 
  Let 
  $\Gamma_{3}(D)
   := 
   \{
   \gamma \in \mathcal{G}: 
   \sqrt{T} \| \gamma \| \le D 
   \}
  $
  for any $D>0$. Then, for any  $\delta \in (0,1)$,\\
  (a)
  there exists a $D>0$ such that 
  \begin{eqnarray*} 
    \sup_{\delta T \leq k\leq T} \ 
    \sup_{\gamma \in \Gamma \setminus \Gamma_{3}(D)}
    \ell_{k}^{(0)}(\gamma)
    \le 
    -
    |O_p(D^2)|,
  \end{eqnarray*}
  (b)
  for any $D>0$,
  \begin{eqnarray*} 
    \sup_{\delta T \leq k\leq T} \ 
    \sup_{\gamma \in \Gamma_{3}(D)}
    \ell_{k}^{(0)}(\gamma)
    =
    O_p(D).
  \end{eqnarray*}
\end{property}
\begin{proof}[\textbf{Proof}.]
  Let 
  $\delta \in (0,1)$ 
  be fixed.
  Then, we have, for every $\delta T \le k \le T$ and for any $D>0$,  
  \begin{eqnarray}
    \label{eq:v-1q}
    \sup_{\gamma \in \Gamma \setminus \Gamma_{3}(D)}
    \ell_{k}^{(0)}(\gamma)
    \le 
    \sup_{\gamma \in \Gamma \setminus \Gamma_{3}(D)}
    \bigg (
    \frac{1}{\sqrt{T}}
    \bigg \|
    \sum_{t=1}^{k}
    \xi_{t}
    \bigg \|
    - 
    \delta D
    \bigg )
    \sqrt{T}
    \|\gamma \|,
  \end{eqnarray}
  and
  \begin{eqnarray}
    \label{eq:v-2q}
    \sup_{\gamma \in  \Gamma_{3}(D)}
    |\ell_{k}^{(0)}(\gamma)|
    \le 
    \frac{1}{\sqrt{T}}
    \bigg \|
    \sum_{t=1}^{k}
    \xi_{t}
    \bigg \|
    D.
  \end{eqnarray}
  Lemma \ref{lemma:HRI}
  implies that 
  $
    \sup_{\delta T \leq k\leq T} 
    \big \|
    \sum_{t=1}^{k}
    \xi_{t}
    \big \|
    = O_p(\sqrt{T})
  $.
  It follows from (\ref{eq:v-1q}) that,
  for some $D>0$,
  $
    \sup_{\delta T \leq k\leq T} 
    \sup_{\gamma \in \Gamma \setminus \Gamma_{3}(D)}
    \ell_{k}^{(0)}(\gamma)
    \le 
    - 
    C_{1} D^2 
  $
  with probability approaching 1,
  while 
  it follows from (\ref{eq:v-2q}) that 
  $
    \sup_{\delta T \leq k\leq T} 
    \sup_{\gamma \in  \Gamma_{3}(D)}
    |\ell_{k}^{(0)}(\gamma)|
    \le 
     C_{2} D
  $
  with probability approaching 1,
  for any $D>0$. 
  Hence, the desired result follows.
\end{proof}
\vspace{0.1cm}
%%%%%%%%%%%%%%%%%%%%%%%%%%%%%%%%%%%%%%%%%%%%%%%%%%%%%%%%%%%%%%%%%%%%%%%%%%%%%%%%%%%%%%%

%%%%%%%%%%%%%%%%%%%%%%%%%%%%%%%%%%%%%%%%%%%%%%%%%%%%%%%%%%%%%%%%%%%%%%%%%%%%%%%%%%%%%%%

%%%%%%%%%%%%%%%%%%%%%%%%%%%%%%%%%%%%%%%%%%%%%%%%%%%%%%%%%%%%%%%%%%%%%%%%%%%%%%%%%%%%%%%
\begin{property}
  \label{property:smallin} 
  For any constant $M>0$ and a deterministic sequence $b_{T}>0$, we have 
  \begin{equation*}
    \sup_{1\leq k\leq M v_{T}^{-2}} 
    \sup_{ \gamma: \| \gamma \| \le b_{T} }
    \ell_{k}^{(0)}\big (\gamma \big )
    =  O_p(M^{1/2} v_{T}^{-1} b_{T} ).
  \end{equation*}
\end{property}
\begin{proof}[\textbf{\textrm{Proof}}.]
  We have that 
  $
    \sup_{1 \le k \le M v_{T}^{-2}}
    \sup_{ \gamma: \| \gamma \| \le b_{T} }
    |\ell_{k}^{(0)}
    \big ( \gamma \big )|
    \leq 
    \sup_{1 \le k \le M v_{T}^{-2}}
    \|
    \sum_{t=1}^{k}
    \xi_{t}
    \|
    b_{T}
  $
  for any $M>0$.
  Lemma \ref{lemma:HRI} yields 
  $\sup_{1 \le k \le M v_{T}^{-2}}
    \|
      \sum_{t=1}^{k}
      \xi_{t}
    \|
    \le 
    O_p\big ( (M v_{T}^{-2})^{1/2} \big )
  $.
\end{proof}
\vspace{0.1cm}
%%%%%%%%%%%%%%%%%%%%%%%%%%%%%%%%%%%%%%%%%%%%%%%%%%%%%%%%%%%%%%%%%%%%%%%%%%%%%%%%%%%%%%%

%%%%%%%%%%%%%%%%%%%%%%%%%%%%%%%%%%%%%%%%%%%%%%%%%%%%%%%%%%%%%%%%%%%%%%%%%%%%%%%%%%%%%%%

%%%%%%%%%%%%%%%%%%%%%%%%%%%%%%%%%%%%%%%%%%%%%%%%%%%%%%%%%%%%%%%%%%%%%%%%%%%%%%%%%%%%%%%

For $\tau_{G, l-1} + 1\le t \le \tau_{Gl}$,
we can show that 
\begin{eqnarray*}
  \| \Psi_{l} \|
  \le 
  \| (\Sigma _{t, \mathcal{K}^{0}}^{0})^{- 1 / 2}\|^2 
  \| \Sigma_{t, \mathcal{K}} -\Sigma_{t, \mathcal{K}^{0}}^{0} \|
  \ \ \ \mathrm{and } \ \ \
  \| \Sigma_{t, \mathcal{K}} -\Sigma_{t, \mathcal{K}^{0}}^{0} \|
  \le 
  \| (\Sigma _{t, \mathcal{K}^{0}}^{0})^{1 / 2}\|^2 
  \| \Psi_{l} \|,
\end{eqnarray*}
Since 
$\| (\Sigma _{t, \mathcal{K}^{0}}^{0})^{1 / 2}\|$
and 
$\| (\Sigma _{t, \mathcal{K}^{0}}^{0})^{-1 / 2}\|$
are bounded and 
$\Vert \Psi_{l}  \Vert = \max_{1 \le i \le n}|\lambda _{il}^{\Psi}|$,
we have 
\begin{eqnarray*}
  d_{1}
  \| \Sigma_{t, \mathcal{K}} -\Sigma_{t, \mathcal{K}^{0}}^{0} \|
  \le 
  \max_{1 \le i \le n}|\lambda _{il}^{\Psi}|
  \le 
  d_{2}
  \| \Sigma_{t, \mathcal{K}} -\Sigma_{t, \mathcal{K}^{0}}^{0} \|,
\end{eqnarray*}
for some constants $d_{1}, d_{2} >0$.
This relation will be used when we restrict the space for the 
covariance matrix of the error.
The next proposition presents a result about 
the break date estimates.

%%%%%%%%%%%%%%%%%%%%%%%%%%%%%%%%%%%%%%%%%%%%%%%%%%%%%%%%%%%%%%%%%%%%%%%%%%%%%%%%%%%%%%%

%%%%%%%%%%%%%%%%%%%%%%%%%%%%%%%%%%%%%%%%%%%%%%%%%%%%%%%%%%%%%%%%%%%%%%%%%%%%%%%%%%%%%%%

%%%%%%%%%%%%%%%%%%%%%%%%%%%%%%%%%%%%%%%%%%%%%%%%%%%%%%%%%%%%%%%%%%%%%%%%%%%%%%%%%%%%%%%
\vspace{0.1cm}
\begin{proposition}
  \label{pro:date1} 
  Under Assumptions A1-A5,
  there exists a $B>0$ such that 
  \begin{eqnarray*}
    \lim_{T \to \infty}
    \Pr
    \big \{
    \big |
     \hat{k}_{gj}- k_{gj}^{0}
    \big |
      >
     B v_{T}^{-2} \log T
    \big \} 
    =
    0,
  \end{eqnarray*}
  for every $(g, j) \in \{1,...,G\} \times \{1,...,m \}$.
\end{proposition}
\begin{proof}[\textbf{\textrm{Proof}}.]
  For a constant $B>0$, define
  \begin{eqnarray*}
    \ddot{\Xi}(B)
    :=
    \Big \{
    \mathcal{K}
    \in \Xi_{\nu}
    :
    \max_{1 \le g \le G}
    \max_{1 \le j \le m}
    |k_{gj} - k_{gj}^{0} | \le B v_{T}^{-2} \log T
    \Big \}.
  \end{eqnarray*}
  To prove the assertion, we shall show that, 
  for a sufficiently large $B>0$, 
  \begin{eqnarray}
    \label{eq:suff-ProA1}
    \lim_{T \to \infty}
    \Pr
    \bigg \{
    \sup_{ (\mathcal{K}, \theta) \in  \Xi_{\nu} \setminus  \ddot{\Xi}(B)  \times {\Theta}}
    \ell_{T}
    \big (
    K,\theta 
    \big )
    \ge 0
    \bigg \} = 0.
  \end{eqnarray}
  Since the normalized log likelihood evaluated at the maximum likelihood estimates 
  should be non-negative, the desired conclusion follows from 
  (\ref{eq:suff-ProA1}).

  %%%%%%%%%
  % UP
  %%%%%%%%%
  To show (\ref{eq:suff-ProA1}),
  we examine the upper bound in Lemma \ref{lemma:decomposition-l}
  given sets of break dates $\mathcal{K}\not \in  \ddot{\Xi}(B)$ 
  and $\mathcal{K}^{0}$. 
  First, 
  observe that 
  Property \ref{property:opt1} provides a
  not necessarily sharp but general upper bound in probability 
  and 
  that 
  the parameter space is bounded.
  Thus,
  \begin{eqnarray}
    \label{eq:bound-Pro1-1}
    \sup_{ (\mathcal{K}, \theta) \in  \Xi_{\nu} \setminus  \ddot{\Xi}(B)  \times {\Theta}}
    \bar{\ell}_{g, l}(\mathcal{K}, \theta)
    \le |O_P(\log T)|
    \ \ \ \mathrm{and} \ \ \
    \sup_{ (\mathcal{K}, \theta) \in  \Xi_{\nu} \setminus  \ddot{\Xi}(B)  \times {\Theta}}
    \Delta(\mathcal{K}, \theta) \le C_{1},   
  \end{eqnarray}
  for every $1 \le g \le G+1$ and $1 \le l \le 2(m{+}1)$.

  %%%%%%%%%
  % out
  %%%%%%%%%
  Next, for $\mathcal{K} \not \in  \ddot{\Xi}(B)$,
  there exits a pair $(g, j) \in \{1,...,G\} \times \{1,...,m \}$
  such that 
  some neighborhood 
  $\mathcal{N}_{gj}:=\{t \in [1,T]: |t - k_{gj}^{0} | \le B v_{T}^{-2} \log T \}$ 
  of a true break date, $k_{gj}^{0}$,
  contains none of the break dates $\mathcal{K}_{g}$ of the $g^{th}$ group, i.e.,
  $ \mathcal{K}_{g} 
   \not\subset 
   \mathcal{N}_{gj}
  $.  
  This implies that 
  there is a
  $\tau_{gl} = k_{gj}^{0}$
  with 
  a union of sub-intervals 
  \begin{eqnarray*}
    [\tau_{g,l-1}{+}1, \tau_{gl}]
    \cup
    [\tau_{gl}{+}1, \tau_{g,l+1}]
    \ \ \mathrm{with} \ \
    \min_{l \le j \le l+1}
    (\tau_{gj}- \tau_{g,j-1}) \ge B v_{T}^{-2} \log T.
  \end{eqnarray*}
  Since 
  $ \mathcal{K}_{g} 
   \not\subset 
   (\tau_{g,l-1}, \tau_{g,l+1})
  $,  
  the $g^{th}$ group estimates are constant 
  for $\tau_{g,l-1}+1 \le t \le \tau_{g,l+1}$
  and 
  both 
  $\bar{\ell}_{g,l}(\mathcal{K}, \theta)$
  and 
  $\bar{\ell}_{g,l+1}(\mathcal{K}, \theta)$
  depend on the same $g^{th}$ group estimates.  
  Note that the triangle inequality yields that 
  \begin{eqnarray*}
    C_{2} v_{T}
    \le 
    2 \max
    \big \{ 
    \big \|
    {\beta}_{g, \tau_{g,l+1}, \mathcal{K}} 
    -
    \beta_{g, \tau_{gl}, \mathcal{K}^{0}}^{0} 
    \big \|,
    \big \|
    {\beta}_{g, \tau_{g,l+1}, \mathcal{K}} 
    -
    \beta_{g, \tau_{g,l+1}, \mathcal{K}^{0}}^{0} 
    \big \|
    \big \},
  \end{eqnarray*}
  and additionally when $g = G$, 
  \begin{eqnarray*}
    C_{3} v_{T}
    \le 
    2 \max
    \big \{ 
    \big \|
    {\Sigma}_{\tau_{G,l+1}, \mathcal{K}} 
    -
    \Sigma_{\tau_{Gl}, \mathcal{K}^{0}}^{0} 
    \big \|,
    \big \|
    {\Sigma}_{\tau_{G,l+1}, \mathcal{K}} 
    -
    \Sigma_{\tau_{G,l+1}, \mathcal{K}^{0}}^{0} 
    \big \|
    \big \}.
  \end{eqnarray*}
  This implies that 
  either 
  $\bar{\ell}_{g,l}(\mathcal{K}, {\theta})$
  or
  $\bar{\ell}_{g,l+1}({\mathcal{K}}, {\theta})$
  satisfies the condition in Property \ref{property:smallout} 
  with $m_{T} = B v_{T}^{-2} \log T$,
  which together with 
  (\ref{eq:bound-Pro1-1})
  implies that, for a sufficiently large $B$,
  \begin{eqnarray*}
    \sup_{ (\mathcal{K}, \theta) \in  \Xi_{\nu} \setminus  \ddot{\Xi}(B)  \times {\Theta}}
    {\ell}_{T}(\mathcal{K}, {\theta})     
    \le 
    - |O_p( B \log T)|
    + O_p(\log T).
  \end{eqnarray*}
  This yields (\ref{eq:suff-ProA1}) 
  and 
  thus completes the proof.
\end{proof}
\vspace{0.1cm}
%%%%%%%%%%%%%%%%%%%%%%%%%%%%%%%%%%%%%%%%%%%%%%%%%%%%%%%%%%%%%%%%%%%%%%%%%%%%%%%%%%%%%%%

%%%%%%%%%%%%%%%%%%%%%%%%%%%%%%%%%%%%%%%%%%%%%%%%%%%%%%%%%%%%%%%%%%%%%%%%%%%%%%%%%%%%%%%

%%%%%%%%%%%%%%%%%%%%%%%%%%%%%%%%%%%%%%%%%%%%%%%%%%%%%%%%%%%%%%%%%%%%%%%%%%%%%%%%%%%%%%%

%%%%%%%%%%%%%%%%%%%%%%%%%%%%%%%%%%%%%%%%%%%%%%%%%%%%%%%%%%%%%%%%%%%%%%%%%%%%%%%%%%%%%%%

%%%%%%%%%%%%%%%%%%%%%%%%%%%%%%%%%%%%%%%%%%%%%%%%%%%%%%%%%%%%%%%%%%%%%%%%%%%%%%%%%%%%%%%
 
%%%%%%%%%%%%%%%%%%%%%%%%%%%%%%%%%%%%%%%%%%%%%%%%%%%%%%%%%%%%%%%%%%%%%%%%%%%%%%%%%%%%%%%
\begin{proposition}
  \label{proposition:sqrN} 
  Suppose that  Assumptions A1-A5 hold. Then,
  \begin{eqnarray*}
    \hat{\beta}_{gj}-\beta _{gj}^{0} 
    = 
    o_p(v_{T})
    \ \ \ \ \mathrm{and}  \ \ \ \
    \hat{\Sigma}_{j}-\Sigma _{j}^{0}
    =
    o_p(v_{T}),
  \end{eqnarray*}
  for every $(g, j) \in \{1,...,G\} \times \{1,...,m +1\}$.
\end{proposition}
\begin{proof}[\textbf{\textrm{Proof}}.]
  Let $\epsilon>0$ be fixed
  and 
  define a subset of the parameter space ${\Theta}$:
  \begin{eqnarray*}
    \ddot{\Theta}(\epsilon)
    :=
    \Big \{
     \theta
     \in 
     {\Theta}
     :
     \max_{1 \le g \le G}
     \max_{1 \le j \le m+1}
     \|\beta _{gj}-\beta _{gj}^{0}\| \le \epsilon v_{T}
     \ \mathrm{and} \ 
     \max_{1 \le j \le m+1}
     \|\Sigma_{j}-\Sigma _{j}^{0}\| \le \epsilon v_{T}
    \Big \}.
  \end{eqnarray*}
  Proposition \ref{pro:date1}
  shows that 
  the break date estimates 
  $\hat{\mathcal{K}}$
  are included in 
  $\ddot{\Xi}(B)$
  with probability approaching 1
  for a sufficiently large $B$
  and thus we consider the case
  where 
  $\mathcal{K} \in \ddot{\Xi}(B)$.
  For 
  $\theta \in {\Theta} \setminus \ddot{\Theta}(\epsilon)$,
  there exists  
  a pair $(g, j) \in \{1,...,G\} \times \{1,...,m \}$
  such that 
  either 
  \begin{eqnarray}
    \label{eq:out-e}
    \|{\beta} _{gj}-\beta _{gj}^{0}\| \ge \epsilon v_{T}
    \ \ \ \mathrm{or} \ \ \ 
    \|{\Sigma}_{j}-\Sigma _{j}^{0}\| \ge  \epsilon v_{T}.
  \end{eqnarray}
  Observe that 
  ${k}_{gj} - {k}_{g, j-1} \ge \nu T $
  and 
  $k_{gj}^{0} - k_{g, j-1}^{0} \ge \nu T $,
  while 
  $|{k}_{gj} - k_{gj}^{0}| \le B v_{T}^{-2} \log T$. 
  For some $l \in\{1, \dots, N_{g}\}$,
  we have 
  $\tau_{g, l-1} 
   = 
   \max
   \{
   {k}_{g,j-1},
   k_{g,j-1}^{0}
   \}
  $
  and 
  $\tau_{g l} 
   = 
   \max
   \{
   {k}_{g,j},
   k_{gj}^{0}
   \}
  $
  satisfying 
  $\tau_{g l} - \tau_{g, l-1} \ge \delta T$
  for some $\delta \in (0,1)$
  and that 
  (\ref{eq:out-e}) holds over a sub-interval
  $[\tau_{g, l-1}+1, \tau_{g l}]$.
  Thus, Property \ref{property:smallout} 
  with $m_{T} = \delta T$
  implies that 
  \begin{eqnarray*}
    \sup_{ 
      ( \mathcal{K}, \theta) 
      \in 
      \ddot{\Xi}(B) 
      \times
      {\Theta} \setminus \ddot{\Theta}(\epsilon)
    }
    \bar{\ell}_{g, l}(\mathcal{K}, \theta)
    \le 
    - |O_p(\epsilon^2 T v_{T}^{2})|.
  \end{eqnarray*}
  For the other sub-intervals, 
  Property \ref{property:opt1} provides an upper bound 
  of order 
  $|O_p(\log T)|$.  
  Since $\sqrt{T} v_{T}/ \log T \to \infty$ as $T \to \infty$, 
  we can show that 
  \begin{eqnarray*}
    \sup_{ 
      ( \mathcal{K}, \theta) 
      \in 
      \ddot{\Xi}(B) 
      \times
      {\Theta} \setminus \ddot{\Theta}(\epsilon)
    }
    \ell_{T}
    \big (\mathcal{K}, \theta \big)
    \le 
    - |O_p(\epsilon^2 T v_{T}^2)|.
  \end{eqnarray*}  
  This leads to the desired result.
\end{proof}
\vspace{0.1cm}
%%%%%%%%%%%%%%%%%%%%%%%%%%%%%%%%%%%%%%%%%%%%%%%%%%%%%%%%%%%%%%%%%%%%%%%%%%%%%%%%%%%%%%%

%%%%%%%%%%%%%%%%%%%%%%%%%%%%%%%%%%%%%%%%%%%%%%%%%%%%%%%%%%%%%%%%%%%%%%%%%%%%%%%%%%%%%%%

%%%%%%%%%%%%%%%%%%%%%%%%%%%%%%%%%%%%%%%%%%%%%%%%%%%%%%%%%%%%%%%%%%%%%%%%%%%%%%%%%%%%%%%

Propositions \ref{pro:date1} and  \ref{proposition:sqrN} are important intermediate steps to establish the convergence rates of the estimates as stated in the theorem below. 
A similar approach was used by \cite{BaiLumsdaineStock1998}, \cite{Bai2000AEF}
and \cite{QuPerron2007Emtca}
when break dates are assumed to either have a common location or be asymptotically distinct. 
A key difference between their approach and ours is that 
we allow for the possibility that the break dates associated with different basic parameters may not be asymptotically distinct.

%%%%%%%%%%%%%%%%%%%%%%%%%%%%%%%%%%%%%%%%%%%%%%%%%%%%%%%%%%%%%%%%%%%%%%%%%%%%%%%%%%%%%%%

%%%%%%%%%%%%%%%%%%%%%%%%%%%%%%%%%%%%%%%%%%%%%%%%%%%%%%%%%%%%%%%%%%%%%%%%%%%%%%%%%%%%%%%

%%%%%%%%%%%%%%%%%%%%%%%%%%%%%%%%%%%%%%%%%%%%%%%%%%%%%%%%%%%%%%%%%%%%%%%%%%%%%%%%%%%%%%%
\vspace{0.1cm}
\begin{proof}[\textbf{\textrm{Proof of Theorem 1}}.]
  %%%%%%%%%%
  % (a)
  %%%%%%%%%%
  \textbf{(a)}
  Proposition \ref{pro:date1} shows that 
  $\hat{\mathcal{K}}
   \in 
   \ddot{\Xi}(B)
  $
  with probability approaching 1
  for some $B>0$,
  while  
  both $\hat{\mathcal{K}}$ and $\mathcal{K}^{0}$ are 
  included in $\Xi_{\nu}$. 
  Thus, 
  it suffices to consider the case where 
  either 
  $\tau_{gl} - \tau_{g, l-1} \ge \delta T$
  for some $\delta>0$
  or 
  $\tau_{gl} - \tau_{g, l-1} \le  B v_{T}^{-2} \log T$
  for every 
  $(g, l) \in \{ 1, \dots, G\} \times \{1, \dots, N\}$.
  If $\tau_{gl} - \tau_{g, l-1} \ge \delta T$, then 
  Property \ref{property:optT} implies that 
  \begin{eqnarray}
    \label{eq:th-1}
    \sup_{(\mathcal{K}, \theta) \in \ddot{\Xi}(B)\times {\Theta}}
    \bar{\ell}_{g,l}({\mathcal{K}}, {\theta}) \le 
    |O_p(1)|.
  \end{eqnarray}
  When 
  $\tau_{gl} - \tau_{g, l-1} \le  B v_{T}^{-2} \log T$,
  there are two cases:
  $M v_{T}^{-2} \le \tau_{gl} - \tau_{g, l-1} \le  B v_{T}^{-2} \log T$
  and 
  $ \tau_{gl} - \tau_{g, l-1} \le  M v_{T}^{-2}$
  for some $M>0$.
  For sake of concreteness, 
  let 
  $\tau_{g, l-1} = k_{gj}^{0}$
  and 
  $\tau_{gl} = \hat{k}_{gj} $
  in both cases. 
  When $k_{gj}^{0}+1 \le t \le \hat{k}_{gj}$,
  we have 
  $
  (
  \hat{\beta}_{g, t, \hat{\mathcal{K}}},
  \beta_{g, t, \mathcal{K}^{0}}^{0}
  )
  = 
  (
  \hat{\beta}_{gj}, 
  \beta_{g,j+1}^{0}
  )
  $
  for $1 \le g \le G$
  and 
  $
  (
  \hat{\Sigma}_{t, \hat{\mathcal{K}}},
  \Sigma_{t, \mathcal{K}^{0}}^{0}
  )
  = 
  (
  \hat{\Sigma}_{j}, 
  \Sigma_{j+1}^{0}
  )
  $
  for $g = G$.
  Since 
  $\|\beta _{g,j+1}^{0}-\beta _{gj}^{0}\| = v_{T} \|\delta _{gj}\|$
  and 
  $
    \| {\Sigma}_{j+1}^{0} - \Sigma_{j}^{0} \|
    =
    v_{T} \|\Phi_{j}\|
  $, 
  we can show\footnote{
    To prove this, 
    we use the inequality, 
    $
      \| a - b \| - \| b - c\|
      \le 
      \|a- c \| 
      \le 
      \| a - b \| + \| b - c\|
    $
    for any elements $a$, $b$ and $c$ 
    of some space with the norm $\|\cdot \|$,
    which is due to 
    the triangle inequality.
  } 
  \begin{eqnarray*}
    \Big |
    \| \hat{\beta}_{gj}-\beta _{g,j+1}^{0} \|
    -
    v_{T} \|\delta_{gj} \|
    \Big |
    \le 
    \| \hat{\beta} _{gj}-\beta_{gj}^{0} \|
     \ \ \mathrm{and} \ \
    \Big | 
    \| \hat{\Sigma}_{j}-\Sigma _{j+1}^{0} \|
    -
    v_{T} \|\Phi_{j} \|
    \Big |
    \le 
    \| \hat{\Sigma}_{j} - \Sigma_{j}^{0} \|.
  \end{eqnarray*}
  Moreover,
  Proposition \ref{proposition:sqrN} 
  shows that 
  $
    \|  \hat{\beta}_{gj}-\beta _{gj}^{0} \| = o_p(v_{T})
  $
  and 
  $
    \| \hat{\Sigma}_{j} - \Sigma_{j}^{0} \|
    = o_p(v_{T})
  $.
  Thus,
  \begin{eqnarray}
    \label{eq:diff-B-1}
    \| \hat{\beta}_{gj}-\beta _{g,j+1}^{0} \|
    =
    v_{T} \|\delta_{gj} \|
    +o_p(v_{T})
     \ \ \mathrm{and} \ \
    \| \hat{\Sigma}_{j}-\Sigma _{j+1}^{0} \|
    =
    v_{T} \|\Phi_{j} \|
    + o_p(v_{T}).
  \end{eqnarray}
  When 
  $M v_{T}^{-2} \le \tau_{gl} - \tau_{g, l-1} \le  B v_{T}^{-2} \log T$,
  Property \ref{property:smallout} together with (\ref{eq:diff-B-1})
  implies that 
  \begin{eqnarray}
    \label{eq:th-2}
    \bar{\ell}_{g,l}(\hat{\mathcal{K}}, \hat{\theta}) \le 
    - |O_p(M)|,
  \end{eqnarray}
  for a sufficiently large $M$,
  while, 
  for $ \tau_{gl} - \tau_{g, l-1} \le  M v_{T}^{-2}$, 
  Property \ref{property:smallin} 
  with (\ref{eq:diff-B-1})
  implies 
  \begin{eqnarray}
    \label{eq:th-3}
    \bar{\ell}_{g,l}(\hat{\mathcal{K}}, \hat{\theta}) \
    =
    O_p(M^{1/2}).
  \end{eqnarray}
  Since 
  $\sup_{(\mathcal{K}, \theta) \in \ddot{\Xi}(B) \times \ddot{\Theta}({\epsilon})}
   \Delta(\mathcal{K}, \theta)
   = o(1)
  $,
  Lemma \ref{lemma:decomposition-l}
  with
  (\ref{eq:th-1}),
  (\ref{eq:th-2})
  and 
  (\ref{eq:th-3}) 
  implies 
  \begin{eqnarray*}
    \sup_{(\mathcal{K}, \theta) \in \ddot{\Xi}(B) \setminus \bar{\Xi}_{M} \times \ddot{\Theta}({\epsilon})}
    \ell_{T}(\mathcal{K}, \theta) < 
    - |O_p(M)|,
  \end{eqnarray*}
  for a sufficiently large $M$.
  This completes the proof of part (a).

  %%%%%%%%%%
  % (b)
  %%%%%%%%%%
  \textbf{(b)}
  From part (a), there exists an $M>0$ such that 
  $\max_{1 \le g \le G}
   \max_{1 \le j \le m}
   |\hat{k}_{gj}-k_{gj}^{0}| \le M v_{T}^{-2}$
  with probability approaching 1.
  Thus it suffices to consider the case where
  either 
  $\tau_{gl} - \tau_{g,l-1} \le M v_{T}^{-2}$
  or 
  $\tau_{gl} - \tau_{g,l-1} > \delta T$
  for some $\delta>0$. 
  As in (\ref{eq:th-1}) and (\ref{eq:th-3}), 
  we can show that 
  $\bar{\ell}_{g,l}(\hat{\mathcal{K}}, \hat{\theta})$
  is bounded by a term of order $|O_p(1)|$
  for every 
  $(g, l) \in 
    \{1, \dots, G+1\}
    \times 
    \{1, \dots, 2(m+1)\}
  $.
  If 
  $\sqrt{T}
   \|\hat{\beta}_{gj} - \beta_{gj}^{0}\|
   \ge 
   M
  $
  for some group and regime $(g,j)$
  and 
  for some $M>0$,
  then 
  there is a corresponding sub-interval 
  $[\tau_{g,l-1}+1, \tau_{gl}]$
  with 
  $\tau_{gl} - \tau_{g,l-1} > \delta T$
  and thus
  Property \ref{property:optT}(a) implies that 
  $\bar{\ell}_{g,l}(\hat{\mathcal{K}}, \hat{\theta}) \le  - |O_p(M^2)|$
  for a sufficiently large $M$. 
  Thus, on the event 
  that 
  $
   \max_{1 \le g \le G}
   \max_{1 \le j \le m+1}
   \|\hat{\beta}_{gj} - \beta_{gj}^{0}\|
   \ge 
   M T^{-1/2}
  $
  for a sufficiently large $M$, 
  Lemma \ref{lemma:decomposition-l} implies that 
  the normalized log likelihood takes negative value 
  with probability approaching 1.
  The same result holds 
  when 
  $  
   \max_{1 \le j \le m+1}
   \|\hat{\Sigma}_{j} - \Sigma_{j}^{0}\|
   \ge 
   M T^{-1/2}
  $
  for a sufficiently large $M$.
\end{proof}
\vspace{0.1cm}
%%%%%%%%%%%%%%%%%%%%%%%%%%%%%%%%%%%%%%%%%%%%%%%%%%%%%%%%%%%%%%%%%%%%%%%%%%%%%%%%%%%%%%%%%%%%

%%%%%%%%%%%%%%%%%%%%%%%%%%%%%%%%%%%%%%%%%%%%%%%%%%%%%%%%%%%%%%%%%%%%%%%%%%%%%%%%%%%%%%%%%%%%

Having established the convergence rates of the estimates, we are now in a position to 
prove results about the asymptotic independence of the break date estimates 
and the estimates of the basic parameters.
In order to proceed, we let the likelihood based on the observations in
the interval $[t_{1},t_{2}] \subset [1, T]$ 
be denoted as
$
  L(t_{1},t_{2}; \mathcal{K}, \theta)=
  \prod_{t=t_{1}}^{t_{2}}f(y_{t}|X_{tT}, \theta _{t,\mathcal{K}}).
$
Then, using the partition $\{[\tau_{l-1}+1, \tau_{l}]\}_{l=1}^{N}$
of an interval $[1,T]$ given $\mathcal{K}$ and $\mathcal{K}^{0}$,
we can express the normalized log likelihood as
\begin{eqnarray*}
  \ell_{T}(\mathcal{K},\theta) 
  =
  \sum_{l=1}^{N}
  \big \{
  \log L(\tau_{l-1}+1,\tau_{l}; \mathcal{K},\theta )  
  -
  \log L(\tau_{l-1}+1,\tau_{l}; \mathcal{K}^{0},\theta^{0})  
  \big \}.
\end{eqnarray*}

%%%%%%%%%%%%%%%%%%%%%%%%%%%%%%%%%%%%%%%%%%%%%%%%%%%%%%%%%%%%%%%%%%%%%%%%%%%%%%%%%%%%%%%%%%%%
\vspace{0.1cm}
\begin{proof}[\textbf{\textrm{Proof of Theorem 2}}.]
  Consider the case where
  $(\mathcal{K}, \theta) \in  \bar\Xi_{M} \times \bar{\Theta}_{M}$
  for a sufficiently large $M$
  with the restriction $R(\theta) =0$.
  By definition, we can write 
  \begin{eqnarray}
    \notag
    &&
    \ell_{T}(\mathcal{K},\theta) 
    -
    \ell_{T}(\mathcal{K}^{0}, \theta)
    -
    \ell_{T}(\mathcal{K}, \theta^{0}) \\
    && 
    \hspace{1cm} 
    \label{eq:vv-1}
    =
    \sum_{l=1}^{N}
    \big \{
    \log L(\tau_{l-1}+1,\tau_{l};\mathcal{K},\theta )
    -
    \log L(\tau_{l-1}+1,\tau_{l};\mathcal{K}^{0},\theta)
    \big \} \\
    && 
    \hspace{1.5cm} 
    \label{eq:vv-2}
    -
    \sum_{l=1}^{N}
    \big \{
    \log L(\tau_{l-1}+1,\tau_{l}; \mathcal{K},\theta^{0})
    -
    \log L(\tau_{l-1}+1,\tau_{l}; \mathcal{K}^{0},\theta^{0})
    \big \}.
  \end{eqnarray}
  If
  $\tau_{l}  - \tau_{l-1} > M v_{T}^{-2}$, 
  then we have
  $
   \theta_{t, \mathcal{K}} 
   = 
   \theta_{t, \mathcal{K}^{0}} 
  $
  and  
  $
   \theta_{t, \mathcal{K}}^{0}  
   = 
   \theta_{t, \mathcal{K}^{0}}^{0} 
  $
  for all 
  $\tau_{l-1}+1 \le t \le  \tau_{l}$.
  % sub-intervals $\{[\tau_{l-1}+1, \tau_{l}]\}_{l=1}^{N}$
  % satisfy 
  % either 
  % $\tau_{l}  - \tau_{l-1} \le M v_{T}^{-2}$
  % or
  % $\tau_{l}  - \tau_{l-1} \ge \nu T$.
  Thus,  
  it suffices to consider 
  the quantities
  in
  (\ref{eq:vv-1})
  and
  (\ref{eq:vv-2})
  with the index 
  $l$ satisfying $\tau_{l} - \tau_{l-1} \le M v_{T}^{-2}$.
  Property \ref{property:smallin} 
  with $b_{T} = M T^{-1/2}$ 
  implies that, uniformly in  
  $(\mathcal{K}, \theta) \in  \bar\Xi_{M} \times \bar{\Theta}_{M}$,
  \begin{eqnarray*} 
    \ell_{T}(\mathcal{K},\theta) 
    =
    \ell_{T}(\mathcal{K}, \theta^{0})
    +
    \ell_{T}(\mathcal{K}^{0},\theta) 
    +
    O_p
    \big ( (\sqrt{T} v_{T})^{-1} \big).
  \end{eqnarray*} 
  Hence, we obtain the desired result.
\end{proof}
\vspace{0.1cm}
%%%%%%%%%%%%%%%%%%%%%%%%%%%%%%%%%%%%%%%%%%%%%%%%%%%%%%%%%%%%%%%%%%%%%%%%%%%%%%%%%%%%%%%%%%%%

%%%%%%%%%%%%%%%%%%%%%%%%%%%%%%%%%%%%%%%%%%%%%%%%%%%%%%%%%%%%%%%%%%%%%%%%%%%%%%%%%%%%%%%%%%%%

To derive the limit distribution of the test, we first present a technical lemma, which 
is a direct consequence of Lemma A.1(b).
To this end, we introduce some notation.
For $j=1, \dots, m$,
we define, for $s<0$,
\begin{eqnarray*}
  V_{T, z\eta,j}^{(1)} (-s) 
  :=
  v_{T}
  \sum_{t=T_{j}^{0}+[s v_{T}^{-2}]}^{T_{j}^{0}}
  (z_{t} \otimes \eta_{t})
  \ \ \ \mathrm{and} \ \ \ 
  V_{T, \eta\eta,j}^{(1)} (-s) 
  :=
  v_{T}
  \sum_{t=T_{j}^{0}+[s v_{T}^{-2}]}^{T_{j}^{0}}
  ( \eta _{t}\eta _{t}{}^{\prime }-I_n) ,
\end{eqnarray*}
and, for $s > 0$,
\begin{eqnarray*}
  V_{T, z\eta,j}^{(2)} (s) 
  :=
  v_{T}
  \sum_{t=T_{j}^{0}}^{T_{j}^{0}+[s v_{T}^{-2}]}
  (z_{t} \otimes \eta_{t} )
  \ \ \ \mathrm{and} \ \ \, 
  V_{T, \eta\eta,j}^{(2)} (s) 
  :=
  v_{T}\sum_{t=T_{j}^{0}}^{T_{j}^{0}+[s v_{T}^{-2}]}
  ( \eta _{t}\eta _{t}^{\prime }-I_n ).
\end{eqnarray*}

%%%%%%%%%%%%%%%%%%%%%%%%%%%%%%%%%%%%%%%%%%%%%%%%%%%%%%%%%%%%%%%%%%%%%%%%%%%%
\begin{lemma}
  \label{lemma:add-1}
Under Assumptions A6-A9 with a sequence $v_{T}$ defined in Assumption A4,
we have, for $j=1, \dots, m$,
\begin{eqnarray*}
  V_{T, z\eta,j}^{(1)} (\cdot) 
  \Rightarrow 
  \mathbb{V}_{z\eta,j}^{(1)} (\cdot) 
  \ \ \ \mathrm{and} \ \ \
  V_{T, z\eta,j}^{(2)} (\cdot) 
  \Rightarrow 
  \mathbb{V}_{z\eta,j}^{(2)} (\cdot), 
\end{eqnarray*}
where the weak convergence is in the space $D[0,\infty )^{nq}$ and 
the Brownian motions $\mathbb{V} _{z\eta,j}^{(1)}(\cdot) $ and $\mathbb{V}_{z\eta,j}^{(2)}(\cdot) $ 
are defined in the main text. 
Furthermore, for $j=1, \dots, m$,
\begin{eqnarray*}
  V_{T, \eta\eta,j}^{(1)} (\cdot) 
  \Rightarrow 
  \mathbb{V} _{\eta \eta, j}^{(1)} (\cdot)
  \ \ \mathrm{and} \ \
  V_{T, \eta\eta,j}^{(2)} (\cdot) 
  \Rightarrow 
  \mathbb{V}_{\eta\eta, j}^{(2)}(\cdot),
\end{eqnarray*}
where the weak convergence is in the space $D[0,\infty )^{n^{2}}$ and 
the $n\times n$ matrices 
$\mathbb{V}_{\eta\eta, j}^{(1)}( \cdot )$ and 
$\mathbb{V}_{\eta\eta, j}^{(2)}( \cdot )$ are Brownian motion defined in the main text.
\end{lemma}
\vspace{0.1cm}
%%%%%%%%%%%%%%%%%%%%%%%%%%%%%%%%%%%%%%%%%%%%%%%%%%%%%%%%%%%%%%%%%%%%%%%%%%%%%%%%%%%%%%%%%%%%

%%%%%%%%%%%%%%%%%%%%%%%%%%%%%%%%%%%%%%%%%%%%%%%%%%%%%%%%%%%%%%%%%%%%%%%%%%%%%%%%%%%%%%%%%%%%

%%%%%%%%%%%%%%%%%%%%%%%%%%%%%%%%%%%%%%%%%%%%%%%%%%%%%%%%%%%%%%%%%%%%%%%%%%%%%%%%%%%%%%%%%%%%
\begin{proof}[\textbf{\textrm{Proof of Lemma 1}}.]
  Consider a regime $j \in \{1, \dots, m\}$.
  For 
  $s \in \mathbb{R}$
  and 
  for $\underline{T}_{j}^{0}(s) \le t \le \overline{T}_{j}^{0}(s)$,
  observe that
  \begin{eqnarray*}
    (\Sigma_{t,j+\mathbbm{1}_{\{T_{j}(r) \le t  \} }}^{0} )^{-1}
    =
    \left \{
      \begin{array}{ll}
        (\Sigma_{j+1}^{0})^{-1},
        &
        \mathrm{if} \ T_{j}(r) \le \underline{T}_{j}^{0}(s)\\
        (\Sigma_{j+1}^{0})^{-1}
        -
        \mathbbm{1}_{ \{ T_{j}^{0} < t \le T_{j}(r)\}}
        \{ (\Sigma_{j+1}^{0})^{-1} - (\Sigma_{j}^{0}))^{-1} \},
        & \mathrm{if} \ T_{j}^{0} < T_{j}(r) \le T_{j}^{0}(s)\\
        (\Sigma_{j}^{0})^{-1}
        \hspace{0.4cm} +
        \mathbbm{1}_{ \{ T_{j}(r) < t \le T_{j}^{0} \}}
        \{ 
        (\Sigma_{j+1}^{0})^{-1} - (\Sigma_{j}^{0})^{-1}
        \},
        & \mathrm{if} \ T_{j}^{0}(s) < T_{j}(r) \le T_{j}^{0}\\
        (\Sigma_{j}^{0})^{-1},
        &
        \mathrm{if} \ \overline{T}_{j}^{0}(s) \le T_{j}(r) , 
      \end{array} 
    \right . 
  \end{eqnarray*}
  which yields
  \begin{eqnarray*}
    (\Sigma_{t,j+\mathbbm{1}_{\{T_{j}(r) \le t  \} }}^{0} )^{-1}
    =
    (\Sigma_{j+\mathbbm{1}_{ \{r \le s \}}}^{0})^{-1}
    -
    \mathrm{sgn}(r)
    \mathbbm{1}_{ \{ |r| \le |s|\} }
    \{
     (\Sigma_{j+1}^{0})^{-1}
     -
     (\Sigma_{j}^{0})^{-1}
    \}.
  \end{eqnarray*}
  %%%%%%%%
  % B
  %%%%%%%%
  Let
  $
    D_{T,j}(s)
    := 
    v_{T}^{2}
    \sum_{t= \underline{T}_{j}^{0}(s)+1}^{\overline{T}_{j}^{0}(s)}
    x_{tT}x_{tT}' 
  $.
  We have,
  for every $\underline{T}_{j}^{0}(s) \le t \le \overline{T}_{j}^{0}(s)$
  and 
  for $r \in \mathbb{R}$,
  \begin{eqnarray*}
    B_{T,j}(s, r)
    &=&
    S'
    D_{T,j}(s) 
    \otimes 
    (\Sigma_{j+\mathbbm{1}_{ \{ r \le s \}}}^{0})^{-1} S\\
    &&
    -
    \mathrm{sgn}(r)
    \mathbbm{1}_{ \{|r| \le |s|\} }
    S'
    D_{T,j}(r)
     \otimes
     \{
     (\Sigma_{j+1}^{0})^{-1}
     -
     (\Sigma_{j}^{0})^{-1}
     \}S,
  \end{eqnarray*}
  since 
  $
    X_{tT}
    (\Sigma_{t,j+\mathbbm{1}_{\{T_{j}(r) \le t  \} }}^{0} )^{-1}
    X_{tT}^{\prime}
    =
    S^{\prime }
    x_{tT}x_{tT}^{\prime }\otimes 
    (\Sigma_{t,j+\mathbbm{1}_{\{T_{j}(r) \le t  \} }}^{0} )^{-1}
    S
  $,
  and also
  \begin{eqnarray}
    \label{eq:approx-asym}
    \varphi(t/T) 
    = 
    \varphi(\lambda_{j}^{0}) 
    + 
    O \big (  (\sqrt{T}v_{T})^{-2} \big)    
    \ \ \ \ \mathrm{and} \ \ \ \
    w_{t} 
    = 
    w_{T_{j}^{0}} 
    +
    O \big (  (\sqrt{T}v_{T})^{-2} \big),
  \end{eqnarray}
  uniformly in $s \in \mathbb{R}$.\footnote{
    We have that 
    $a^{r} - b^{r} 
     = 
     (a - b)
    \sum_{l=0}^{r-1}
    a^{r-1-l} b^{l} 
    $
    for $a, b\in \mathbb{R}$ 
    and 
    for an integer $r \ge 2$.
    It follows that 
    $|(t/T)^{r} - (T_{j}^{0}/T)^{r} |
    \le 
     C
     |
     (t - T_{j}^{0} )/T 
     |
   $.
  }
  Under Assumption A6,
  we can show that,
  uniformly in $s \in \mathbb{R}$,
  \begin{eqnarray*}
    v_{T}^{2}
    \sum_{t = \underline{T}_{j}^{0}(s)+1}^{\overline{T}_{j}^{0}(s)}
    z_{t}
    =
    |s| \mu _{z,j+\mathbbm{1}_{ \{0 < s\}}}
    +
    o_p(1)
    \ \ \ \mathrm{and} \ \ \ 
    v_{T}^{2} 
    \sum_{t = \underline{T}_{j}^{0}(s)+1}^{\overline{T}_{j}^{0}(s)}
    z_{t}z_{t}^{\prime}
    =
    |s| Q_{zz,j+\mathbbm{1}_{ \{0 < s\}}}
    +o_p(1).
  \end{eqnarray*}
  It follows that,
  uniformly in $s \in \mathbb{R}$,
  \begin{equation*}
    D_{T,j}(s)
    =
    |s|
    \left( 
      \begin{array}{ccc}
        Q_{zz, j+\mathbbm{1}_{ \{0 < s\}}} & 
        \mu_{z, j+\mathbbm{1}_{ \{0 < s\}}} \varphi(\lambda_{j}^{0})' & 
        \mu_{z,j+\mathbbm{1}_{ \{0 < s\}}} T^{-1/2}w_{T_{j}^{0}}^{\prime } \\ 
        \varphi(\lambda_{j}^{0}) \mu_{z, j+\mathbbm{1}_{ \{0 < s\}}}^{\prime }  & 
        \varphi(\lambda_{j}^{0})\varphi(\lambda_{j}^{0})'&  
        \varphi(\lambda_{j}^{0})T^{-1/2}w_{T_{j}^{0}}' \\ 
        T^{-1/2}w_{T_{j}^{0}} \mu_{z, j+\mathbbm{1}_{ \{0 < s\}}}^{\prime }        & 
        T^{-1/2}w_{T_{j}^{0}}    \varphi(\lambda_{j}^{0})'& 
        (T^{-1/2}w_{T_{j}^{0}})(T^{-1/2}w_{T_{j}^{0}})^{\prime }        
      \end{array}%
    \right)
    + o_p(1).
  \end{equation*}%

  %%%%%%%%%%
  % W
  %%%%%%%%%%
  Also, we have 
  $
    X_{tT}
    (\Sigma_{t,j+\mathbbm{1}_{\{T_{j}(r) \le t  \} }}^{0} )^{-1}
    u_{t}
   = 
   S'
   \big ( 
   I \otimes 
   (\Sigma_{t,j+\mathbbm{1}_{\{T_{j}(r) \le t  \} }}^{0} )^{-1}
   \big )
   (x_{tT}\otimes u_{t})    
  $
  and
  $u_{t} = (\Sigma _{j+\mathbbm{1}\{0 < s\}}^{0})^{1/2}\eta _{t}$.
  Thus, 
  for $\underline{T}_{j}^{0}(s) \le t \le \overline{T}_{j}^{0}(s)$,
  \begin{eqnarray*}
    W_{T, j}(s, r)
    &=&
    S'
    \big (
    I_{q} \otimes (\Sigma_{j+\mathbbm{1}_{ \{r \le s \}}}^{0})^{-1}
    \big )
    V_{T, j}(s) \\
    &&
    -
    \mathrm{sgn}(r)
    \mathbbm{1}_{ \{ |r| \le |s|\} }
    S'
    \big (
     I_{q} \otimes
     \{
     (\Sigma_{j+1}^{0})^{-1}
     -
     (\Sigma_{j}^{0})^{-1}
     \}
    \big )
    V_{T,j}(r),
  \end{eqnarray*}
  where
  $
    V_{T, j}(s)
    :=
    \big (
    I_{q} 
    \otimes 
    (\Sigma_{j+\mathbbm{1}\{0 < s\}}^{0})^{1/2}
    \big )
    v_{T}
    \sum_{t=\underline{T}_{j}^{0}(s) +1}^{\overline{T}_{j}^{0}(s)}
    (x_{tT} \otimes \eta_{t})
  $.
  It follows from (\ref{eq:approx-asym}) that 
  \begin{equation*}
    v_{T}
    \sum_{t=\underline{T}_{j}^{0}(s) +1}^{\overline{T}_{j}^{0}(s)}
    (x_{tT} \otimes \eta_{t})
    =
    \Bigg( 
      v_{T}
      \sum_{t= \underline{T}_{j}^{0}(s)+1}^{\overline{T}_{j}^{0}(s)}
      (z_{t}\otimes \eta_{t})',
      \Big (
      \varphi(\lambda_{j}^{0})',
      T^{-1/2}
      w_{T_{j}^{0}}'
      \Big )
      \otimes
      v_{T}
      \sum_{t= \underline{T}_{j}^{0}(s)+1}^{\overline{T}_{j}^{0}(s)}
      \eta _{t}'
    \Bigg)'
    +
    o_p(1),
  \end{equation*}
  uniformly in $s \in \mathbb{R}$.
  Hence, 
  Lemma \ref{lemma:add-1} with the continuous mapping theorem yields 
  $\{B_{T, j}(\cdot), W_{T, j}(\cdot)\}_{j=1}^{m} 
   \Rightarrow 
   \{\mathbb{B}_{j}(\cdot), \mathbb{W}_{j}(\cdot)\}_{j=1}^{m}
  $.
\end{proof}
\vspace{0.1cm}
%%%%%%%%%%%%%%%%%%%%%%%%%%%%%%%%%%%%%%%%%%%%%%%%%%%%%%%%%%%%%%%%%%%%%%%%%%%%%%%%%%%%%%%%%%%%

%%%%%%%%%%%%%%%%%%%%%%%%%%%%%%%%%%%%%%%%%%%%%%%%%%%%%%%%%%%%%%%%%%%%%%%%%%%%%%%%%%%%%%%%%%%%

%%%%%%%%%%%%%%%%%%%%%%%%%%%%%%%%%%%%%%%%%%%%%%%%%%%%%%%%%%%%%%%%%%%%%%%%%%%%%%%%%%%%%%%%%%%%
\begin{proof}[\textbf{\textrm{Proof of Theorem \ref{theorem:limit-dist}}}.]
    
  Theorems \ref{theorem:rate} and \ref{theorem:parts} imply that,
  for a sufficiently large $M>0$,
  \begin{eqnarray}
    \label{eq:CB-part}
    CB_{T}
    = 
    2
    \Big \{
    \sup_{\mathcal{K} \in \bar\Xi_{M} }
    \ell_{T}(\mathcal{K},\theta^{0})
    -
    \sup_{\mathcal{K} \in \bar\Xi_{M, H_{0}}}
    \ell_{T}(\mathcal{K},\theta^{0})
    \Big \}
    +
    o_{p}(1).
  \end{eqnarray}
  Let $M$ be an arbitrary large constant.
  For 
  $(g, j) \in \{1, \dots, G\} \times \{1, \dots, m\}$,
  define 
  $\bm{r}_{j}:=(r_{1j}, \dots, r_{Gj})'$
  with $r_{gj} \in [-M, M]$
  and 
  consider $\mathcal{K} \in \bar\Xi_{M}$
  such that 
  $k_{gj} = T_{j}^{0} + [r_{gj}v_{T}^{-2}]$.
  Then, we can write 
  $\underline{k}_{j}{=}T_{j}^{0} + \min \{[r_{1j}v_{T}^{-2} ],\dots, [r_{Gj}v_{T}^{-2} ], 0\}$
  and 
  $\overline{k}_{j}{=}T_{j}^{0} + \max \{[r_{1j}v_{T}^{-2} ],\dots, [r_{Gj}v_{T}^{-2} ], 0\}$.
  Also,
  $
    \ell_{T}(\mathcal{K},\theta^{0})
    {=}
    \sum_{j=1}^{m} \ell_{T}^{(j)}(\mathbf{r}_{j})
  $,
  where 
  $
  \ell_{T}^{(j)}(\mathbf{r}_{j})
  {:=}
    \sum_{\underline{k}_{j}+1}^{\bar{k}_{j}}
    \big\{ 
      \log f(y_{t}|X_{tT},\theta_{t,\mathcal{K}}^{0})
      -
      \log f(y_{t}|X_{tT},\theta_{t,\mathcal{T}^{0}}^{0})
    \big\}   
  $.
  Observe that,
  for $1 \le t \le T$,
  \begin{eqnarray*} 
    \log f(y_{t}|X_{tT},\theta _{t,\mathcal{K}}^{0})
    &=&
    -
    \frac{1}{2}
    \Big \{
    \log (2 \pi)^{n}
    +
    \log |\Sigma _{t,\mathcal{K}}^{0}|
    +
    \| (\Sigma _{t,\mathcal{K}}^{0})^{-1/2}u_{t} \|^{2} 
     \\ 
    &&  \hspace{2cm}
    -
    2
    (\Delta \beta _{t,\mathcal{K}}^{0})^{\prime}
    X_{tT}
    (\Sigma _{t,\mathcal{K}}^{0})
    u_{t}
    +
    \| (\Sigma _{t,\mathcal{K}}^{0})^{-1/2}
       X_{tT}^{\prime }\Delta \beta _{t,\mathcal{K}}^{0})
    \|^{2}
    \Big \}.
  \end{eqnarray*}
  Let
  $\underline{k}_{Gj}{:=}T_{j}^{0}{+}\min \{[r_{Gj}v_{T}^{-2} ], 0\}$
  and
  $\overline{k}_{Gj}{:=}T_{j}^{0}{+}\max \{[r_{Gj}v_{T}^{-2} ], 0\}$
  for $j \in \{1, \dots, m\}$.
  Then,
  \begin{eqnarray*}
    \ell_{T}^{(j)}(\bm{r}_{j})
    =
    \ell_{T, 1}^{(j)}(\bm{r}_{j})
    +
    \ell_{T, 2}^{(j)}(\bm{r}_{j}),  
  \end{eqnarray*}
  where 
  \begin{eqnarray*}
    \ell_{T, 1}^{(j)}( \bm{r}_{j})
    &:=&
    \frac{1}{2}
    \sum_{t=\underline{k}_{Gj}+1}^{ \overline{k}_{Gj}}
    \Big \{
    \log
    \big | 
    \Sigma_{t, \mathcal{T}^{0}}^{0}
    (\Sigma_{t, \mathcal{K}}^{0})^{-1}
    \big |
    +
    \mathrm{tr}
    \big (
    \big \{
      (\Sigma _{t,\mathcal{T}^{0}}^{0})^{-1}
      -
      (\Sigma _{t,\mathcal{K}}^{0})^{-1}
    \big \} 
    u_{t}u_{t}^{\prime } 
    \big )
    \Big \},\\
    \ell_{T,2}^{(j)}(\bm{r}_{j}) 
    &:=& 
    \frac{1}{2}
    \sum_{t=\underline{k}_{j}+1}^{\overline{k}_{j}}
    \big \{
    2 
    (\Delta \beta _{t, \mathcal{K}}^{0})'
    X_{tT}(\Sigma_{t,\mathcal{T}^{0}}^{0})^{-1}u_{t}
    -
    \| 
    (\Sigma_{t,\mathcal{T}^{0}}^{0})^{-1/2}
    X_{tT}'\Delta \beta _{t,\mathcal{K}}^{0}\| ^{2}
    \big \}.
  \end{eqnarray*}

  %%%%%%%%%%
  % l1 
  %%%%%%%%%%
  First, we consider the term $\ell_{T,1}^{(j)}(\bm{r}_{j})$.
  We can write 
  $
    \Sigma_{t, \mathcal{T}^{0}}^{0}
    (\Sigma_{t, \mathcal{K}}^{0})^{-1}
    =
    I_{n}
    -
    (
    \Sigma_{t, \mathcal{K}}^{0}
    -
    \Sigma_{t, \mathcal{T}^{0}}^{0}
    )
    (\Sigma_{t, \mathcal{K}}^{0})^{-1}
  $
  and 
  $\Sigma _{t,\mathcal{K}}^{0} - \Sigma _{t,\mathcal{T}^{0}}^{0}=v_{T}\Phi_{t,\mathcal{K}}$,
  where 
  $\Phi_{t, \mathcal{K}} = \Phi_{j}$
  if 
  $k_{Gj} < t \le T_{j}^{0}$
  and 
  $\Phi_{t, \mathcal{K}} = - \Phi_{j}$
  if 
  $T_{j}^{0} < t \le k_{Gj}$.
  Thus,
  an application of the Taylor series expansion yields that, 
  for 
  $\underline{k}_{Gj}\le t \le \overline{k}_{Gj}$,
  \begin{eqnarray}
    \label{eq:lr1-part-1}
    \log 
    |
    \Sigma_{t, \mathcal{T}^{0}}^{0}
    (\Sigma_{t, \mathcal{K}}^{0})^{-1}
    |
    =
    \mathrm{tr}
    \big (
      - v_{T}\Phi _{t,\mathcal{K}}(\Sigma _{t,\mathcal{K}}^{0})^{-1}
    \big )
    +
    \frac{1}{2}
    \mathrm{tr}
    \big (
    \big \{
    v_{T}\Phi _{t,\mathcal{K}}(\Sigma _{t,\mathcal{K}}^{0})^{-1} 
    \big \}^{2} 
    \big )
    +
    O_{p}(v_{T}^{3}).
  \end{eqnarray}
  Also we can write 
  $
   (\Sigma _{t,\mathcal{T}^{0}}^{0})^{-1}
   -
   (\Sigma _{t,\mathcal{K}}^{0})^{-1}
   =   
   (\Sigma_{t,\mathcal{T}^{0}}^{0})^{-1}
   ( 
   \Sigma_{t,\mathcal{K}}^{0}
   -
   \Sigma _{t,\mathcal{T}^{0}}^{0}
   ) 
   (\Sigma _{t, \mathcal{K}}^{0})^{-1}
  $
  and 
  $u_{t} 
   = 
   (\Sigma _{t, \mathcal{T}^{0}}^{0})^{1/2}
   \eta_{t}
   $, which implies,   
  for 
  $\underline{k}_{Gj}\le t \le \overline{k}_{Gj}$,
  \begin{eqnarray}
    \label{eq:lr1-part-2}
    \mathrm{tr}
    \big (
    \big \{
      (\Sigma _{t,\mathcal{T}^{0}}^{0})^{-1}
      -
      (\Sigma _{t,\mathcal{K}}^{0})^{-1}
    \big \} 
    u_{t}u_{t}^{\prime } 
    \big )
    =
    \mathrm{tr}
    \Big (
    ( \Sigma _{t,\mathcal{T}^{0}}^{0} )^{-1/2}
    v_{T}\Phi _{t,\mathcal{K}}
    ( \Sigma _{t,\mathcal{K}}^{0} )^{-1}
    ( \Sigma _{t,\mathcal{T}^{0}}^{0} )^{1/2}
    \eta_{t}\eta_{t}^{\prime }
    \Big ).
  \end{eqnarray}
  For 
  $\underline{k}_{Gj} \le t \le \overline{k}_{Gj}$,
  we have 
  \begin{eqnarray*}
    (\Phi_{t,\mathcal{K}}, \Sigma_{t, \mathcal{T}^{0}}^{0}, \Sigma_{t, \mathcal{K}}^{0})    
    =
    \left \{
    \begin{array}{rl}
      (\Phi_{j}, \Sigma_{j}^{0}, \Sigma_{j+1}^{0}), & \mathrm{if} \ r_{Gj} \le 0\\     
      (-\Phi_{j}, \Sigma_{j+1}^{0}, \Sigma_{j}^{0}), & \mathrm{if} \ r_{Gj} > 0.
    \end{array}
    \right .
  \end{eqnarray*}
  Using 
  (\ref{eq:lr1-part-1})
  and 
  (\ref{eq:lr1-part-2})
  with 
  $ \pi_{j}(r_{Gj}):=
    (\Sigma _{t,\mathcal{T}^{0}}^{0}
    )^{-{1}/{2}}
    \Phi _{t,\mathcal{K}}
    (  
    \Sigma _{t,\mathcal{K}}^{0}
    )^{-1}
    (\Sigma _{t,\mathcal{T}^{0}}^{0})^{{1}/{2}}
  $,
  we obtain 
  \begin{eqnarray}
   \label{eq:LR1-part1} 
   \ell_{T,1}^{(j)}(\bm{r}_{j})
    =
    \frac{1}{2}
    \mathrm{tr}
    \Big (
    \pi_{j}(r_{Gj}) 
    V_{T,\eta\eta, j}(r_{Gj}) 
    \Big )
    +  
    \frac{|r_{Gj}|}{4}
    \mathrm{tr}
    \big (
     \{\pi_{j}(r_{Gj})\}^2
    \big )
    +o_{p}(1),
  \end{eqnarray}
  where
  $ V_{T,\eta\eta, j}(r_{Gj}):=v_{T}
    \sum_{t = \underline{k}_{Gj}+1}^{\overline{k}_{Gj}}
    (\eta _{t}\eta_{t}^{\prime }-I_{n})$.

  %%%%%%%%%%%%
  % part 2
  %%%%%%%%%%%%
  Next, we consider the term $\ell_{T,2}^{(j)}(\bm{r}_{j})$.
  Define 
  $
   \Delta \beta _{g,t,\mathcal{K}}^{0}
   :=
   \sum_{i \in \mathcal{G}_{g}} e_{i} \circ (\beta _{t,\mathcal{K}}^{0}
   -
   \beta _{t,\mathcal{T}^{0}}^{0})$. 
  Then
  $\Delta \beta _{t,\mathcal{K}}^{0}
  =
  \sum_{g=1}^{G}\Delta \beta _{g,t,\mathcal{K}}^{0}$
  and 
  we have 
  \begin{eqnarray*}
    \ell_{T,2}^{(j)}(\bm{r}_{j})
    =
    \sum_{t=\underline{k}_{j}+1}^{\overline{k}_{j}}
    \bigg (
    \sum_{g=1}^{G}
    (\Delta \beta_{g,t,\mathcal{K}}^{0})' 
    X_{tT}(\Sigma_{t,\mathcal{K}}^{0})^{-1}u_{t} 
    -
    \frac{1}{2}
    \sum_{g=1}^{G}
    \sum_{l=1}^{G}
    (\Delta \beta _{g,t,\mathcal{K}}^{0})
    'X_{tT}(\Sigma_{t,\mathcal{K}}^{0})^{-1}X_{tT}'
    \Delta \beta _{l,t,\mathcal{K}}^{0}
    \bigg ) .
  \end{eqnarray*} 
  For a group $g \in \{1, \dots, G\}$,
  we have that 
  $\Delta \beta_{g,t,\mathcal{K}}^{0} 
   = 
   \beta_{g, j+1}^{0}
   - 
   \beta_{g j}^{0}
  $
  for 
  $k_{gj} < t \le  T_{j}^{0}$
  and 
  that 
  $\Delta \beta_{g,t,\mathcal{K}}^{0} 
   = 
   -
   (
   \beta_{g, j+1}^{0}
   - 
   \beta_{g j}^{0}
   )
  $
  for
  $T_{j}^{0} < t \le  k_{gj}$.
  It follows that 
  \begin{eqnarray*}
    \sum_{t=\underline{k}_{j}+1}^{\overline{k}_{j}}
    (\Delta \beta_{g,t,\mathcal{K}}^{0})' 
    X_{tT}(\Sigma_{t,\mathcal{K}}^{0})^{-1}u_{t} 
    =
    -
    \mathrm{sgn}(r_{gj})
    \delta_{gj}'
    W_{T,j}(r_{gj}, r_{Gj}).
  \end{eqnarray*}
  Similarly, 
  for groups $g, h \in \{1, \dots, G\}$,
  we have that
  \begin{eqnarray*}
    &&
    \sum_{t=\underline{k}_{j}+1}^{\overline{k}_{j}}
    (\Delta \beta _{g,t,\mathcal{K}}^{0})
    'X_{tT}(\Sigma_{t,\mathcal{K}}^{0})^{-1}X_{tT}'
    \Delta \beta _{h,t,\mathcal{K}}^{0} \\
    && \ \ \ =  
    \mathbbm{1}_{ \{k_{gj} \vee k_{hj} \le T_{j}^{0} \}  }
    \delta_{gj}'
    B_{T,j}(r_{gj}\vee r_{hj}, r_{Gj} )
    \delta_{h j}
    +
    \mathbbm{1}_{ \{ T_{j}^{0} < k_{gj} \wedge k_{hj}  \}  }
    \delta_{gj}'
    B_{T,j}(r_{gj} \wedge  r_{hj}, r_{Gj} )
    \delta_{h j}.
  \end{eqnarray*}
  Thus, we have 
  \begin{eqnarray*}
    \ell_{T,2}^{(j)}(\bm{r}_{j})
    &=&
    -
    \sum_{g=1}^{G}
    \mathrm{sgn}(r_{gj})
    \delta_{gj}'
    W_{T,j}(r_{gj}, r_{Gj}) \\
    &&  
    -
    \frac{1}{2}
    \sum_{g=1}^{G}
    \sum_{l=1}^{G}
    \delta_{gj}' 
    \Big \{
    \mathbbm{1}_{ \{r_{gj} \vee r_{lg} \le  0\}}
    B_{T,j}
    \big ( r_{gj}{\vee}r_{lj}, r_{Gj} \big )
    +
    \mathbbm{1}_{ \{0 < r_{gj} \wedge r_{lg} \}}
    B_{T, j}
    \big ( r_{gj}{\wedge}r_{lj}, r_{Gj} \big )
    \Big \}
    \delta_{lj}.  
  \end{eqnarray*}

  %%%%%%%%%%%
  % Limit
  %%%%%%%%%%%
  Applying Lemma 1 
  with 
  (\ref{eq:LR1-part1}) and the above equation, we can obtain 
  \begin{eqnarray*}
    \big (
    \ell_{T}^{(1)}(\bm{r}_{1}), \dots, \ell_{T}^{(m)}(\bm{r}_{m}) 
    \big )
    \Rightarrow
    \big (
    \ell_{\infty}^{(1)}(\bm{r}_{1}), \dots, \ell_{\infty}^{(m)}(\bm{r}_{m}) 
    \big ),
  \end{eqnarray*}
  where, for $j=1,\dots, m$,
  \begin{eqnarray*}
    \ell_{\infty}^{(j)}(\bm{r}_{j})
    &:=&
    \frac{1}{2}
    \mathrm{tr}
    \Big (
    \pi_{j}(r_{Gj})
    \mathbb{V}_{\eta\eta,j}(r_{G})
    \Big ) 
     +
    \frac{|r_{Gj}|}{4}
    \mathrm{tr}
    \Big (
    \big \{
    \pi_{j}(r_{Gj})
    \big \}^2
    \Big ) 
     -
    \sum_{g=1}^{G}
    \mathrm{sgn}(r_{gj})
    \delta_{gj}' \mathbb{W}_{j}(r_{gj}, r_{Gj})
    \\
    && 
    -
    \frac{1}{2}
    \sum_{g=1}^{G}
    \sum_{h=1}^{G}
    \delta_{gj}' 
    \Big \{
    \mathbbm{1}_{ \{r_{gj} \vee r_{hg} \le  0\}}
    \mathbb{B}_{j}
    \big ( r_{gj}{\vee}r_{hj}, r_{Gj} \big )
    +
    \mathbbm{1}_{ \{0 <  r_{gj} \wedge r_{hg} \}}
    \mathbb{B}_{j}
    \big ( r_{gj}{\wedge}r_{hj}, r_{Gj} \big )
    \Big \}
    \delta_{hj}. 
  \end{eqnarray*}
  Applying a change of variables with 
  $\bm{s}_{j} := 
  \big ( \|\delta_{j}\|^2+ \mathrm{tr}(\Phi_{j}^2) \big ) \bm{r}_{j}$
  with $\bm{s}_{j}=(s_{1},\dots, s_{m})'$
  for $j=1, \dots, m$, 
  we can show that 
  $
  2
  \ell_{\infty}^{(j)}(\bm{r}_{j})
  =
  CB_{\infty}^{(j)}(\bm{s}_{j})
  $
  for all $j =1, \dots, m$.
  Thus, 
  the continuous mapping theorem leads to the desired result.
\end{proof}
\vspace{0.1cm}
%%%%%%%%%%%%%%%%%%%%%%%%%%%%%%%%%%%%%%%%%%%%%%%%%%%%%%%%%%%%%%%%%%%%%%%%%%%%%%%%%%%%%%%%%%%%

%%%%%%%%%%%%%%%%%%%%%%%%%%%%%%%%%%%%%%%%%%%%%%%%%%%%%%%%%%%%%%%%%%%%%%%%%%%%%%%%%%%%%%%%%%%%

%%%%%%%%%%%%%%%%%%%%%%%%%%%%%%%%%%%%%%%%%%%%%%%%%%%%%%%%%%%%%%%%%%%%%%%%%%%%%%%%%%%%%%%%%%%%
\begin{proof}[\textbf{\textrm{Proof of Theorem \ref{theorem:alternative}}}.]
  Under both alternatives $H_{1}$ and $H_{1T}$, 
  the convergence rates of Theorem \ref{theorem:rate}
  apply to the estimates $\hat{\theta}$ and $\hat{\mathcal{K}}$. 
  Thus, given collections of break dates $\hat{\mathcal{K}}$ and $\mathcal{K}^{0}$,
  the sub-intervals $\{[\tau_{g,l-1}+1, \tau_{gl}]\}_{l=1}^{N_{g}}$
  for each group $g$ 
  satisfy 
  either 
  $\tau_{gl} - \tau_{g,l-1} \ge \nu T$
  or 
  $\tau_{gl} - \tau_{g,l-1} \le M v_{T}^{-2}$
  for some $M>0$.
  If $\tau_{gl} - \tau_{g,l-1} \ge \nu T$, 
  then 
  the arguments used to prove 
  Property \ref{property:optT}(b) 
  with $\sqrt{T}$-consistent estimate $\hat{\theta}$
  show that 
  $\bar{\ell}_{g,l}(\hat{\mathcal{K}}, \hat{\theta}) = O_p(1)$,
  while 
  the arguments to obtain (\ref{eq:th-3})
  show that 
  $\bar{\ell}_{g,l}(\hat{\mathcal{K}}, \hat{\theta}) = O_p(1)$
  if $\tau_{gl} - \tau_{gl} \le M v_{T}^{-2}$.
  Also, 
  Theorem \ref{theorem:rate}(b) implies that 
  $\Delta
  (\hat{\mathcal{K}}, \hat{\theta})
  = o_p(1)
  $.
  It follows from 
  Lemma \ref{lemma:decomposition-l}  that 
  \begin{eqnarray}
    \label{eq:L-H1}
    \ell_{T}(\hat{\mathcal{K}}, \hat{\theta}) = O_p(1).
  \end{eqnarray}
  It remains to consider the normalized likelihood 
  $\ell_{T}(\tilde{\mathcal{K}}, \tilde{\theta})$
  under the null hypothesis $H_{0}$. 
  
  %%%%%%%%%%%%
  % (a)
  %%%%%%%%%%%%
  \textbf{(a)}
  Let $\delta \in (0,1)$ be fixed.
  If
  $
  \max_{1 \le j \le m}
  \max_{1 \le g_{1}, g_{2}\le G}
  |k_{g_{1}j}^{0} - k_{g_{2}j}^{0}| \ge \delta T$,
  then we have 
  $
  \max_{1 \le j \le m}
  \max_{1 \le g \le G}
  |\tilde{k}_{j} - k_{gj}^{0}| \ge \delta T / 2
  $.
  Applying a similar argument used in Proposition \ref{pro:date1},
  we can show that 
  Properties \ref{property:opt1} 
  and 
  \ref{property:smallout} with $m_{T} = \delta T / 2$
  imply that
  \begin{eqnarray}
    \label{eq:L-H0}
    \ell_{T}(\tilde{\mathcal{K}}, \tilde{\theta}) 
    \le   
    -
    |O_p(T v_{T}^{2})|.
  \end{eqnarray}
  It follows from 
  (\ref{eq:L-H1})
  and 
  (\ref{eq:L-H0})
  that 
  $CB_{T}
   = 
   2
   \{
   \ell_{T}(\hat{\mathcal{K}}, \hat{\theta})   
   -
   \ell_{T}(\tilde{\mathcal{K}}, \tilde{\theta})
   \}
   \ge 
   |O_p(T v_{T}^{2})|
  $.
  Since the critical value $c_{\alpha}^{\ast}$ is a finite value, 
  we obtain the desired result.

  %%%%%%%%%%%%
  % (b)
  %%%%%%%%%%%%
  \textbf{(b)}
  If
  $
  \max_{1 \le j \le m}
  \max_{1 \le g_{1}, g_{2}\le G}
  |k_{g_{1}j}^{0} - k_{g_{2}j}^{0}| \ge M v_{T}^{-2} $
  for some constant $M>0$,
  then we have 
  $
  \max_{1 \le j \le m}
  \max_{1 \le g \le G}
  |\tilde{k}_{j} - k_{gj}^{0}| \ge  M v_{T}^{-2} / 2
  $.
  When 
  $
  \max_{1 \le j \le m}
  \max_{1 \le g \le G}
  |\tilde{k}_{j} - k_{gj}^{0}| \ge  D v_{T}^{-2} \log T
  $
  for a sufficiently large $D$,
  it was shown that 
  $
    \ell_{T}(\tilde{\mathcal{K}}, \tilde{\theta})
    \le  
    - |O_p(M)|
  $
  in the proof of Proposition \ref{pro:date1}.
  When
  $
  M v_{T}^{-2} 
  \le 
  \max_{1 \le j \le m}
  \max_{1 \le g \le G}
  |\tilde{k}_{j} - k_{gj}^{0}| \le  D v_{T}^{-2} \log T
  $, 
  it follows from the proof of Theorem \ref{theorem:rate}(a)
  that 
  $
    \ell_{T}(\tilde{\mathcal{K}}, \tilde{\theta})
    \le  
    - |O_p(M)|
  $
  for a sufficiently large $M>0$.
  Thus,
  there is some $M>0$
  such that 
  $CB_{T}
   \ge 
   |O_p(M)|
  $
  and 
  the proof is completed.
\end{proof}
%%%%%%%%%%%%%%%%%%%%%%%%%%%%%%%%%%%%%%%%%%%%%%%%%%%%%%%%%%%%%%%%%%%%%%%%%%%%%%%%%%%%%%%%%%%%
 
%%%%%%%%%%%%%%%%%%%%%%%%%%%%%%%%%%%%%%%%%%%%%%%%%%%%%%%%%%%%%%%%%%%%%%%%%%%%%%%%%%%%%%%%%%%%

\clearpage
\thispagestyle{empty}

\setstretch{0.1}      

\fontsize{9pt}{12pt}\selectfont

%\addtolength{\oddsidemargin}{-0.482in}
\addtolength{\evensidemargin}{-.410in}
\addtolength{\topmargin}{-0.90in}  
\addtolength{\textheight}{2.000in}

\vspace{3cm}

\begin{table}[htbp]
\centering
Table 1. Empirical Rejection Frequencies under the Null Hypotheses
\begin{tabular}{cccccccccccccc} \hline \hline
& & \multicolumn{11}{c}{AR Coefficient} &  \\ 
\cline{4-14}
& & &
\multicolumn{3}{c}{$\alpha=0.0$}&  & 
\multicolumn{3}{c}{$\alpha=0.4$}&  &
\multicolumn{3}{c}{$\alpha=0.8$}\\ 
\cline{4-6} \cline{8-10} \cline{8-10} \cline{12-14}
\multicolumn{2}{c}{Break Size}
&&
\multicolumn{3}{c}{Nominal Size}&  &
\multicolumn{3}{c}{Nominal Size}&  & 
\multicolumn{3}{c}{Nominal Size} \\ 
\cline{1-2} \cline{4-6} \cline{8-10} \cline{8-10} \cline{12-14}
$\delta_1$ & $\delta_2$ &  & 
10\% & 5\% & 1\% &  & 
10\% & 5\% & 1\% &  & 
10\% & 5\% & 1\% \\ \hline
0.50 & 0.50 &  & 0.064 & 0.036 & 0.004 &  & 0.086 & 0.050 & 0.004 &  & 0.162 &0.104 & 0.032 \\ 
 & 0.75 &  & 0.070 & 0.036 & 0.004 &  & 0.094 & 0.054 & 0.006 &  & 0.158 &0.088 & 0.032 \\ 
 & 1.00 &  & 0.084 & 0.036 & 0.004 &  & 0.106 & 0.060 & 0.010 &  & 0.170 &0.098 & 0.038 \\ 
 & 1.25 &  & 0.086 & 0.044 & 0.004 &  & 0.108 & 0.058 & 0.014 &  & 0.182 &0.104 & 0.040 \\ 
 & 1.50 &  & 0.096 & 0.050 & 0.006 &  & 0.120 & 0.056 & 0.010 &  & 0.186 &0.108 & 0.036 \\ 
0.75 & 0.75 &  & 0.084 & 0.032 & 0.004 &  & 0.112 & 0.046 & 0.004 &  & 0.158 &0.086 & 0.030 \\ 
 & 1.00 &  & 0.088 & 0.040 & 0.004 &  & 0.108 & 0.050 & 0.010 &  & 0.154 &0.082 & 0.030 \\ 
 & 1.25 &  & 0.086 & 0.050 & 0.006 &  & 0.104 & 0.060 & 0.006 &  & 0.156 &0.088 & 0.028 \\ 
 & 1.50 &  & 0.090 & 0.052 & 0.006 &  & 0.118 & 0.058 & 0.010 &  & 0.166 &0.090 & 0.028 \\ 
1.00 & 1.00 &  & 0.090 & 0.044 & 0.008 &  & 0.104 & 0.060 & 0.012 &  & 0.150 &0.078 & 0.022 \\ 
 & 1.25 &  & 0.086 & 0.050 & 0.010 &  & 0.090 & 0.060 & 0.010 &  & 0.140 &0.072 & 0.026 \\ 
 & 1.50 &  & 0.092 & 0.050 & 0.012 &  & 0.096 & 0.056 & 0.012 &  & 0.152 &0.070 & 0.026 \\ 
1.25 & 1.25 &  & 0.080 & 0.044 & 0.008 &  & 0.084 & 0.052 & 0.012 &  & 0.118 &0.058 & 0.018 \\ 
 & 1.50 &  & 0.074 & 0.042 & 0.010 &  & 0.080 & 0.044 & 0.010 &  & 0.112 &0.056 & 0.018 \\ 
1.50 & 1.50 &  & 0.074 & 0.038 & 0.010 &  & 0.088 & 0.040 & 0.010 &  & 0.106 &0.048 & 0.018 \\ 
\hline
\end{tabular}
\label{}
  \begin{minipage}{140mm}
  \textit{Notes:} The data generating process is 
   the bivariate system:
    \begin{align}
      y_{1t} &= 1 + \delta_1 \mathbbm{1}_{\{k_{1}<t\}} + \alpha y_{1,t-1} + u_{1t} 
       \tag{EQ1} \\
      y_{2t} &= 1 + \delta_2 \mathbbm{1}_{ \{k_{2}<t \}} + \alpha y_{2,t-1} + u_{2t}
       \tag{EQ2}, 
    \end{align}
  for $t=1,\dots, T$,
  where 
  $(u_{1t}, u_{2t})' \sim i.i.d.\hspace{0.5mm}N(0, I_2)$
  and 
  $\delta_i$ is the break size for the $i^{th}$ equation for $i = 1, 2$. 
  We set the sample size $T=100$,
  the break date $k_{1}=k_{2} = 50$
  and 
  the trimming value $\nu = 0.15$.
  \end{minipage}

\end{table}

%%%%%%%%%%%%%%%%%%%%%%%%%%%%%%%%%%%%%%%%%%%%%%%%%%%%%%%%%%%%%%%%%%%%%%%%%%%%%%%%%%%%%%%%%%%%%%%

% Table generated by Excel2LaTeX from sheet 'Sheet1'
\begin{table}[htbp]
  \def\arraystretch{1.0}  
  \centering
  Table 2. Empirical Rejection Frequencies under the Null and Alternative Hypotheses  \\
  (the significance level: 5\%)
  \begin{tabular}{ccccccrccrcc} \hline \hline
          &       &       &       & (1)   & (2)   & \multicolumn{1}{c}{} & (3)   & (4)   & \multicolumn{1}{c}{} & (5)   & (6) \\
          &       &       &       & \multicolumn{8}{c}{Break dates ($k_1$, $k_2$)} \\ \cline{5-12}
          &       &       &       & \multicolumn{2}{c}{(50,50)} & \multicolumn{1}{c}{} & \multicolumn{2}{c}{(35, 35)} & \multicolumn{1}{c}{} & \multicolumn{2}{c}{(35, 50)} \\
\cline{5-6} \cline{8-9} \cline{11-12}
          & AR      &   \multicolumn{2}{c}{Break Size}           & \multicolumn{2}{c}{Trimming value} & \multicolumn{1}{l}{} & \multicolumn{2}{c}{Trimming value} & \multicolumn{1}{l}{} & \multicolumn{2}{c}{Trimming value} \\ \cline{5-6} \cline{8-9} \cline{11-12}
    Correlation & $\alpha$    &$\delta_{1}$ & $\delta_{2}$  & 0.15  & 0.10  & \multicolumn{1}{c}{} & 0.15  & 0.10  & \multicolumn{1}{c}{} & 0.15  & 0.10 \\ \hline
    0.0   & 0.0   & 0.5  & 0.5  & 0.024 & 0.030 & \multicolumn{1}{c}{} & 0.018 & 0.030 & \multicolumn{1}{c}{} & 0.05 & 0.06 \\
          &       &       & 1.0  & 0.030 & 0.034 & \multicolumn{1}{c}{} & 0.026 & 0.038 & \multicolumn{1}{c}{} & 0.154 & 0.166 \\
          &       &       & 1.5  & 0.036 & 0.038 & \multicolumn{1}{c}{} & 0.034 & 0.048 & \multicolumn{1}{c}{} & 0.226 & 0.228 \\
          &       & 1.0  & 1.0  & 0.032 & 0.034 & \multicolumn{1}{c}{} & 0.048 & 0.028 & \multicolumn{1}{c}{} & 0.550 & 0.554 \\
          &       &       & 1.5  & 0.036 & 0.038 & \multicolumn{1}{c}{} & 0.022 & 0.022 & \multicolumn{1}{c}{} & 0.728 & 0.730 \\
          &       & 1.5  & 1.5  & 0.034 & 0.034 & \multicolumn{1}{c}{} & 0.012 & 0.012 & \multicolumn{1}{c}{} & 0.932 & 0.932 \\
          & 0.4   & 0.5  & 0.5  & 0.036 & 0.044 & \multicolumn{1}{c}{} & 0.026 & 0.040 & \multicolumn{1}{c}{} & 0.064 & 0.080 \\
          &       &       & 1.0  & 0.038 & 0.050 & \multicolumn{1}{c}{} & 0.040 & 0.056 & \multicolumn{1}{c}{} & 0.182 & 0.188 \\
          &       &       & 1.5  & 0.048 & 0.056 & \multicolumn{1}{c}{} & 0.036 & 0.050 & \multicolumn{1}{c}{} & 0.250 & 0.300 \\
          &       & 1.0  & 1.0  & 0.044 & 0.044 & \multicolumn{1}{c}{} & 0.054 & 0.036 & \multicolumn{1}{c}{} & 0.586 & 0.569 \\
          &       &       & 1.5  & 0.048 & 0.048 & \multicolumn{1}{c}{} & 0.062 & 0.032 & \multicolumn{1}{c}{} & 0.732 & 0.734 \\
          &       & 1.5  & 1.5  & 0.036 & 0.036 & \multicolumn{1}{c}{} & 0.018 & 0.018 & \multicolumn{1}{c}{} & 0.934 & 0.945 \\
          & 0.8   & 0.5  & 0.5  & 0.082 & 0.092 & \multicolumn{1}{c}{} & 0.096 & 0.102 & \multicolumn{1}{c}{} & 0.172 & 0.215 \\
          &       &       & 1.0  & 0.078 & 0.084 & \multicolumn{1}{c}{} & 0.100 & 0.104 & \multicolumn{1}{c}{} & 0.300 & 0.390 \\
          &       &       & 1.5  & 0.090 & 0.104 & \multicolumn{1}{c}{} & 0.178 & 0.096 & \multicolumn{1}{c}{} & 0.370 & 0.445 \\
          &       & 1.0  & 1.0  & 0.068 & 0.068 & \multicolumn{1}{c}{} & 0.080 & 0.082 & \multicolumn{1}{c}{} & 0.668 & 0.710 \\
          &       &       & 1.5  & 0.056 & 0.056 & \multicolumn{1}{c}{} & 0.056 & 0.056 & \multicolumn{1}{c}{} & 0.774 & 0.805 \\
          &       & 1.5  & 1.5  & 0.044 & 0.044 & \multicolumn{1}{c}{} & 0.032 & 0.032 & \multicolumn{1}{c}{} & 0.942 & 0.955 \\ \hline 
    0.5   & 0.0   & 0.5  & 0.5  & 0.018 & 0.022 & \multicolumn{1}{c}{} & 0.020 & 0.026 & \multicolumn{1}{c}{} & 0.106 & 0.106 \\
          &       &       & 1.0  & 0.028 & 0.034 & \multicolumn{1}{c}{} & 0.038 & 0.038 & \multicolumn{1}{c}{} & 0.256 & 0.248 \\
          &       &       & 1.5  & 0.038 & 0.038 & \multicolumn{1}{c}{} & 0.040 & 0.046 & \multicolumn{1}{c}{} & 0.300 & 0.298 \\
          &       & 1.0  & 1.0  & 0.028 & 0.028 & \multicolumn{1}{c}{} & 0.026 & 0.028 & \multicolumn{1}{c}{} & 0.730 & 0.730 \\
          &       &       & 1.5  & 0.036 & 0.036 & \multicolumn{1}{c}{} & 0.030 & 0.030 & \multicolumn{1}{c}{} & 0.826 & 0.828 \\
          &       & 1.5  & 1.5  & 0.020 & 0.020 & \multicolumn{1}{c}{} & 0.020 & 0.020 & \multicolumn{1}{c}{} & 0.978 & 0.978 \\
          & 0.4   & 0.5  & 0.5  & 0.022 & 0.034 & \multicolumn{1}{c}{} & 0.030 & 0.038 & \multicolumn{1}{c}{} & 0.130 & 0.138 \\
          &       &       & 1.0  & 0.044 & 0.044 & \multicolumn{1}{c}{} & 0.032 & 0.038 & \multicolumn{1}{c}{} & 0.262 & 0.268 \\
          &       &       & 1.5  & 0.044 & 0.046 & \multicolumn{1}{c}{} & 0.048 & 0.052 & \multicolumn{1}{c}{} & 0.318 & 0.324 \\
          &       & 1.0  & 1.0  & 0.038 & 0.038 & \multicolumn{1}{c}{} & 0.036 & 0.042 & \multicolumn{1}{c}{} & 0.752 & 0.752 \\
          &       &       & 1.5  & 0.036 & 0.036 & \multicolumn{1}{c}{} & 0.034 & 0.034 & \multicolumn{1}{c}{} & 0.832 & 0.834 \\
          &       & 1.5  & 1.5  & 0.022 & 0.022 & \multicolumn{1}{c}{} & 0.022 & 0.022 & \multicolumn{1}{c}{} & 0.978 & 0.978 \\
          & 0.8   & 0.5  & 0.5  & 0.060 & 0.070 & \multicolumn{1}{c}{} & 0.074 & 0.082 & \multicolumn{1}{c}{} & 0.214 & 0.214 \\
          &       &       & 1.0  & 0.068 & 0.070 & \multicolumn{1}{c}{} & 0.076 & 0.084 & \multicolumn{1}{c}{} & 0.362 & 0.364 \\
          &       &       & 1.5  & 0.062 & 0.064 & \multicolumn{1}{c}{} & 0.068 & 0.074 & \multicolumn{1}{c}{} & 0.396 & 0.400 \\
          &       & 1.0  & 1.0  & 0.046 & 0.046 & \multicolumn{1}{c}{} & 0.052 & 0.056 & \multicolumn{1}{c}{} & 0.778 & 0.776 \\
          &       &       & 1.5  & 0.044 & 0.044 & \multicolumn{1}{c}{} & 0.042 & 0.044 & \multicolumn{1}{c}{} & 0.838 & 0.838 \\
          &       & 1.5  & 1.5  & 0.026 & 0.026 & \multicolumn{1}{c}{} & 0.026 & 0.026 & \multicolumn{1}{c}{} & 0.978 & 0.978 \\ \hline
    \end{tabular}
  \begin{minipage}{140mm}
    \textit{Notes:} 
    The data generating process is the bivariate system
    as in (EQ1) and (EQ2) of Table 1
    and 
    standard normal errors $(u_{1t}, u_{2t})'$ 
    are either 
    uncorrelated 
    or 
    correlated with 
    $cov(u_{1t}, u_{2t}) = 0.5$. 
    The number of observations $T$ is set to 100. 
    Columns (1)-(4) report 
    empirical size at a 5\% nominal level
    and 
    Columns (5)-(6) show 
    empirical power 
    given 
    break dates $(k_{1}, k_{2}) = (35, 50)$
    and 
    critical values at a 5\% significance level.
    The AR coefficient $\alpha$ is set to 0.0, 0.4 and 0.8.
    We use
    $0.5$, $1.0$ and $1.5$ as magnitude of the break sizes.
  \end{minipage} 
\end{table}%

%%%%%%%%%%%%%%%%%%%%%%%%%%%%%%%%%%%%%%%%%%%%%%%%%%%%%%%%%%%%%%%%%%%%%%%%%%%%%%%%%%%%%%%%%%%%%%%
\newpage

\begin{table}
  \centering
  Table 3. Structural breaks in the U.S. disaggregated inflation series
  \begin{tabular}{lccc} \hline \hline
    \multicolumn{4}{c}{Replication of the results in Clark (2006)} \\\hline
    & \multicolumn{3}{c}{OLS without breaks} \\ \cline{2-4}
    & Durables & Nondurables & Service \\ \cline{2-4}
    Persistency & 0.921  & 0.878  & 0.855  \\ \hline
    & \multicolumn{3}{c}{OLS with common break} \\ \cline{2-4}
    & Durables & Nondurables & Service \\\cline{2-4}
    Persistency & 0.800  & 0.367  & 0.137  \\ \cline{2-4}
    Break Date (Known)& \multicolumn{3}{c}{93:Q1}  \\ \hline
    \multicolumn{4}{c}{Evidence from SUR system} \\\hline
%    & \multicolumn{3}{c}{With the common break assumption} \\\cline{2-4}
    & \multicolumn{3}{c}{SUR with common breaks ($k_{1} = k_{2} = k_{3}$)} \\ \cline{2-4}
%    & \multicolumn{3}{c}{($k_{1} = k_{2} = k_{2}$)} \\ \cline{2-4}
    & Durables & Nondurables & Service \\\cline{2-4} 
    Persistency & 0.805  & 0.356  & 0.166  \\\cline{2-4}
    Break Date & \multicolumn{3}{c}{92:Q1}  \\ \hline
    \multicolumn{4}{c}{Test for common break} \\ \hline
    Null Hypothesis   & LR test &  Critical value (5\%) & \\\cline{2-4}
    $H_{0}: k_{1} = k_{2} = k_{3}$  & 9.015   & 5.242  &  \\
    $H_{0}: k_{1} = k_{2}$         & 9.735   & 3.473 &  \\ 
    $H_{0}: k_{1} = k_{3}$         & 7.684   & 3.259 &  \\
    $H_{0}: k_{2} = k_{3}$         & 0.749   & 2.501 &  \\ \hline
%    & \multicolumn{3}{c}{Without the common break assumption} \\ \cline{2-4}
    & \multicolumn{3}{c}{SUR with common break ($k_{2} = k_{3}$)} \\ \cline{2-4}
%    & \multicolumn{3}{c}{($k_{2} = k_{2}$)} \\ \cline{2-4}
    & Durables & Nondurables & Service \\ \cline{2-4}
    Persistency  & 0.324  & 0.406  & 0.153  \\ \cline{2-4}
    Break Date   & 95:Q1  & \multicolumn{2}{c}{\hspace{.5cm} 92:Q1} \\
    95\% C.I.         & [94:Q2, 95:Q4]  
    & \multicolumn{2}{c}{\hspace{.5cm}  [91:Q3, 92:Q3]} \\\hline
\end{tabular}
\begin{minipage}{110mm}
  \textit{Notes:}
  The sample period is 1984 to 2002. The estimated model is the AR
  model with the intercept
  and 
  the AR lag length selected by the AIC is 4, 5 or 3 for 
  durables, nondurables or service, respectively. 
  \textit{Persistency} is measured by the sum of AR coefficients. 
  The critical values at the 5\% significance level are  
  obtained through 
  a computationally efficient algorithm
  with 3,000 repetitions.
  C.I. denotes the 95\% confidence interval of the break date.
\end{minipage}

\end{table}

\clearpage
\begin{figure}[htbp]
  \label{fig1}
  \centering 
  Figure 1: Finite-sample power of the test 
  \vspace{0.5cm}

    \caption*{Panel A: AR Coefficient = 0.00}  
    \begin{subfigure}[b]{0.31\textwidth}
        \caption{Break Size in EQ1: 0.5}
        \includegraphics[width=\textwidth]{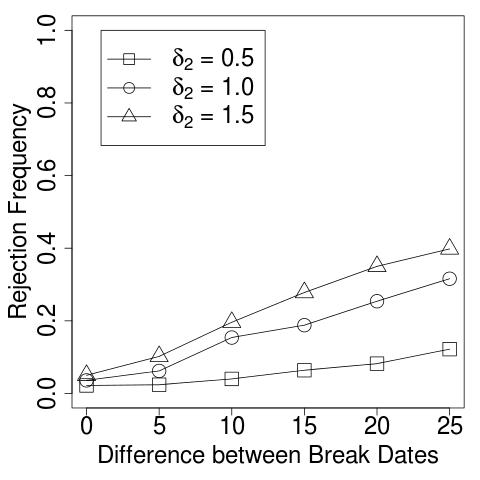}
    \end{subfigure}
    \begin{subfigure}[b]{0.31\textwidth}
        \caption{Break Size in EQ1: 1.0}
        \includegraphics[width=\textwidth]{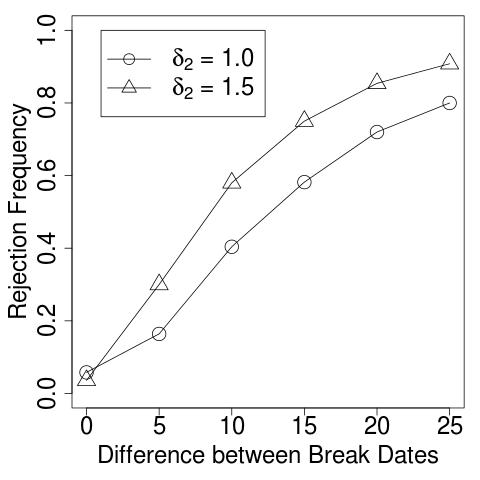}
    \end{subfigure}
    \begin{subfigure}[b]{0.31\textwidth}
        \caption{Break Size in EQ1: 1.5}
        \includegraphics[width=\textwidth]{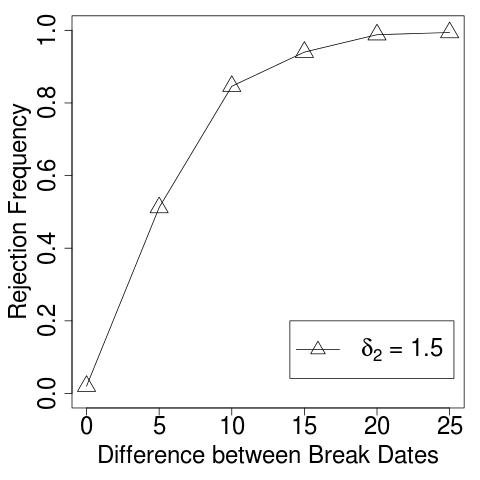}
    \end{subfigure}

    \caption*{Panel B: AR Coefficient = 0.40}  
    \begin{subfigure}[b]{0.31\textwidth}
        \caption{Break Size in EQ1: 0.5}
        \includegraphics[width=\textwidth]{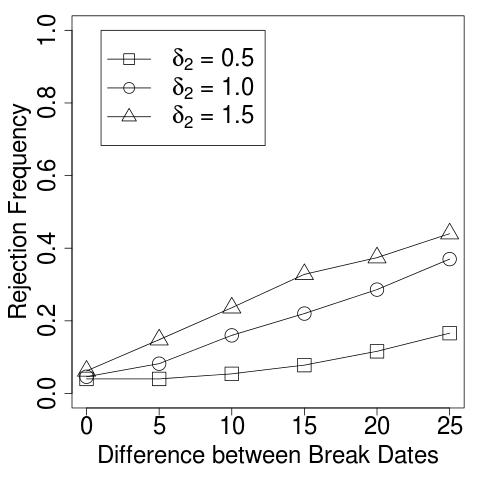}
    \end{subfigure}
    \begin{subfigure}[b]{0.31\textwidth}
        \caption{Break Size in EQ1: 1.0}
        \includegraphics[width=\textwidth]{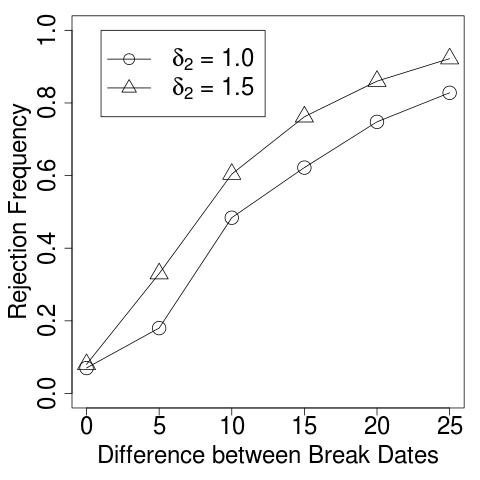}
    \end{subfigure}
    \begin{subfigure}[b]{0.31\textwidth}
        \caption{Break Size in EQ1: 1.5}
        \includegraphics[width=\textwidth]{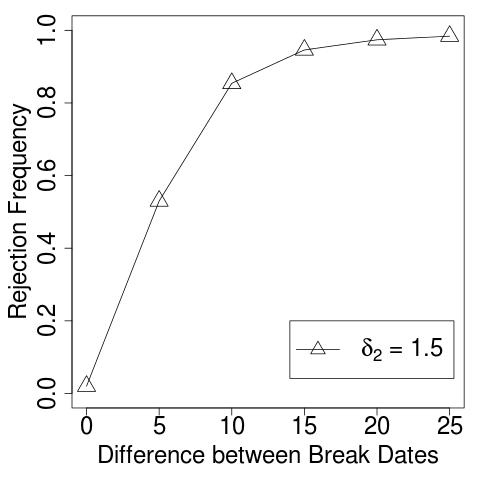}
    \end{subfigure}

    \caption*{Panel C: AR Coefficient = 0.80}  
    \begin{subfigure}[b]{0.31\textwidth}
        \caption{Break Size in EQ1: 0.5}
        \includegraphics[width=\textwidth]{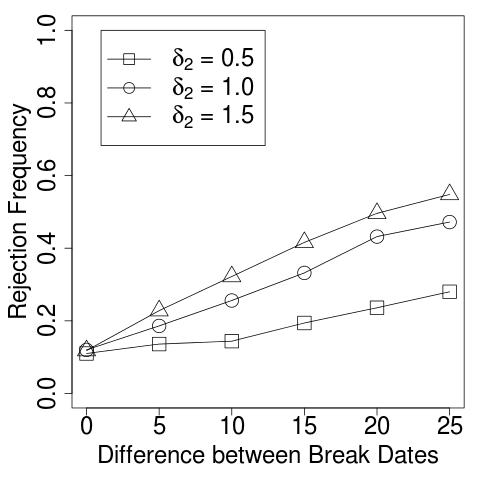}
    \end{subfigure}
    \begin{subfigure}[b]{0.31\textwidth}
        \caption{Break Size in EQ1: 1.0}
        \includegraphics[width=\textwidth]{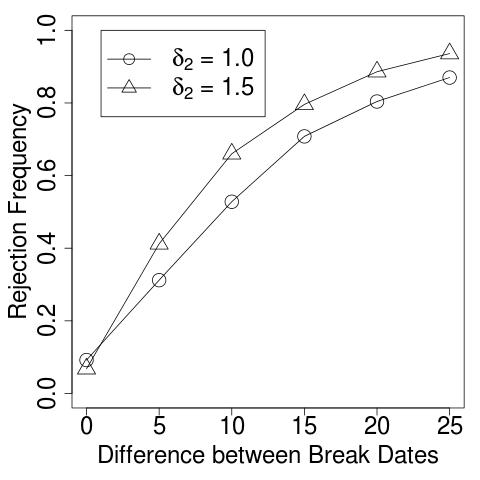}
    \end{subfigure}
    \begin{subfigure}[b]{0.31\textwidth}
        \caption{Break Size in EQ1: 1.5}
        \includegraphics[width=\textwidth]{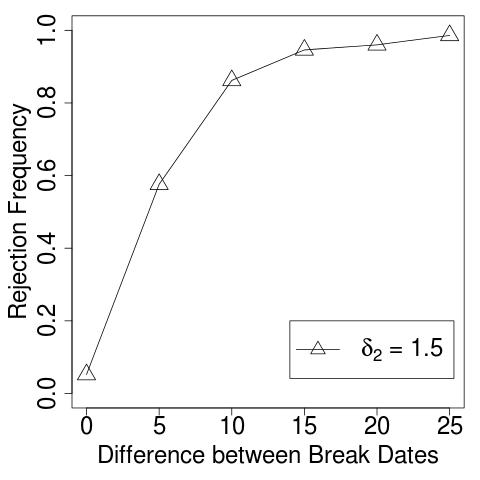}
    \end{subfigure}

    \vspace{0.3cm}
  \begin{minipage}{150mm}
    \textit{Notes:} 
    The data generating process is the bivariate system
    as in (EQ1) and (EQ2) of Table 1.
    The number of observations $T$ is set to 100. 
    The break date $k_{1}$ in (EQ1) is kept fixed at $k_1 = 35$, while
    the break date $k_{2}$ in (EQ2) changes from 30 to 55. 
    The horizontal axis shows the difference between break dates: $k_2
    - k_1$. 
    The AR coefficient $\alpha$ is set to 0.0, 0.4 and 0.8 
    for Panel A, B and C, respectively.
    The break size $\delta_1$ in (EQ1) changes 
    across panel (a)-(c), (d)-(f) and (g)-(i),
    while 
    the break size $\delta_2$ in (EQ2) changes within each panel.
    We use
    $0.5$, $1.0$ and $1.5$ as magnitude of the break size.
  \end{minipage}

\end{figure}

%%%%%%%%%%%%%%%%%%%%%%%%%%%%%%%%%%%%%%%%%%%%%%%%%%%%%%%%%%%%%%%%%%%%%%%%%%%%%%%%%%%%%%%%%%%%
%%%%%%%%%%%%%%%%%%%%%%%%%%%%%%%%%%%%%%%%%%%%%%%%%%%%%%%%%%%%%%%%%%%%%%%%%%%%%%%%%%%%%%%%%%%%
%%%%%%%%%%%%%%%%%%%%%%%%%%%%%%%%%%%%%%%%%%%%%%%%%%%%%%%%%%%%%%%%%%%%%%%%%%%%%%%%%%%%%%%%%%%%
\end{document}